%% file: main_preprint.tex
\documentclass{article}
\usepackage{graphicx} 

\usepackage[T1]{fontenc} 
\usepackage{geometry}
\usepackage{natbib}
\usepackage{authblk}
\usepackage{booktabs}
\usepackage{enumitem}

\usepackage{setspace}

\usepackage{color}

\usepackage{hyperref}
\hypersetup{
    colorlinks=true,
    linkcolor=blue,
    citecolor=magenta,      
    urlcolor=cyan,
    pdfpagemode=FullScreen,
}

\input{commands}

\title{Neumann-series corrections for regression adjustment\\in randomized experiments}
\author{Dogyoon Song\thanks{Email: \texttt{dgsong@ucdavis.edu}}}
\affil{Department of Statistics, University of California, Davis}
\date{\today}

\begin{document}

\maketitle

\begin{abstract}
    We study average treatment effect (ATE) estimation under complete randomization with many covariates in a design-based, finite-population framework. 
    In randomized experiments, regression adjustment can improve precision of estimators using covariates, without requiring a correctly specified outcome model. 
    However, existing design-based analyses establish asymptotic normality only up to $p = o(n^{1/2})$, extendable to $p = o(n^{2/3})$ with a single de-biasing. 
    We introduce a novel theoretical perspective on the asymptotic properties of regression adjustment through a Neumann‑series decomposition, yielding a systematic higher-degree corrections and a refined analysis of regression adjustment. 
    Specifically, for ordinary least squares regression adjustment, the Neumann expansion sharpens analysis of the remainder term, relative to the residual difference-in-means. 
    Under mild leverage regularity, we show that the degree-$d$ Neumann-corrected estimator is asymptotically normal whenever $p^{ d+3}(\log p)^{ d+1}=o(n^{ d+2})$, strictly enlarging the admissible growth of $p$.
    The analysis is purely randomization-based and does not impose any parametric outcome models or super-population assumptions.
\end{abstract}

\newpage
\tableofcontents
\newpage

\input{contents/01_introduction}
\input{contents/02_setup}

\input{contents/03_method}
\input{contents/04_neumann_expectation}
\input{contents/05_theory_OLS}

\input{contents/06_experiments}
\input{contents/07_conclusion}

\bibliographystyle{plainnat}
\bibliography{bibliography}
\newpage

\appendix

\begin{center}
\LARGE\bfseries
    \textsc{
    Supplementary Material
    }
\end{center}
\vspace{10pt}

\begin{itemize}[leftmargin=1.25em]
    \item 
    \textbf{Appendix~\ref{sec:assumption_rate}} (Non‑asymptotic envelopes) states concentration tools for sampling without replacement and derive high‑probability envelopes that instantiate Assumption~\ref{assumption:rates}.
    \item 
    \textbf{Appendix~\ref{sec:additional_examples}} (Additional examples) illustrates the Neumann weight (Definition~\ref{defn:neumann_weight}) and the combinatorial notions (words, partitions, design monomials) used in Section~\ref{sec:design_expectation}.
    \item 
    \textbf{Appendix~\ref{sec:computation_expectation}} (Computation of design expectations) describes algorithms to compute the Neumann weights and the correction terms efficiently.
    \item 
    \textbf{Appendix~\ref{sec:proof_main_theorem}} (Proof of the main theorem) provides a proof of Theorem~\ref{thm:main-OLS}, including the swap‑sensitivity and Poincaré–Johnson arguments.
    \item 
    \textbf{Appendix~\ref{sec:additional_experiments}} (Additional experiments) reports complementary empirical results and implementation details that augment Section~\ref{sec:experiments}.
\end{itemize}

\input{contents/Appx01_assumptions}

\input{contents/Appx02_additional_examples}
\input{contents/Appx03_algorithm}

\input{contents/Appx04_proof_main_thm1}

\input{contents/Appx05_additional_experiments}

\end{document}

%% file: commands.tex
\usepackage{amsmath, amsfonts, amssymb, amsthm}
\usepackage{mathtools}
\usepackage{bm}
\usepackage{algorithm, algpseudocode}
\usepackage{chngcntr}

\mathtoolsset{showonlyrefs}

\newcommand{\RR}{\mathbb{R}}
\newcommand{\ZZ}{\mathbb{Z}}
\newcommand{\ZZp}{\ZZ_{>0}}
\newcommand{\ZZnn}{\ZZ_{\geq 0}}
\newcommand{\bbE}{\mathbb{E}}

\newcommand{\Ind}{\mathbb{I}}

\newcommand{\bone}{\boldsymbol{1}}

\newcommand{\cE}{\mathcal{E}}

\newcommand{\cG}{\mathcal{G}}

\newcommand{\cN}{\mathcal{N}}

\newcommand{\cQ}{\mathcal{Q}}

\newcommand{\cS}{\mathcal{S}}
\newcommand{\cT}{\mathcal{T}}

\newcommand{\cV}{\mathcal{V}}

\newcommand{\res}{r}

\newcommand{\sfL}{\mathsf{L}}
\newcommand{\sfR}{\mathsf{R}}

\newcommand{\whsfR}{\widehat{\sfR}}
\newcommand{\sfT}{\mathsf{T}}

\newcommand{\tr}{\mathrm{tr}}

\newcommand{\diag}{\mathrm{diag}}

\newcommand{\avg}{\mathrm{avg}}

\newcommand{\Var}{\mathrm{Var}}

\newcommand{\htaudim}{\hat{\tau}_{\mathrm{DiM}}}

\newcommand{\hresdim}{\hat{\tau}^{\mathrm{res}}_{\mathrm{DiM}}}
\newcommand{\htauols}{\hat{\tau}_{\mathrm{OLS}}}

\newcommand{\opnorm}[1]{\left\lVert #1 \right\rVert_{\mathrm{op}}}

\newcommand{\arm}{\alpha}
\newcommand{\Sone}{\cS_1}
\newcommand{\Szero}{\cS_0}
\newcommand{\Sa}{\cS_{\arm}}

\newcommand{\na}{n_{\arm}}

\newcommand{\Xa}{X_{\arm}}
\newcommand{\ya}{y_{\arm}}

\newcommand{\Ca}{C_{\arm}}

\newcommand{\xbara}{\bar{x}_{\arm}}

\newcommand{\ybara}{\bar{y}_{\arm}}
\newcommand{\onea}{\boldsymbol{1}_{\na}}
\newcommand{\Part}{\Pi}
\newcommand{\all}{\mathrm{all}}
\newcommand{\disc}{\mathsf{disc}}
\DeclareMathOperator{\Mob}{Mob}
\newcommand{\uni}{\mathrm{uni}}
\newcommand{\Unif}{\mathrm{Unif}}
\newcommand{\degree}{\mathsf{deg}}
\newcommand{\DeltaJK}{\Delta^{(J\leftrightarrow K)}}

\newcommand{\CompFold}[1]{\breve{\Phi}_{#1}}
\newcommand{\init}{\mathrm{init}}

\theoremstyle{plain}
\newtheorem{theorem}{Theorem}
\newtheorem{proposition}[theorem]{Proposition}
\newtheorem{lemma}{Lemma}[section]
\newtheorem{corollary}[theorem]{Corollary}
\newtheorem{definition}{Definition}[section]
\newtheorem{example}{Example}
\newtheorem{assumption}{Assumption}
\theoremstyle{remark}
\newtheorem{remark}{Remark}[section]

%% file: contents/01_introduction.tex
\section{Introduction}

Estimating causal effects is a fundamental problem, and randomized experiments remain among the most powerful tools for causal inference across scientific domains; see, e.g., \citet{fisher1935design, neyman, kempthorne1952design} for foundational developments and \citet{imbens2015causal, rosenberger2015randomization, ding2024first} for modern treatments.
In a treatment–control experiment, each unit $i\in [n]$ receives treatment or control per the design, and the average treatment effect (ATE) $\tau$ is often the primary estimand.
Statistical inference for randomized experiments typically proceeds from one of two perspectives.
Following \citet{neyman}, the finite-population view treats the two potential outcomes $\{ (y^{(1)}_i, y^{(0)}_i)\}_{i=1}^n$ as fixed, with randomness arising solely from treatment assignment \citep{aronow2014sharp, dasgupta2015causal,fogarty2018regression,mukerjee2018using, li2020rerandomization}.
By contrast, the superpopulation view assumes units are i.i.d. draws from a hypothetical infinite population \citep{tsiatis2008covariate, cattaneo2018inference}.
While both are widely used, theoretical developments under the finite-population framework—especially in modern, complex settings—remain less explored.
We adopt the design-based, finite-population perspective to foreground the role of randomization itself, without positing outcome‑generating models.

A baseline estimator of the ATE is the difference-in-means (DiM) estimator $\hat\tau_{\mathrm{DiM}}$, which is unbiased with an explicit variance formula \citep{neyman}.
Experiments often collect pre-treatment covariates; when these are predictive of outcomes, incorporating them can improve precision.
Let $x_i\in\RR^p$ be covariates.
Early work on analysis of covariance often assumed a constant treatment effect, under which a common approach is OLS of $Y_i$ on $(T_i,x_i)$ with a treatment indicator $T_i$ \citep{fisher1935design,kempthorne1952design,hinkelmann2007design}.
\citet{freedman2008regression} cautioned that (i) regression adjustment can be less efficient than DiM under heterogeneity; (ii) usual OLS standard errors can be inconsistent; and (iii) the OLS estimator has $O_p(n^{-1})$ bias.
These issues are resolved in fixed dimensions: \citet{lin_ols2013} recommends including all treatment–covariate interactions (equivalently, two arm‑wise OLS fits) and shows the resulting estimator is consistent, asymptotically normal, and never less efficient than DiM; \citet{chang2024exact} provide exact bias corrections with and without interaction.

Modern experiments often involve many covariates; with large $p$, fixed‑$p$ asymptotics may fail to capture the behavior of OLS adjustment. 
Recently, \citet{lei2021regression} extend the analysis of \citet{lin_ols2013}, establishing asymptotic normality when $p=o(n^{1/2})$, and then up to $p=o(n^{2/3})$ after a single degree‑0 correction (both ignoring logarithms, under balanced designs).
\citet{chiang2023regression} show that asymptotic normality can hold up to $p=o(n^{3/4})$ via cross‑fitting. 
Refinements based on lasso adjustments or related high-dimensional methods deliver guarantees under sparsity, semiparametric structure, or superpopulation modeling \citep{bloniarz2016lasso, wager2016high, chernozhukov2018double}. 
Nevertheless, to the best of our knowledge, a \emph{design-based}, finite-population theory for regression adjustment in high-dimensional setting, or even in the intermediate growing-$p$ regime with $p = o_n(n)$ that avoids ad hoc assumptions or techniques (e.g., sparsity, cross-fitting) remains limited.

\subsection{Contributions}
We study regression adjustment under the design-based, finite‑population perspective when $p<n$ but $p$ grows with $n$. 
We focus on OLS‑based regression adjustment (OLS-RA), which is widely used in practice for its simplicity and the lack of any ad hoc tuning parameters, to investigate fundamental properties of regression adjustment. 
Our contributions in this paper are three‑fold.

\begin{enumerate}
    \item 
    \emph{Neumann-series correction} (Section \ref{sec:methods}).  
    We introduce an \emph{arm‑wise} Neumann–series expansion of the centered Gram inverse that appears in the remainder term of the OLS-adjusted estimator (relative to the DiM of population residuals); see \eqref{eq:master-decomp} and \eqref{eqn:neumann_remainder}. 
    For a user‑chosen $d \geq 0$, we replace the random inverse by its Neumann series truncated at degree $d$. 
    We show each degree-$d'$ term has design expectation equal to the inner product between the \emph{Neumann weights} (Definition \ref{defn:neumann_weight}) and the population OLS residuals (Proposition \ref{prop:neumann_exp}). 
    By sample-analog estimation, this yields a drop‑in linear‑algebraic modification to the vanilla OLS-RA estimator that adds correction terms $\whsfR_\arm^{[d']}$ ($d' = 0, \dots, d$), computable from arm‑wise OLS fits; see \eqref{eqn:Neumann_k_sample}. 

    \item 
    \emph{Computation of Neumann weights} (Section \ref{sec:design_expectation}). 
    We derive a closed-form expression for the Neumann weight $\xi^{[d']}$, cf. \eqref{eqn:xi_weight}, by evaluating the \emph{design expectation} of the degree‑$d'$ Neumann term $\sfR^{[d']}_{\arm}$ under complete randomization. 
    To this end, we develop a finite‑population counting argument that yields a closed‑form for $\bbE[\sfR^{[d']}_{\arm}]$ (Theorem \ref{thm:xi_formula}). 
    We also describe algorithms to compute the Neumann weights (Appendix \ref{sec:computation_expectation}) used in our experiments. 
    
    \item 
    \emph{Finite-population asymptotics under growing $p$} (Section \ref{sec:theory-OLS}). 
    We prove consistency and the rates of convergence for the Neumann‑corrected estimator (Theorem \ref{thm:main-OLS}). 
    Further, we establish asymptotic normality for any fixed degree $d\ge 0$ under the scaling (ignoring logarithms)
    \[
        p^{d+3}(\log p)^{d+1}=o\!\left(n^{d+2}\right)\quad\Rightarrow\quad p=o\!\left(n^{(d+2)/(d+3)}\right),
    \]
    strictly extending \citet{lei2021regression} (Corollary \ref{cor:diffuse}).
    Intuitively, each additional degree cancels one more layer of the random‑inverse fluctuation, enlarging the admissible $p$‑regime.

    \item 
    \emph{Empirical validation} (Section \ref{sec:experiments}). 
    Through controlled finite‑population simulations that vary dimensionality ($p=\lfloor n^\gamma\rfloor$, $\gamma\in[0,0.7]$), covariate distributions (Gaussian and heavy‑tailed), and an adversarial residual configuration aligned with leverage, we show that small Neumann degrees ($d \in \{1, 2, 3\}$) markedly reduce design bias with modest increase in variance relative to plain OLS‑RA and DiM.
\end{enumerate}

\subsection{Related work}

We situate our results within the design-based literature on regression adjustment for randomized experiments. 
In fixed dimensions, \citet{freedman2008regression} highlighted pitfalls of regression adjustment (efficiency loss under heterogeneity, inconsistent OLS SEs, $O_p(n^{-1})$ bias), while \citet{lin_ols2013} advocated fitting all treatment--covariate interactions (equivalently, arm-wise OLS with an intercept) to achieve asymptotic normality and non-inferiority to the difference-in-means (DiM) baseline. 
Recent refinements provide exact bias corrections and sharpen inference under complete randomization \citep[e.g.,][]{chang2024exact, fogarty2018regression}.

For growing covariate dimension, \citet{lei2021regression} established design-based asymptotic normality of OLS-RA up to $p=o(n^{1/2})$ (ignoring logs) and, with a single degree-$0$ de-biasing, up to $p=o(n^{2/3})$ in balanced designs. 
Cross‑fitting can push the regime further \citep[e.g.,][]{chiang2023regression}, though such techniques often blend design‑based reasoning with model-assisted tools. 
Our contribution is complementary: we retain the plain OLS-RA estimator within the finite-population framework and provide a systematic, degree-$d$ Neumann correction that progressively controls the random inverse, enlarging the admissible $p$-growth regime.

A parallel line of work develops high-dimensional or model-assisted adjustments using modern learning methods. 
Lasso- or ML-based procedures can improve precision under sparsity or super-population assumptions \citep{bloniarz2016lasso, wager2016high, chernozhukov2018double}, often aided by sample splitting/cross-fitting. 
These methods modify the estimator and/or posit outcome models, whereas our analysis keeps the estimator as plain OLS-RA and relies solely on the randomization law.

A complementary line of work studies the geometry and statistical behavior of the OLS interpolator in modern, potentially over‑parameterized regimes; see, e.g., \citet{shen2023algebraic}. 
Our results are purely design‑based and focus on $p<n$ with growing $p$, while the algebraic perspective in that line of work is closely aligned with our Neumann‑series viewpoint and suggests a path to $p\geq n$ via stabilized inverses (e.g., ridge or pseudoinverses). 
In this sense, the two strands are synergistic: our corrections control randomization‑induced Gram fluctuations for OLS‑RA, while the OLS‑interpolator identities can inform extensions of Neumann corrections beyond the classical regime.

Technically, our arguments draw on concentration and limit theorems for sampling without replacement and spectral tools for randomization-induced perturbations. 
We use scalar/matrix Bernstein-type bounds adapted to SRSWOR \citep{hoeffding1963probability, gross2010note, bardenet2015concentration} and invoke the finite-population CLT \citep{hajek1960limiting}. 
A Poincar\'e (spectral-gap) inequality on the Johnson graph controls the variance of correction-estimation errors under single row-swaps \citep{levin2017markov, brouwer2011spectra}, enabling $o_P(n^{-1/2})$ remainder control for any fixed degree.

%% file: contents/02_setup.tex
\section{Setup and background}\label{sec:setup}

\paragraph{Notation.} 
Let $\RR$ denote the set of real numbers, and $\ZZ$ denote the set of integers. 
Let $\ZZnn \coloneqq \{ z \in \ZZ: z \geq 0\}$ and $\ZZ_{> 0} \coloneqq \{ z \in \ZZ: z > 0\}$.
For $n \in \ZZp$, write $[n] \coloneqq \{1, \dots, n\}$. 
A vector $v \in \RR^n$ is identified with its column vector form $v \in \RR^{n \times 1}$. 
For a matrix $M$, $M^{\top}$ denotes transpose, and $M^{-1}$ the inverse (when it exists). 
For readability, we use calligraphic symbols to denote sets (e.g., $\cS$). 
We write $\bone_n \coloneqq (1, \cdots, 1) \in \RR^n$ and $\Ind_E$ to denote the indicator of event $E$.

\subsection{Average treatment effect estimation}
\paragraph{Finite population and potential outcomes.}
We consider a finite population of $n$ units. 
Each unit $i\in[n]$ has two potential outcomes $y_i^{(\arm)}$ for $\arm\in\{0,1\}$, where $y_i^{(1)}$ is the treated outcome and $y_i^{(0)}$ is the control outcome. 
Collect $y^{(\arm)}=(y_1^{(\arm)},\dots,y_n^{(\arm)})\in\RR^n$.

\paragraph{Randomized experiments.}
A completely randomized design selects a treatment set $\Sone \subset [n]$ of a fixed size $n_1$ uniformly at random; the control set is $\Szero \coloneqq [n] \setminus \Sone$, with $n_0 = |\Szero| = n - n_1$. 
Let $T_i \coloneqq \Ind\{ i \in \Sone \} \in \{0,1\}$ denote the treatment indicator. 
Then the observed outcome is
\[
    Y_i = T_i y_i^{(1)} + (1 - T_i) y_i^{(0)}. 
\]
Randomness arises \emph{only} from the treatment assignment. 

For each $\arm \in \{0,1\}$, let $S_{\arm} \in \{0,1\}^{n_{\arm} \times n}$ be the selection matrix such that $(S_{\arm})_{ij} = 1$ if and only if $j$ is the $i$-th smallest element in $\Sa$. 
Define $\ya=S_{\arm} y^{(\arm)}$, the \emph{observed} subvector of $y^{(\arm)}$ corresponding to $\Sa$.

\paragraph{Average treatment effect (ATE) estimation.}
The finite-population average treatment effect is
\[
    \tau = \frac{1}{n} \sum_{i=1}^n (y^{(1)}_i - y^{(0)}_i).
\]
The difference-in-means (DiM) estimator of ATE is
\[
    \hat\tau_{\mathrm{DiM}} = \frac{1}{n_1}\sum_{i\in\Sone}y_i^{(1)} - \frac{1}{n_0}\sum_{i\in\Szero}y_i^{(0)}.
\]

\subsection{Regression adjustment}\label{sec:reg_adjustment}
\paragraph{Covariates, design and normalization.}
Each unit $i$ carries pre-treatment covariates $x_i\in\RR^p$. 
The design matrix refers to $X=\begin{bmatrix}x_1&\cdots&x_n\end{bmatrix}^\top\in\RR^{n\times p}$. 

\setcounter{assumption}{-1}
\begin{assumption}\label{assump:covariates}
    Suppose that
    \begin{enumerate}
        \item 
        $\bone_n^\top X=0$, i.e., columns of $X$ are centered, and
        \item 
        $X^\top X=nI_p$, i.e., columns of $X$ are linearly independent and normalized.
    \end{enumerate}
\end{assumption}

For $\arm \in \{0,1\}$, let $(\mu_{\arm},\beta_{\arm})$ denote the \emph{population} ordinary least squares (OLS) coefficients from regressing $y^{(\arm)}$ on $[\bone_n,X]$, and define the population residuals
\begin{equation}\label{eqn:population_OLS}
    \res^{(\arm)}  \coloneqq  y^{(\arm)} - \mu_{\arm} \bone_n - X\beta_{\arm},
        \qquad \arm\in\{0,1\}.
\end{equation}
By OLS orthogonality, $\bone_n^\top \res^{(\arm)}=0$ and $X^\top \res^{(\arm)}=0$. 
Computing the population OLS coefficients $(\mu_{\arm},\beta_{\arm})$ is only a notional device to aid in our analysis; it requires access to the potential outcome $y^{(\arm)}$ for the entire population, but we only observe $\ya = S_{\arm} y^{(\arm)}$.

\paragraph{Arm-wise OLS with intercept.} 
For each $\arm \in \{0,1\}$, let $(\hat\mu_{\arm},\hat\beta_{\arm})$ be the arm-wise OLS coefficients obtained by fitting observed data:
\begin{equation*}
    (\hat\mu_{\arm},\hat\beta_{\arm}) \in \arg\min_{(\mu, \beta)} \left\{ \frac{1}{n_{\arm}} \sum_{i \in \Sa}\bigl( y^{(\arm)}_i - \mu - x_i^{\top} \beta  \bigr)^2 \right\}.
\end{equation*}
For each $\arm \in \{0,1\}$, define $\Xa=S_{\arm}X$ as the submatrix of $X$ that retains only the rows with indices in $\Sa$, and define arm-wise means 
\begin{equation}\label{eqn:armwise_means}
    \xbara\coloneqq \frac{1}{n_{\arm}} \Xa^\top\onea
    \qquad\text{and}\qquad
    \ybara\coloneqq \frac{1}{n_{\arm}} \onea^\top \ya.
\end{equation}
Define the arm-wise (row-) centering projection matrix $\Ca \coloneqq I_{n_{\arm}} - \frac{1}{n_{\arm}}\onea\onea^\top$ so $\Ca\onea = 0$.
Then $\Ca X_{\arm} = X_{\arm} - \onea \xbara^{\top}$, and the arm-wise sample covariance 
\begin{equation}\label{eqn:armwise_covariance}
    \Sigma_{\arm} 
        \coloneqq  \frac{1}{n_{\arm}}  \sum_{i \in \Sa} ( x_i - \xbara )( x_i - \xbara )^{\top}
        = \frac{1}{n_{\arm}} \Xa^\top \Ca \Xa .
\end{equation}
By standard algebra for OLS-with-intercept,
\begin{equation*}
    \hat\beta_{\arm} = \Sigma_{\arm}^{-1} \Bigl(\frac{1}{n_{\arm}}\Xa^\top \Ca \ya\Bigr),
    \qquad
    \hat\mu_{\arm} = \ybara - \xbara^\top \hat\beta_{\arm}.
\end{equation*}
Substituting $\ya=\mu_{\arm}\onea+\Xa\beta_{\arm}+S_{\arm} \res^{(\arm)}$ and using $\Ca\onea=0$ and $X_{\arm}^\top\Ca\onea=0$ yields
\begin{equation} \label{eq:beta-sample-minus-pop}
    \hat\beta_{\arm} = \beta_{\arm} + \Sigma_{\arm}^{-1} \Bigl(\frac{1}{n_{\arm}}\Xa^\top \Ca S_{\arm} \res^{(\arm)}\Bigr). 
\end{equation}
Finally,
\begin{equation}
    \hat\mu_{\arm} - \mu_{\arm} = \underbrace{\frac{1}{n_{\arm}}\onea^\top S_{\arm} \res^{(\arm)}}_{\text{residual mean in arm $a$}}
    - 
    \xbara^\top\bigl(\hat\beta_{\arm}-\beta_{\arm}\bigr).
    \label{eq:mu-deviation}
\end{equation}

\paragraph{Regression adjustment.}
Recall the population residuals in \eqref{eqn:population_OLS}. 
Since $\bone_n^\top X=0$ and $\bone_n^\top \res^{(\arm)}=0$, 
\[
    \tau
        =\frac{1}{n}\bone_n^{\top} y^{(1)} - \frac{1}{n}\bone_n^{\top} y^{(0)} 
        = \mu_1-\mu_0.
\]
The regression-adjusted ATE estimator via OLS (or, the OLS-adjusted ATE estimator) refers to a natural plug-in estimator
\begin{equation}\label{eqn:OLS_RA}
    \hat\tau_{\mathrm{OLS}}\coloneqq \hat\mu_1-\hat\mu_0.
\end{equation}

\subsection{Asymptotics of regression adjustment}
For fixed $p$, OLS adjustment with treatment-covariate interactions (equivalently, arm-wise OLS with intercept as in \eqref{eqn:OLS_RA}) is a classical method that is asymptotically non-inferior to DiM under standard regularity. 
In particular, \citet[Theorem 1]{lin_ols2013} shows
\[
    \sqrt{n}( \hat\tau_{\mathrm{OLS}}-\tau ) \ \xRightarrow{d}\ \cN\bigl(0, \sigma^2_{\mathrm{res}}\bigr),
    \qquad\text{where}\qquad
    \sigma^2_{\mathrm{res}} \coloneqq \frac{v_1}{\rho}+\frac{v_0}{1-\rho}-v_\tau,
\]
with
$v_{\arm} \!\coloneqq \lim_{n \to \infty} \!\frac{1}{n}\sum_{i=1}^n\bigl(r^{(\arm)}_i\bigr)^2$ for each $\arm \in \{0,1\}$, and $v_\tau\! \coloneqq \lim_{n \to \infty} \!\frac{1}{n}\sum_{i=1}^n\bigl(r^{(1)}_i \!-\! r^{(0)}_i\bigr)^2$. 
For $p$ growing with $n$, \citet{lei2021regression} establish asymptotic normality (AN) for OLS-RA up to $p=o(n^{1/2})$ (modulo logs), and extend AN to $p=o(n^{2/3})$ with de-biasing.

At the core of these analyses is the decomposition obtained by \eqref{eq:beta-sample-minus-pop}--\eqref{eq:mu-deviation} and $\tau=\mu_1-\mu_0$:
\begin{equation}\label{eq:master-decomp}
    \hat\tau_{\mathrm{OLS}} - \tau
        = \underbrace{ \frac{1}{n_1} \bone_{n_1}^\top S_1 r^{(1)} - \frac{1}{n_0} \bone_{n_0}^\top S_0 r^{(0)}}_{\eqqcolon \hresdim;~~\text{Residual DiM}}
        - \underbrace{\Bigl(\sfR_1-\sfR_0\Bigr)}_{\text{Remainder}}
\end{equation}
where the arm-wise remainder for each $\arm \in \{0,1\}$ is
\begin{equation}\label{eqn:neumann_remainder}
    \sfR_{\arm} 
        \coloneqq \bar{x}_{\arm}^{\top} \bigl( \hat{\beta}_{\arm} - \beta_{\arm} \bigr) 
        = \bar x_{\arm}^\top \Sigma_{\arm}^{-1} \Bigl(\frac{1}{n_{\arm}}\Xa^\top \Ca S_{\arm} \res^{(\arm)}\Bigr).
\end{equation}
A sampling-without-replacement central limit theorem (e.g., \citet{hajek1960limiting}) implies $\sqrt{n} \cdot \hresdim \ \xRightarrow{d}\ \cN\bigl(0, \sigma^2_{\mathrm{res}}\bigr)$; hence, if the remainder is $o_P(n^{-1/2})$, the AN of $\hat\tau_{\mathrm{OLS}}-\tau$ follows. 
Therefore, ensuring $\sfR_1 - \sfR_0 = o_P(n^{-1/2})$ provides a sufficient condition for asymptotic normality.

\paragraph{Gaps and motivation.}
Despite strong asymptotic results for regression adjustment, several gaps remain.

\begin{itemize}
    \item 
    \emph{Systematic control of the random inverse.}
    The bottleneck in the analysis of \eqref{eqn:neumann_remainder} is the random inverse $\Sigma_{\arm}^{-1}$. 
    Prior work either (i) keeps $p$ fixed, or (ii) handles $p$ only up to $o(n^{1/2})$, with considering a \emph{single} (degree‑$0$) correction pushing to $o(n^{2/3})$. 
    A general scheme that progressively controls higher powers of $\Delta_{\arm}\coloneqq I_p-\Sigma_{\arm}$ is missing.
    
    \item
    \emph{Growth of $p$ in a design‑based framework.}
    Refinements such as lasso adjustments, cross‑fitting/double machine learning (ML), or super‑population models achieve AN under sparsity or extra modeling assumptions, but they either modify the estimator or depart from the finite‑population viewpoint. 
    A design‑based theory for \emph{plain} OLS‑RA that scales across the $p=o(n)$ remains limited. 
    Addressing this matters because OLS underlies many modern statistics/ML pipelines (e.g., learned features ultimately feed a linear fit).
    
    \item
    \emph{Tunable correction and tradeoff.}
    Degree-0 de‑biasing is a one‑shot fix. 
    With many covariates, analysts may want a principled “dial” that trades computation for performance as $p$ grows, retaining design‑based validity. 
\end{itemize}


%% file: contents/03_method.tex
\section{Method: Neumann series correction}\label{sec:methods}

\subsection{Motivation and intuition}
We develop a Neumann‑series correction of degree $d$ that leaves the OLS‑RA estimator intact but augments it with computable arm‑wise terms.
Recall the bottleneck for analyzing \eqref{eq:master-decomp} is the \emph{random inverse} $\Sigma_{\arm}^{-1}$ inside $\sfR_{\arm}$, cf. \eqref{eqn:neumann_remainder}. 
To tackle this challenge, we represent $\sfR_{\arm}$ by expanding $(I_p-\Delta_{\arm})^{-1}$ via a Neumann series, truncate at degree $d$, and estimate the truncated series from observed data.

\paragraph{Neumann series expansion and truncation.}
For each $\arm \in \{0,1\}$, define shorthand notations
\begin{equation*}
    \Delta_{\arm}\coloneqq I_p - \Sigma_{\arm},
    \qquad
    u_{\arm} \coloneqq \frac{1}{n_{\arm}} X_{\arm}^\top \Ca S_{\arm} \res^{(\arm)}\in\RR^p
\end{equation*}
so that $\sfR_{\arm} = \xbara^\top (I_p - \Delta_{\arm})^{-1} u_{\arm}$.
On the event that $\opnorm{\Delta_{\arm}}<1$, 
\begin{align}
        \Sigma_{\arm}^{-1} 
        = (I_p - \Delta_{\arm})^{-1} 
        = \sum_{d'=0}^\infty \Delta_{\arm}^{d'}    \qquad
        &\implies\qquad
        \sfR_{\arm} 
            =
            \sum_{d' = 0}^{d}  \underbrace{ \xbara^\top  \, \Delta_{\arm}^{d'} \, u_{\arm} }_{\eqqcolon \sfR_{\arm}^{[d']}}
            + \sum_{d'= d+1}^{\infty} \xbara^\top \, \Delta_{\arm}^{d'} \, u_{\arm}.
                \label{eqn:R-approx}
\end{align}
Assumption~\ref{assumption:rates} (in Section \ref{sec:theory-OLS}) ensure $\|\Delta_{\arm}\|<1$ with high probability, so the Neumann series converges and the truncation error is controlled; see Appendix~\ref{sec:assumption_rate} for further discussion.

We estimate the truncated Neumann series $\sum_{d'=0}^{d} \sfR_{\arm}^{[d']}$ and add it to $\hat{\tau}_{\mathrm{OLS}}$ to counterbalance the remainder term in \eqref{eq:master-decomp} up to degree $d$, in hopes of de-biasing or variance reduction effects. 
In the rest of this section, we describe how to estimate the truncated series from available data. 

\subsection{Estimating Neumann terms from data}\label{sec:neumann_estimation}
Each Neumann term $\sfR_{\arm}^{[d']} = \xbara^\top  \, \Delta_{\arm}^{d'} \, u_{\arm}$ depends on the population residuals $\res^{(\arm)}$ through the term $u_{\arm}$, which requires the entire potential outcomes and is not computable from the available data. 
Thus, we compute its design expectation $\bbE\bigl[ \sfR_{\arm}^{[d]} \bigr]$ and use its sample-analog estimator for data-driven correction of $\hat{\tau}_{\mathrm{OLS}}$.

\paragraph{Design expectation.} 
For $\cS \subseteq [n]$ with size $m$, define
\[
    \bar{x}_{\cS} \coloneqq \frac{1}{m} \sum_{i \in \cS} x_i
    \qquad\text{and}\qquad
    \Sigma_{\cS} \coloneqq \frac{1}{m} \sum_{i \in \cS} ( x_i - \bar{x}_{\cS} )( x_i - \bar{x}_{\cS} )^{\top},
\]
in a compatible way with \eqref{eqn:armwise_means}, \eqref{eqn:armwise_covariance} so that $\bar{x}_{\arm} = \bar{x}_{\Sa}$ and $\Sigma_{\arm} = \Sigma_{\Sa}$.

\begin{definition}\label{defn:neumann_weight}
    Let $n, m \in \ZZp$ such that $1 \leq m \leq n$. 
    Let $\cS \subset [n]$ be a simple random sample without replacement (SRSWOR) of size $m$. 
    For $i \in [n]$ and $d \in \ZZnn$, the $i$-th Neumann weight (of size $m$ and degree $d$) associated with $X \in \RR^{n \times p}$ is defined as
    \begin{equation}\label{eqn:xi_weight}
    \begin{aligned}
        \xi^{[d]}_i(m; X) &= \bbE\left[ \, \bar{x}_{\cS}^{\top} \bigl( I_p - \Sigma_{\cS} \bigr)^d \bigl( x_i - \bar{x}_{\cS} \bigr) \,\Bigm|\, i \in \cS \, \right].
    \end{aligned}
    \end{equation}
\end{definition}

\begin{proposition}\label{prop:neumann_exp}
    For every $d \in \ZZnn$, 
    \begin{equation}\label{eqn:design_expectation}
        \bbE\bigl[ \sfR^{[d]}_{\arm} \bigr]
            = \frac{1}{n} \sum_{i=1}^n \xi^{[d]}_i(n_{\arm}; X) \cdot r^{(\arm)}_i,
            \qquad \forall \arm \in \{0,1\}.
    \end{equation}
\end{proposition}

\begin{proof}[Proof of Proposition \ref{prop:neumann_exp}]
    By definition, cf. \eqref{eqn:R-approx}, the degree-$d$ Neumann term
    \begin{align*}
        \sfR^{[d]}_{\arm} 
            &= \bar x_{\Sa}^\top \bigl( I_p - \Sigma_{\Sa} \bigr)^d \biggl( \frac{1}{n_{\arm}}\Xa^\top \Ca S_{\arm} \res^{(\arm)} \biggr)\\
            &= \bar x_{\Sa}^\top \bigl( I_p - \Sigma_{\Sa} \bigr)^d \biggl(\frac{1}{n_{\arm}} \sum_{i \in \Sa} (x_i - \bar{x}_{\Sa}) \res^{(\arm)}_i\biggr)\\
            &= \frac{1}{n_{\arm}} \sum_{i =1}^n \Ind\{i \in \Sa\}
                \cdot \underbrace{\bar x_{\Sa}^\top \bigl( I_p \!-\! \Sigma_{\Sa} \bigr)^d  (x_i \!-\! \bar{x}_{\Sa})}_{\eqqcolon \zeta_i^{[d]}(\Sa; X)}
                \cdot  \res^{(\arm)}_i.
    \end{align*}
    Taking expectation over the random set $\Sa$,
    \begin{align*}
        \bbE \bigl[ \sfR^{[d]}_{\arm} \bigr]
            &= \frac{1}{n_{\arm}} \sum_{i =1}^n 
                \bbE\Bigl[ \Ind\{i \in \Sa\} \cdot \zeta_i^{[d]}(\Sa; X) \Bigr] \cdot \res^{(\arm)}_i\\
            &= \frac{1}{n_{\arm}} \sum_{i =1}^n 
                \Pr\bigl( i \in \Sa \bigr) \cdot \bbE\Bigl[ \zeta_i^{[d]}(\Sa; X) \bigm| i \in \Sa \Bigr] \cdot \res^{(\arm)}_i\\
            &= \frac{1}{n} \sum_{i =1}^n \xi^{[d]}_i(n_{\arm}; X) \cdot r^{(\arm)}_i,
    \end{align*}
    because $\Pr\bigl( i \in \Sa \bigr) = \tfrac{|\Sa|}{n} = \tfrac{n_{\arm}}{n}$ under SRSWOR. 
\end{proof}

\begin{example}[Degree $d=0$: closed form]\label{example:exp_k.0}
    Assume the standard design normalization $1_n^\top X = 0$ and $X^\top X = n I_p$ (cf. Assumption \ref{assump:covariates}).
    For $d=0$, 
    \begin{equation}\label{eqn:Neumann_exp_k.0}
        \xi^{[0]}_i(m; X)
            = \frac{(m-1)(n-m)n }{ m^2(n-1)(n-2)} \, \Bigl( \|x_i\|_2^2 - p \Bigr).
    \end{equation}
    See Appendix \ref{sec:example_design_k.0} for the derivation of \eqref{eqn:Neumann_exp_k.0}.
\end{example}

For $d \geq 1$, higher-order interaction terms arise from $(I_p - \Sigma_{\cS})^d$; see Section \ref{sec:design_expectation} for a combinatorial argument to account for these.

\paragraph{Sample analogs.}
Note that the population residuals $\res^{(\arm)}_i$ in \eqref{eqn:design_expectation} are not available in practice. 
Thus, we use sample analog estimators of $\bbE\bigl[ \sfR^{[d]}_{\arm} \bigr]$ using in-sample OLS residuals. 
Specifically, for each $d \in \ZZnn$, 
\begin{equation}\label{eqn:Neumann_k_sample}
    \whsfR^{[d]}_{\arm} \coloneqq \frac{1}{n_{\arm}} \sum_{i \in \Sa}  \xi^{[d]}_i(n_{\arm}; X) \cdot \hat{\res}^{(\arm)}_i.
\end{equation}
Under the mild conditions that $p = o(n)$ and $\sup_{\arm \in \{0,1\}} \sup_n n^{-1} \bigl\| r^{(\arm)} \bigr\|_2^2 < \infty$, 
this sample-analog closely approximates the design expectation: $\whsfR_{\arm}^{[d]} - \bbE \bigl[ \sfR_{\arm}^{[d]} \bigr] = o_P(n^{-1/2})$; see Theorem \ref{thm:main-OLS} and Appendix \ref{sec:proof_main_theorem}.

\begin{example}[Degree $d=0$: sample analog estimator]\label{example:sample_k.0}
    For $d=0$, cf. \eqref{eqn:Neumann_exp_k.0} in Example \ref{example:exp_k.0}, for each $a \in \{0,1\}$,
    \begin{equation}\label{eqn:Neumann_sample_k.0}
        \begin{aligned}
            \whsfR^{[0]}_{\arm} &\coloneqq \frac{n_{\arm} - 1}{n_{\arm}}\frac{n^2}{(n-1)(n-2)}
                \cdot \frac{n_{1-\arm}}{n_{\arm}} \cdot \frac{1}{n_{\arm}} \sum_{i \in \Sa}  \biggl( \frac{\|x_i\|_2^2}{n} - \frac{p}{n} \biggr) \cdot \hat{\res}^{(\arm)}_i.
        \end{aligned}
    \end{equation}
\end{example}

We compare \eqref{eqn:Neumann_sample_k.0} with the debiased estimator proposed in \citet{lei2021regression}. 
They proposed 
\[
    \hat{\tau}^{\mathrm{de}}_{\mathrm{OLS}}
        = \hat{\tau}_{\mathrm{OLS}} - \biggl( \frac{n_1}{n_0} \hat{\Delta}_0 - \frac{n_0}{n_1} \hat{\Delta}_1 \biggr),
\]
where $\hat{\Delta}_{\arm} = n_{\arm}^{-1} \sum_{i \in \Sa} h_i \cdot \hat{r}^{(\arm)}_i$ where $h_i := \bigl[ X(X^{\top}X)^{-1} X^{\top} \bigr]_{ii}$ is the $i$-th leverage score. 
Noticing that $h_i = \tfrac{\|x_i\|_2^2}{n}$ (because $X^{\top}X = n I_p$ per Assumption \ref{assump:covariates}) and that $\sum_{i=1}^n h_i = p$, the degree-0 Neumann correction term in \eqref{eqn:Neumann_sample_k.0} is a scalar multiple of the Lei-Ding debiased estimator by the finite-sample correction factors $\frac{n_{\arm} - 1}{n_{\arm}}\frac{n^2}{(n-1)(n-2)}$. 
Since $\bone_n^\top r^{(\arm)}=0$ by the population normal equations with $\bone_n^\top X=0$, centering the leverage $h_i$ by $p/n$ does not change the sum.

\paragraph{OLS-adjusted ATE estimator with Neumann correction of degree up to $d$.} 
Finally, we define the Neumann-corrected estimator as follows. 

\begin{definition}
    For a nonnegative integer $k$, the \emph{OLS-adjusted ATE estimator with Neumann correction of degree up to $d$} is
    \begin{equation}\label{eq:NPm-def-final}
        \htauols^{[d]}
            \coloneqq
            \htauols   
            + \sum_{d'=0}^{d} \left( \whsfR_1^{[d']}-\whsfR_0^{[d']} \right).
    \end{equation}
    where the degree-$d'$ correction term $\whsfR_{\arm}^{[d']}$ is as in \eqref{eqn:Neumann_k_sample}.
\end{definition}

%% file: contents/04_neumann_expectation.tex
\section{Design expectation of Neumann terms}\label{sec:design_expectation}

Recall from \eqref{eqn:design_expectation} in Proposition \ref{prop:neumann_exp} that for every integer $d \geq 0$, 
$\bbE\bigl[ \sfR^{[d]}_{\arm} \bigr] = \frac{1}{n} \sum_{i=1}^n \xi^{[d]}_i(n_{\arm}; X) \cdot r^{(\arm)}_i$, $\forall \arm \in \{0,1\}$. 
Thus, computing the expectation reduces to computing $\xi^{[d]}_i(m;X)$ as defined in \eqref{eqn:xi_weight} with $m=n_{\arm}$.

Let
\[
    U_{\cS} \coloneqq  \frac{1}{m} \sum_{j \in \cS}  x_j x_j^{\top},
    \qquad
    V_{\cS} \coloneqq \bar{x}_{\cS} \bar{x}_{\cS}^{\top}.
\]
Then $\Sigma_{\cS} = U_{\cS} - V_{\cS}$ and the argument inside the conditional expectation is
\begin{equation}\label{eqn:neumann_template}
        \bar{x}_{\cS}^{\top} \bigl( I_p - U_{\cS} + V_{\cS} \bigr)^d \bigl( x_i - \bar{x}_{\cS} \bigr),
\end{equation}
which is a polynomial in $(x_1, \cdots, x_n)$ of degree $2d+2$. 
Its expectation depends only on which unit indices coincide across the factors, i.e., on the multi‑index pattern of each constituent monomial rather than on the specific values of those indices, conditioned on $i \in \cS$.

\paragraph{Combinatorial definitions.} 
We formalize the multi-index patterns to properly enumerate and stratify them, as well as identify their joint inclusion probability conditioned on the event $i \in \cS$.

\begin{definition}[Word]\label{defn:word}
    Let $d \in \ZZnn$. 
    A \emph{word} is a pair $\theta = (\phi, \chi)$ with
    \[
        \phi = (\phi_1, \cdots, \phi_d) \in \{ I, U, V \}^d,
        \qquad
        \chi \in \{i, \avg\}.
    \]
    Let $u(\phi) \coloneqq \bigl| \{ a: \phi_a = U \} \bigr|$ and $v(\phi) \coloneqq \bigl| \{ a: \phi_a = V \} \bigr|$.
    The \emph{sign} and \emph{multiplicity} of a word $\theta = (\phi,\chi)$ are
    \begin{align*}
        \mathrm{sgn}_{\theta} &\coloneqq (-1)^{u(\phi) + \Ind\{ \chi = \avg \}},\\
        L_{\theta} &\coloneqq 1 + u(\phi) + 2 v(\phi) + \Ind\{ \chi = \avg \}.
    \end{align*}
\end{definition}

A word encodes, for each factor of $\bigl(I_p - U_{\cS} + V_{\cS}\bigr)$, the choice among $I_p$, $-U_{\cS}$, and $V_{\cS}$, together with the terminal choice between $x_i$ (if $\chi=i$) and $-\bar{x}_{\cS}$ (if $\chi=\avg$), which comprise $3^d \times 2$ total possible configurations. 
The sign of word, $\mathrm{sgn}_{\theta}$, is equal to the sign of the corresponding monomial in the expansion of \eqref{eqn:neumann_template}, and the multiplicity $L_{\theta}$ equals the number of $\cS$-averages introduced by $\theta$ (each contributes a factor of $1/m$ in the pre-factor).

\begin{definition}[Position graph]\label{defn:position_graph}
    For a word $\theta = (\phi,\chi)$, its \emph{position graph} $\cG_{\theta} = (\cV_{\theta}, \cE_{\theta})$ is an undirected graph such that the vertex set $\cV_{\theta} = \{ 1, 2, \cdots, L_{\theta} \}$ and the edge set $\cE_{\theta} \subset \cV_{\theta} \times \cV_{\theta}$ is constructed as follows:
    \begin{enumerate}
        \item 
        Initialize $\cE_{\theta} \gets \emptyset$, $s \gets 1$.
        \item 
        For $a = 1, \cdots, d$:
        \begin{itemize}
            \item 
            If $\phi_a = I$, do nothing.
            \item 
            If $\phi_a = U$, $\cE_{\theta} \gets \cE_{\theta} \cup (s, s+1)$; set $s \gets s+1$.
            \item 
            If $\phi_a = V$, $\cE_{\theta} \gets \cE_{\theta} \cup (s, s+1)$; set $s \gets s+2$.
        \end{itemize}
        \item 
        If $\chi = \avg$, then $\cE_{\theta} \gets \cE_{\theta} \cup (s, s+1)$.
    \end{enumerate}
\end{definition}

Vertices index the appearances of $\bar x_{\cS}$ (hence, each vertex contributes a factor $1/m$ due to averaging); each edge $(s,t)$ records that the monomial contains an inner product between the vectors assigned to positions $s$ and $t$. 
Intuitively, choosing $U_{\cS}$ adds one new average (one vertex) and one edge; choosing $V_{\cS}$ adds two new averages (two vertices) and one edge between them; choosing $I_p$ adds neither. 
If $\chi=\avg$, the terminal $-\bar x_{\cS}$ contributes one more average and one more edge.

\begin{definition}[Partition of a set]\label{defn:partition}
    A \emph{partition} of a set $\cT \subset \ZZ$ is a collection of non-empty, pairwise disjoint subsets of $\cT$, called \emph{blocks} whose union is $\cT$. 
    We fix an enumeration $\pi = \{ P_1, \dots, P_{\ell} \}$ of its blocks by the increasing order of minima, 
    \[
        \min P_1 < \min P_2 < \cdots < \min P_{\ell},
    \]
    and write $|\pi| = \ell$. 
    The \emph{partition index} is a map $\iota_{\pi}: \cT \to [|\pi|]$ defined by 
    \[
        \iota_{\pi}(t) = a 
            \qquad\text{if and only if}\qquad t \in P_{a}.
    \]
    Let $\Part(\cT)$ denote the set of all partitions of $\cT$. 
\end{definition}

\begin{definition}[Allocation and design monomial]\label{defn:allocation}
    For a word $\theta$ and a partition $\pi = \{ P_1, \dots, P_{\ell}\} \in \Part(\cV_{\theta})$, 
    an \emph{allocation} admissible to $(\theta, \pi)$ is an injective map $\omega: [\ell] \to [n]$; the set of all allocations is
    \[
        \Omega(\theta, \pi) \coloneqq \left\{ \omega \in [n]^{\ell}: \omega_a \neq \omega_b, ~\forall (a,b) \text{ with }a \neq b \right\}.
    \]
    For $X \in \RR^{n \times p}$, its $i$-th \emph{design monomial} with respect to $\omega \in \Omega(\theta, \pi)$ is    
    \begin{equation*}
        \varphi_{\omega}(X; i) 
            \coloneqq \biggl(\! \prod_{(s,t) \in \cE_{\theta}} x_{ \omega_{\iota_{\pi}(s)}}^{\top} x_{ \omega_{\iota_{\pi}(t)}} \!\biggr) 
            \times \bigl( x_{ \omega_{\iota_{\pi}(L_{\theta})}}^{\top} x_i \bigr)^{\Ind\{ \chi = i \}}.
    \end{equation*}
\end{definition}

Note $\varphi_{\omega}(X; i)$ is the monomial we get when 
(i) the element in each parenthesis in \eqref{eqn:neumann_template} is chosen per the word $\theta$; 
(ii) each of the unit indices are grouped per the partition $\pi \in \Part( \cV_{\theta})$; and
(iii) for each block, a unit index is assigned by the assignment $\omega \in \Omega(\theta, \pi)$. 
See Appendix \ref{sec:example_word_partition} for an example that illustrates words, partition and monomial for the case $d=1$.

\paragraph{Design expectation formula.} 
The definitions above are all deterministic; no probabilities appeared yet. 
Now we take into account the joint inclusion probabilities of assigned indices.

\begin{lemma}[SRSWOR conditional joint inclusion]\label{lem:SRSWOR_inclusion}
    Fix $i \in [n]$. 
    Let $\cS \subset [n]$ be SRSWOR with $|\cS| = m$. 
    For any $\cS' \subset [n] \setminus\{i\}$ with $|\cS'| = \ell < m$,
    \[
        \Pr( \cS' \subset \cS \mid i \in \cS ) = \frac{(m-1)_\ell}{(n-1)_\ell},
    \]
    where $(n)_\ell = \prod_{a=1}^\ell (n+1-a)$ is the falling factorial.
\end{lemma}

We separate allocations depending on whether the fixed index $i$ is used by the assignment, because the corresponding inclusion probabilities differ. 
Injectivity of $\omega$ implies that at most one block can be assigned $i$, hence if $i$ appears there is a unique $a^\star$. 
Define two disjoint classes, namely, Class 0 and Class 1, as 
\begin{align*}
    \Omega(\theta, \pi)_0 &\coloneqq \{ \omega \in \Omega(\theta, \pi): \omega_{a} \neq i, ~\forall a \in [|\pi|] \},\\
    \Omega(\theta, \pi)_1 &\coloneqq \{ \omega \in \Omega(\theta, \pi): \exists! a^* \text{ with }\omega_{a^*} = i \},
\end{align*}
and the corresponding monomial aggregates:
\begin{equation}\label{eqn:diagram_aggregates}
    \begin{aligned}
        \Phi_{\pi, 0}(X; i) &\coloneqq \sum_{\omega \in \Omega(\theta, \pi)_0} \varphi_{\omega}(X; i),\\
        \Phi_{\pi, 1}(X; i) &\coloneqq \sum_{\omega \in \Omega(\theta, \pi)_1} \varphi_{\omega}(X; i).
    \end{aligned}
\end{equation}

\begin{theorem}[Formula for $\xi^{[d]}_i(m; X) $]\label{thm:xi_formula}
    Let $n \in \ZZp$ and $m \in [n]$. 
    For every $d \in \ZZnn$ and $i \in [n]$,
    \begin{align*}
        &\xi^{[d]}_i(m; X) = \sum_{\theta \in \{ I, U, V \}^d \times \{i, \avg\}} \frac{\mathrm{sgn}_{\theta}}{m^{L_{\theta}}}
            \cdot
                \sum_{\pi \in \Part(\cV_{\theta})} \biggl( \, \psi_0(\pi) \cdot \Phi_{\pi, 0}(X; i)
                + \psi_1(\pi) \cdot \Phi_{\pi, 1}(X; i) \, \biggr),
    \end{align*}
    where $\psi_0(\pi) = \frac{(m-1)_{|\pi|}}{(n-1)_{|\pi|}}$ and $\psi_1(\pi) = \frac{(m-1)_{|\pi|-1}}{(n-1)_{|\pi|-1}}$.
\end{theorem}

\begin{proof}[Proof of Theorem \ref{thm:xi_formula}]
    Fix $d$ and $i$ and expand \eqref{eqn:neumann_template}. 
    Each choice of $\theta=(\phi,\chi)$ contributes monomials with overall sign $\mathrm{sgn}_\theta$ and a total of $L_\theta$ averages, yielding the prefactor $m^{-L_\theta}$. 
    The inner‑product structure of each monomial is encoded by $\cG_\theta$: its edges specify which positions are paired in inner products.
    
    Regroup the sum by set partitions $\pi\in\Part(\cV_\theta)$. 
    Within a block of $\pi$, all positions share the same unit index, and the indices are distinct across distinct blocks. 
    For a fixed partition $\pi$ and an injective block assignment $\omega:[|\pi|]\to[n]$, the corresponding monomial equals $\varphi_\omega(X;i)$ by Definition~\ref{defn:allocation}. 
    
    Taking conditional expectation given $i\in\cS$ reduces to the probability that all block indices (excluding $i$ if it is used) are contained in $\cS$. 
    If $\omega\in\Omega(\theta,\pi)_0$, none of the $|\pi|$ assigned indices equals $i$, so by Lemma~\ref{lem:SRSWOR_inclusion} the inclusion probability is $(m-1)_{|\pi|}/(n-1)_{|\pi|}=\psi_0(\pi)$. 
    If $\omega\in\Omega(\theta,\pi)_1$, exactly one block is fixed to $i$ (uniqueness by injectivity), and the remaining $|\pi|-1$ distinct indices must lie in $\cS\setminus\{i\}$; the conditional probability equals $(m-1)_{|\pi|-1}/(n-1)_{|\pi|-1}=\psi_1(\pi)$. 
    Summing the resulting expectations over $\omega$ produces \(\Phi_{\pi,0}(X;i)\) and \(\Phi_{\pi,1}(X;i)\). 
    Finally, summing over all partitions \(\pi\) and words \(\theta\) yields the stated formula.
\end{proof}

Although Theorem \ref{thm:xi_formula} provides an explicit closed-form expression for $\xi^{[d]}_i(m; X)$, brute-force enumeration over all words, partitions, and (injective) allocations can be computationally prohibitive. 
Appendix \ref{sec:computation_expectation} presents implementable procedures (Algorithms~\ref{alg:folding_mobius} and \ref{alg:scalar_folding_mobius}) to compute $\xi^{[d]}_i(m;X)$ and $\bbE[\sfR^{[d]}_{\arm}] = \tfrac{1}{n} \langle \xi^{[d]}, r^{(\arm)}\rangle$ more efficiently. 
Specifically, these approaches enumerate the partition lattice and apply M\"{o}bius inversion, thereby avoiding the summation over $\omega \in \Omega(\theta, \pi)$ that is intractable when $n$ and $d$ are large.

%% file: contents/05_theory_OLS.tex
\section{Main results: Neumann correction for OLS-RA}\label{sec:theory-OLS}

\subsection{Assumptions}\label{sec:assumptions}
Recall that $\bone_n^\top X=0$ and $X^\top X=nI_p$ (Assumption \ref{assump:covariates}); treatment is completely randomized with $n_1+n_0=n$. 
To establish asymptotic theory, we consider a sequence of problems of increasing size $n$ with
\[
    \lim_{n \to \infty} \rho_n = \rho\in(0,1),
    \qquad\text{where}\qquad
    \rho_n \coloneqq \frac{n_1}{n}.
\]

We impose three assumptions described below.

\begin{assumption}[Non-asymptotic envelopes]\label{assumption:rates}
    For any $\delta\in(0,1)$ there exist an event $\Omega_\delta$ with $\Pr(\Omega_\delta)\ge 1-\delta$ and deterministic envelopes
    $\alpha_{n,p}(\delta),\ \beta_{n,p}(\delta),\ \gamma_{n,p}(\delta)>0$ and  such that on $\Omega_\delta$, for both $\arm \in \{0,1\}$,
    \begin{align*}
        \max_{\arm\in\{0,1\}}\ \|\Delta_{\arm}\| &\le \alpha_{n,p}(\delta),\\
        \max_{\arm\in\{0,1\}}\ \|u_{\arm}\|     &\le \beta_{n,p}(\delta),\\
        \max_{\arm\in\{0,1\}}\ \|\bar x_{\arm}\|&\le \gamma_{n,p}(\delta).
    \end{align*}
\end{assumption}

\begin{remark}
    In Appendix \ref{sec:assumption_rate}, we derive an explicit high-probability bound for $\opnorm{ \Delta_{\arm} }$ (hence, an admissible $\alpha_{n,p}(\delta)$) using matrix Bernstein inequality under sampling without replacement \citep{gross2010note}. 
    In particular, on $\Omega_\delta$, if $\alpha_{n,p}(\delta)<1$, the Neumann series for $(I_p - \Delta_{\arm})^{-1}$ converges for both arms.
\end{remark}

\begin{assumption}[Residual regularity]\label{assump:regularity}
    For each $\arm \in\{0,1\}$, the population residuals $r^{(\arm)}$, 
    cf. \eqref{eqn:population_OLS}, satisfy
    \begin{align*}
        &\frac{1}{n}\sum_{i=1}^n\bigl(r^{(\arm)}_i\bigr)^2  \to v_{\arm}\in(0,\infty),\\
        &\frac{1}{n}\sum_{i=1}^n\bigl(r^{(\arm)}_i\bigr)^2 \cdot \Ind \bigl\{|r^{(\arm)}_i|>\epsilon\sqrt{n} \bigr\} \to 0,
            \quad\forall \epsilon>0,
    \end{align*}
    and the finite-population covariance limit exists:
    \[
        v_\tau \coloneqq \lim_{n \to \infty} \frac{1}{n}\sum_{i=1}^n\bigl(r^{(1)}_i-r^{(0)}_i\bigr)^2 \in[0,\infty).
    \]    
\end{assumption}

\begin{remark}
    Assumption~\ref{assump:regularity} is about the standard Lindeberg--H\'{a}jek regularity conditions ensuring a finite-population central limit theorem (CLT) under sampling without replacement; see \citet{hajek1960limiting}. 
    It is \emph{mild}, as it is implied by either (1) a uniform $(2+\delta)$-moment bound for some $\delta>0$ or (2) $\max_i|r^{(\arm)}_i|=o(\sqrt{n})$ together with $\frac{1}{n}\sum_i (r^{(\arm)}_i)^2\to v_{\arm}\in(0,\infty)$. 
\end{remark}

\begin{assumption}[Diffuse leverage]\label{assump:leverage}
    Let $H_X\coloneqq X(X^\top X)^{-1}X^\top$ and $h^X \coloneqq \diag(H_X)$. 
    There exists $\kappa_n$ with
    \[
        \max_{i\in[n]} h^X_i  \leq  \kappa_n,
        \quad
        \kappa_n \lesssim \frac{p}{n},
        \quad 
        \kappa_n\log\bigl( \max\{p,\, n\} \bigr) = o(1).
    \]
\end{assumption}

\begin{remark}[Canonical rates]\label{rem:canonical_rates}
    Under Assumption~\ref{assump:leverage}, standard sampling‑without‑replacement concentration yields envelopes of the order
    \[
        \alpha_{n,p}(\delta) \asymp \sqrt{\frac{p\log(p/\delta)}{n}},~
        \beta_{n,p}(\delta) \asymp \sqrt{\frac{p}{n}},~
        \gamma_{n,p}(\delta) \asymp \sqrt{\frac{p}{n}}.
    \]
    The constants hidden in $\asymp$ can be chosen independent of $n,p$ once $\rho_n \to \rho \in (0,1)$.
\end{remark}

\subsection{Main results for OLS-RA}\label{sec:main-OLS}
We analyze the degree-$d$ Neumann‑corrected OLS estimator as defined in \eqref{eq:NPm-def-final}. 
Recall from \eqref{eq:master-decomp} that $\hresdim = \frac{1}{n_1} \bone_{n_1}^\top S_1 r^{(1)} - \frac{1}{n_0} \bone_{n_0}^\top S_0 r^{(0)}$ denotes the residual difference in means. 
Given $n, p, d, \delta$, define a shorthand
\begin{equation}\label{eqn:epsilon}
    \varepsilon_{n,p}^{[d]}(\delta) 
        \coloneqq 2 \cdot \frac{\alpha_{n,p}(\delta)^{\,d+1}}{1-\alpha_{n,p}(\delta)} \cdot \beta_{n,p}(\delta) \cdot \gamma_{n,p}(\delta).
\end{equation}

\begin{theorem}[Consistency and rate of convergence]\label{thm:main-OLS}
    Suppose that Assumption~\ref{assumption:rates} holds with $\alpha_{n,p}(\delta) < 1$. 
    If $p = o(n)$ and $\sup_{\arm \in \{0,1\}} \sup_n n^{-1} \bigl\| r^{(\arm)} \bigr\|_2^2 < \infty$, then for any $d\geq 0$,
    \begin{equation}\label{eqn:main_theorem}
        \htauols^{[d]}-\tau
            = \hresdim \, +\, \Delta_{n,p}^{[d]} \, +\, o_P\big(n^{-1/2}\big),
    \end{equation}
    and for any $\delta \in (0,1)$, the remainder term $\Delta_{n,p}^{[d]}$ satisfies with probability at least $1 - \delta$ that
    \begin{equation}\label{eq:two-term-envelope}
        \bigl| \Delta_{n,p}^{[d]} \bigr|
            \leq  \varepsilon_{n,p}^{[d]}(\delta).
    \end{equation}
\end{theorem}

\begin{remark}
    The moment condition required in Theorem \ref{thm:main-OLS}, $\sup_{\arm \in \{0,1\}}\sup_n n^{-1} \bigl\| r^{(\arm)} \bigr\|_2^2 < \infty$, is strictly weaker than Assumption \ref{assump:regularity}. 
\end{remark}

In \eqref{eqn:main_theorem}, $\Delta_{n,p}^{[d]}$ collects the (deterministic) Neumann tail $\sum_{d'= d+1}^{\infty} \xbara^\top \, \Delta_{\arm}^{d'} \, u_{\arm}$ in \eqref{eqn:R-approx}, while the $o_P(n^{-1/2})$ term is the stochastic fluctuation in the correction-estimation error $\whsfR^{[d']}_{\arm} - \sfR^{[d']}_{\arm}$ across $d'=0,\dots,d$. 
See Section \ref{sec:proof_sketch} for a proof sketch and Appendix \ref{sec:proof_main_theorem} for a full proof.

\begin{theorem}[Asymptotic normality]\label{thm:AN-OLS}
    Suppose that Assumptions \ref{assumption:rates} and \ref{assump:regularity} hold. 
    If $p=o(n)$ and $\varepsilon_{n,p}^{[d]}(\delta) = o\big(n^{-1/2}\big)$, then
    \[
        \sqrt{n} \bigl(\htauols^{[d]}-\tau\bigr)\ \xRightarrow{d}\ \cN\bigl(0, \sigma^2_{\mathrm{res}}\bigr),
    \]
    where $\sigma^2_{\mathrm{res}} \coloneqq \frac{v_1}{\rho}+\frac{v_0}{1-\rho}-v_\tau$.
\end{theorem}

\begin{proof}[Proof sketch of Theorem \ref{thm:AN-OLS}]
    By Theorem~\ref{thm:main-OLS}, $\htauols^{[d]}-\tau = \hresdim + o_P(n^{-1/2})$. 
    Since $\hresdim$ is a difference of sample means (of population residuals) under complete randomization, Assumption \ref{assump:regularity} yields the finite-population CLT, namely, $\sqrt{n} \hresdim \xRightarrow{d}\ \cN\bigl(0, \sigma^2_{\mathrm{res}}\bigr)$; see \citet{lin_ols2013}. 
    The remaining $o_P(n^{-1/2})$ term is negligible.
\end{proof}

\begin{corollary}[$p$-scaling under diffuse-leverage]\label{cor:diffuse}
    Suppose Assumptions \ref{assumption:rates}, \ref{assump:regularity}, \ref{assump:leverage} hold. 
    Then $\htauols^{[d]}$ is $\sqrt{n}$-consistent and asymptotically normal with variance $\sigma^2_{\mathrm{res}}$ when $p = o\bigl(n^{\frac{d+2}{d+3}}\bigr)$, ignoring logs. 
\end{corollary}

Note the allowable growth regime for $p$ of asymptotic normality expands with each additional degree of correction (e.g., $d=0$: $p=o(n^{2/3})$; $d=1$: $p=o(n^{3/4})$).

\begin{proof}[Proof sketch of Corollary \ref{cor:diffuse}]
    Under the canonical rates in Remark \ref{rem:canonical_rates}, the envelope in \eqref{eq:two-term-envelope} becomes
    \[
        \varepsilon_{n,p}^{[d]}(\delta)
        \ =\ O_P\biggl(\, \Bigl(\frac{p}{n}\Bigr)^{\frac{d+3}{2}} (\log p)^{\frac{d+1}{2}} \,\biggr).
    \]
    Consequently, for any fixed correction degree $d\geq 0$, if
    \begin{equation}\label{eq:diffuse-scaling}
        p^{ d+3}(\log p)^{ d+1} \ =\ o\big(n^{ d+2}\big),
    \end{equation}
    then the premise of Theorem \ref{thm:AN-OLS} is satisfied.
\end{proof}

\paragraph{Discussion: What enables the gains?} 
Two ingredients drive the improvement. 
\begin{enumerate}[label=(\roman*)]
    \item
    \emph{Geometric bias control.} 
    Truncating the Neumann series cancels successive powers of the perturbation $\Delta_{\arm}$; after degree $d$ the bias tail $\sum_{d'\ge d+1}\bar x_{\arm}^\top(-\Delta_{\arm})^{d'}u_{\arm}$ is bounded on $\Omega_\delta$ by $\alpha_{n,p}^{\,d+1}\,\beta_{n,p}\,\gamma_{n,p}$. 
    
    \item
    \emph{Stable estimation of corrections.} 
    A single treatment–control swap changes $(\bar x_{\arm},\Delta_{\arm},u_{\arm})$ by $O(1/n)$, so the average swap sensitivity of each correction term is $O(1/n)$; the Poincar\'e bound on the Johnson graph yields variance $O(1/n)$ and an $o_P(n^{-1/2})$ aggregate estimation error for fixed $d$ (Appendix \ref{sec:Poincare_Johnson}). 
\end{enumerate}

Together, these imply that under diffuse leverage with $\alpha_{n,p}\asymp\sqrt{p/n}$, each added degree enlarges the admissible dimension of $p$ to $p=o \big(n^{(d+2)/(d+3)}\big)$ (up to logs).

\subsection{Proof sketch of Theorem \ref{thm:main-OLS}}\label{sec:proof_sketch}

Recall $\htauols^{[d]} \coloneqq \hat\tau_{\mathrm{OLS}} + \sum_{d'=0}^{d} \left( \whsfR_1^{[d']}-\whsfR_0^{[d']} \right)$, cf. \eqref{eq:NPm-def-final}, and $\hat\tau_{\mathrm{OLS}}-\tau = \hresdim - \bigl(\sfR_1-\sfR_0\bigr)$, cf. \eqref{eq:master-decomp}. 
On $\Omega_{\delta}$, $\opnorm{\Delta_{\arm}} \leq \alpha_{n,p}(\delta) < 1$, so $(I_p - \Delta_{\arm})^{-1} = \sum_{d'=0}^\infty \Delta_{\arm}^{d'}$ converges, cf. \eqref{eqn:R-approx}, and therefore, $\sfR_{\arm} = \sum_{d' = 0}^{\infty} \sfR_{\arm}^{[d']}$. 
Defining $\sfT^{[d]}_{\arm} \coloneqq \sum_{d' \geq d+1} \sfR^{[d']}_{\arm}$,  we write $\sfR_{\arm} = \sum_{d' = 0}^{d} \sfR_{\arm}^{[d']} + \sfT^{[d]}_{\arm}$. 
Then, we obtain
\begin{align*}
    \htauols^{[d]} - \tau 
        &= 
        \hresdim
        + \underbrace{\left\{ \sum_{d'=0}^d\left( \whsfR_1^{[d']} \!-\!  \sfR_1^{[d']}\right) - \left( \whsfR_0^{[d']} \!-\! \sfR_0^{[d']} \right) \right\}}_{\text{correction-estimation error}}
        - \underbrace{\left( \sfT_1^{[d]} - \sfT_0^{[d]} \right)}_{\text{Neumann tail}}.
\end{align*}

\begin{itemize}
    \item 
    \emph{Tail term.} 
    On $\Omega_\delta$, $|\sfT^{[d]}_{\arm}|\le \|\bar x_{\arm}\|\,\|\Delta_{\arm}\|^{d+1}(1-\|\Delta_{\arm}\|)^{-1}\|u_{\arm}\|\le \frac{\alpha_{n,p}^{\,d+1}}{1-\alpha_{n,p}}\beta_{n,p}\gamma_{n,p}$, so $|\Delta_{n,p}^{[d]}|\le \varepsilon_{n,p}^{[d]}(\delta)$.
    \item 
    \emph{Correction-estimation error.} 
    For each $\arm \in \{0,1\}$ and each $d' \in \{0, 1, \dots, d\}$, define $F_{\arm, d'} \coloneqq \whsfR^{[d']}_{\arm} - \sfR^{[d']}_{\arm}$. 
    Viewing $F_{\arm, d'}$ as a function of random subset $\Sa$, we analyze its swap sensitivity, $F_{\arm, d'}(\Sa') - F_{\arm, d'}(\Sa)$ with $\Sa' = \Sa'(j,k) = (\Sa \setminus \{j\}) \cup \{k\}$ for $(j,k) \in \Sa \times \Sa^c$. 
    Design‑based bounds give average swap sensitivity $O(1/n)$; the Poincar\'e inequality on Johnson‑graph then yields $\Var(F_{\arm, d'}) = o(1/n)$ when $p = o(n)$; see Appendix \ref{sec:Poincare_Johnson}. 
    Since $d$ is fixed, summing over $d'=0,\dots,d$ and $\arm \in \{0,1\}$ preserves $o(1/n)$ variance, and Chebyshev implies an $o_P(n^{-1/2})$ contribution.
\end{itemize}

See Appendix \ref{sec:proof_main_theorem} for a complete proof.

%% file: contents/06_experiments.tex
\section{Experiments}\label{sec:experiments}

\subsection{Setup}

\paragraph{Data generative process.}
We set $n = 500$, $n_1 = \rho n$ with $\rho = 0.3$. 
For each outer replicate ($R = 50$), we generate a single ``master'' matrix $\widetilde{X} \in \RR^{n \times n}$ with i.i.d.\ entries from standard Gaussian or $t(2)$, which we hold fixed for a controlled comparison across covriate dimensionality $p$. 
For each exponent $\gamma \in \{ 0, 0.1, \dots, 0.7\}$, let $p = \lceil n^{\gamma} \rceil$ and form $X\in\RR^{n\times p}$ by taking the first $p$ columns of $\widetilde X$. 
We then column–center and rescale $X$ so that $\bone_n^\top X=0$ and $X^\top X=n I_p$, matching Assumption~\ref{assump:covariates}.

With $X$ instantiated, potential outcomes are generated as $y^{(1)}=X\beta^*_1+\varepsilon^{(1)}$ and $y^{(0)}=X\beta^*_0+\varepsilon^{(0)}$. 
Because $\hat{\beta}_{\arm} - \beta^*_{\arm} = (X_\arm^{\top} X_\arm)^{-1} X_\arm^{\top} \varepsilon^{(\arm)}$ does not depend on $\beta^*_\arm$, we may take $\beta^*_1, \beta^*_0$ arbitrarily without affecting the bias and variance of the estimates $\hat{\tau}$; here, we generate a random unit-norm vector $\beta^*$ and set $\beta^*_1 = \beta^*_0 = \beta^*$ for comparison of the (Neumann-corrected) OLS-RA estimators against the DiM baseline. 
We consider two types of residual models for $\varepsilon$ as follows:


\begin{enumerate}[label=(RM\arabic*)]
    \item\label{item:typical} 
    \emph{Typical case.} 
    $\{ \varepsilon^{(1)}, \varepsilon^{(0)} \}$ are independent random vectors with i.i.d. entries from $\cN(0,1)$. 

    \item\label{item:worst} 
    \emph{Worst case.} 
    Let $H\coloneqq X(X^\top X)^{-1}X^\top$ and $h\coloneqq \diag(H)$. 
    Set $\varepsilon^{(0)}=\varepsilon$ and $\varepsilon^{(1)}=3\varepsilon$, where 
    \begin{equation}\label{eqn:res_worst}
    \begin{aligned}
        \varepsilon ~\in~& \arg\max_{\varepsilon'\in\RR^n}
            \Bigl(\tfrac{3n_0}{n_1}-\tfrac{n_1}{n_0}\Bigr)\,\bigl|h^\top \varepsilon'\bigr|\\
        &\text{s.t.}\quad \tfrac{1}{n}\|\varepsilon'\|_2^2=1,\ \ X^\top\varepsilon'=0,\ \ \bone_n^\top\varepsilon'=0.
    \end{aligned}
    \end{equation}
    An explicit maximizer is 
    $\varepsilon^\star = \sqrt{n} \frac{ (I_n - H_X)\bigl( h - \tfrac{p}{n}\bone_n \bigr) }{ \big\| (I_n - H_X)\bigl( h - \tfrac{p}{n}\bone_n \bigr) \big\|_2 }$, which is unique up to sign; if $(I_n-H_X)(h-\tfrac{p}{n}\bone_n)=0$, any feasible $\varepsilon$ satisfies \eqref{eqn:res_worst}. 
    This choice of $\varepsilon = \varepsilon^\star$ maximizes the leading (zeroth–order) term in the Neumann-series expansion of the remainder in \eqref{eq:master-decomp}; cf.\ \citet[Eq.~(16)]{lei2021regression}. 
\end{enumerate}

We keep $(X, y^{(1)}, y^{(0)})$ fixed while sampling treatment assignments (complete randomization) and compare (Neumann‑corrected) OLS‑RA estimators to the DiM baseline. 
Given $X$ and $(y^{(1)},y^{(0)})$, we draw $N=2000$ treatment assignments $T^{(k)}\in\{0,1\}^n$ uniformly with $\bone_n^\top T^{(k)}=n_1$. 
For each $k$, the observed outcome is $Y^{(k)}=T^{(k)}\odot y^{(1)}+(1-T^{(k)})\odot y^{(0)}$. 
This entire process is repeated $R = 50$ times with different random seeds.

\paragraph{Evaluation.} 
For each instance, we compute $\hat\tau_{\mathrm{DiM}}$, the arm‑wise OLS‑RA estimator $\hat\tau_{\mathrm{OLS}}$, and the Neumann‑corrected estimators $\hat\tau_{\mathrm{OLS}}^{[d]}$ for $d \in \{0,1,2,3\}$.  
Let $\{\hat\tau_k\}_{k=1}^N$ be the $N$ estimates of a given method. 
We report 
\begin{enumerate}[label=(\roman*)]
    \item 
    the normalized absolute bias  
    $\Bigl|\frac{1}{N}\sum_{k=1}^N \hat\tau_k - \tau\Bigr|\cdot\frac{\sqrt n}{\sigma_n}$,
    and
    \item 
    the normalized empirical variance 
    $\Var(\hat{\tau}_1, \dots, \hat{\tau}_N)\cdot \frac{n}{\sigma_n^2}$ (sample variance across replicates), 
\end{enumerate}
where $\sigma_n^2 = n_1^{-1} \| r^{(1)} \|_2^2 + n_0^{-1} \| r^{(0)} \|_2^2 - n^{-1} \| r^{(1)} - r^{(0)} \|_2^2$ with $r^{(a)}$ the \emph{population} OLS residuals. 
This scaling is for comparability only and is not used by the estimators. 
For each $p$, we compute both measures over the $N$ assignments, repeat the outer replication (varying $\widetilde{X}$), and report the median along with the 10\%-90\% envelope across $R = 50$ replicates.

\paragraph{Computing environment.}
All experiments were run on a single Apple Mac mini (M4 Pro, Apple silicon) under macOS using Python (NumPy/SciPy; CPU only.

\subsection{Results}
Here we focus on the “worst-case” construction in \ref{item:worst}, deferring additional results to Appendix~\ref{sec:additional_experiments}\footnote{In the typical case, differences are smaller, as expected for linear outcomes where $\hat\tau_{\mathrm{OLS}}$ has near‑zero bias.}.

\paragraph{Experiment 1. Bias and variance against correction degrees.}
First, we consider the setting where $\widehat{X}$ has i.i.d. standard Gaussian entries (before centering/normalization). 
Figure~\ref{figure:experiment.1} shows that the normalized absolute bias of $\hat\tau_{\mathrm{OLS}}^{[d]}$ decreases monotonically with $d$, and is already small at $d=1$ or $2$.  
On the other hand, the normalized empirical variance slightly increases with $d$ but changes little across $d$ and remains comparable to $\hat\tau_{\mathrm{OLS}}$. 
We observe that the OLS-RA and Neumann-corrected estimators all exhibit smaller variance than the difference-in-means variance baseline, as the part of outcomes explainable by covariates is neutralized to reduce the variance. 

\begin{figure}[t]
    \centering
    \includegraphics[width=0.42\linewidth]{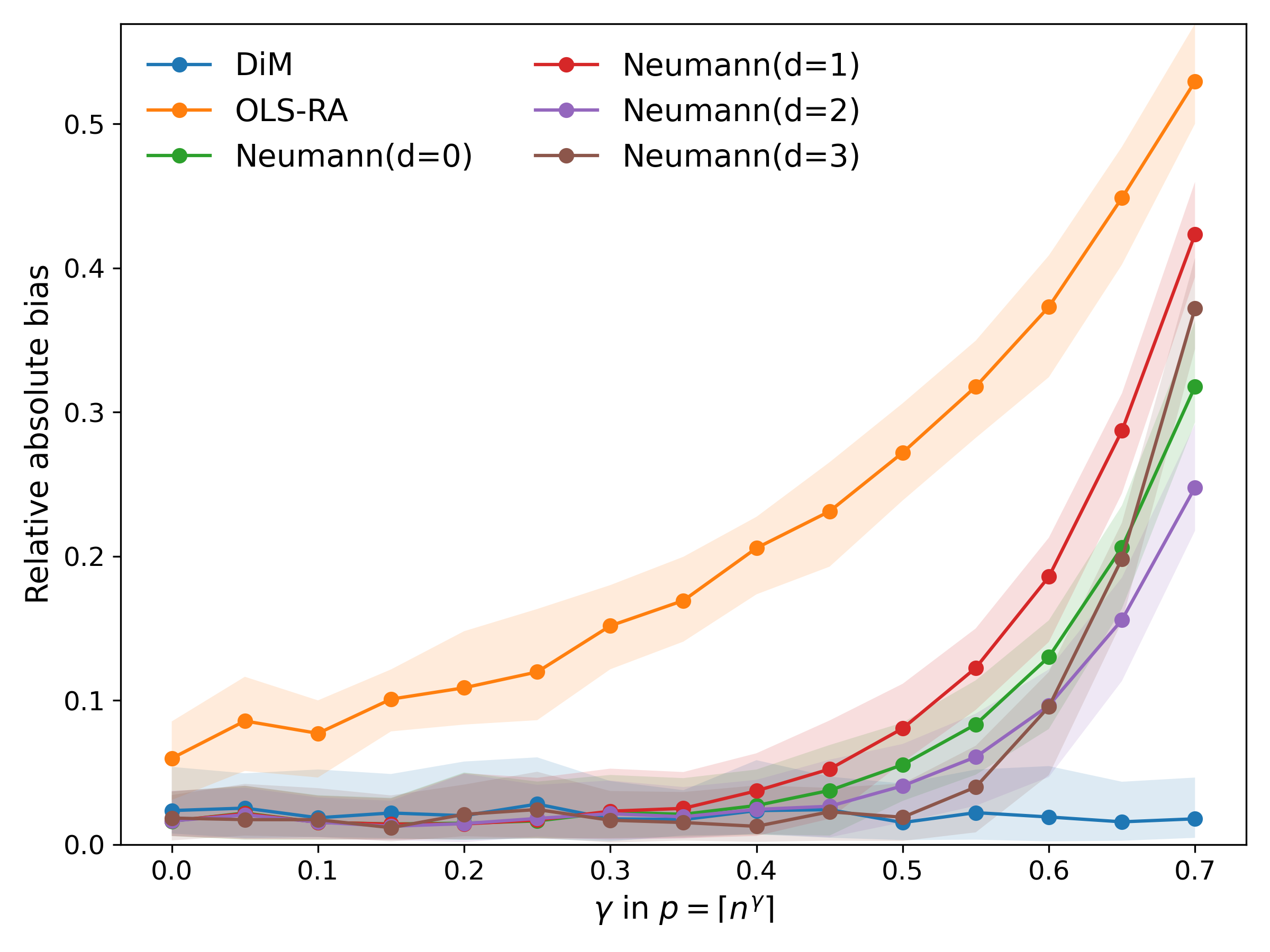}\qquad
    \includegraphics[width=0.42\linewidth]{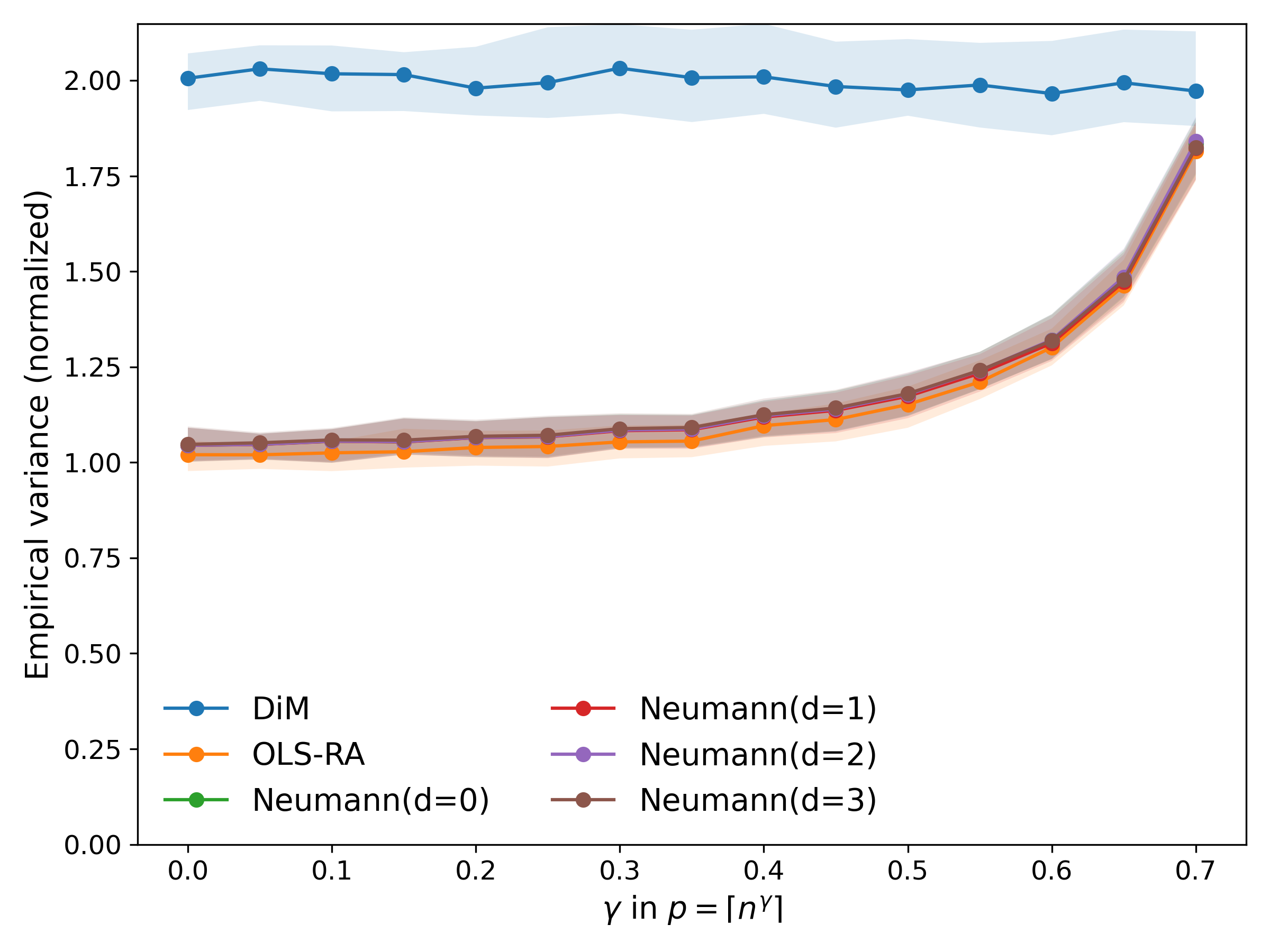}
    \caption{ \small
        Comparison of $\hat\tau_{\mathrm{OLS}}^{[d]}$ for $d \in \{0,1,2,3\}$ against $\hat\tau_{\mathrm{OLS}}$ and $\htaudim$ under Gaussian random $X$ and the worst‑case residual model \ref{item:worst} at $n = 500$, $n_1 = 150$, $N = 2000$.  
        The shaded bands indicate $10\%$--$90\%$ regions across the $R=50$ outer repetitions.
        \textbf{(Left):} Normalized absolute bias tends to decrease as the correction degree $d$ increases. 
        \textbf{(Right):} Normalized empirical variance slightly increases as $d$ grows, but the OLS-RA and Neumann-corrected estimators all exhibit smaller variance than the difference-in-means variance baseline, as regression adjustment \emph{compensates for} the part of outcomes explainable by covariates, thereby reducing the variance. 
        } 
    \label{figure:experiment.1}
\end{figure}

\paragraph{Experiment 2. Heavy-tailed covariates.}
We repeat Experiment~1 with the entries of $\widetilde{X}$ drawn i.i.d. from $t(2)$ distribution, to inspect the effects of heavy-tailed covariates. 
Compared to the Gaussian case above, the bias and the variance exhibit more erratic patterns; we observe different (complementary) patterns for even‑ and odd‑degree corrections.
From Figure~\ref{figure:experiment.2}, observe an alternating bias–variance trade‑off as $d$ increases: moving from $d-1$ to an \emph{even} $d$ reduces normalized bias at some cost in normalized variance, while moving to an \emph{odd} $d$ tends to reverse that trade‑off. 
This is somewhat anticipated because the convergence and stability of the Neumann-series expansion \ref{eqn:R-approx} hinges on the premise $\opnorm{\Delta_{\arm}}<1$, and for the heavy-tailed covariates, $\opnorm{\Delta_{\arm}}$ can be close to, or even exceed $1$, making higher-degree corrections unstable.

\begin{figure}[t]
    \centering
    \includegraphics[width=0.42\linewidth]{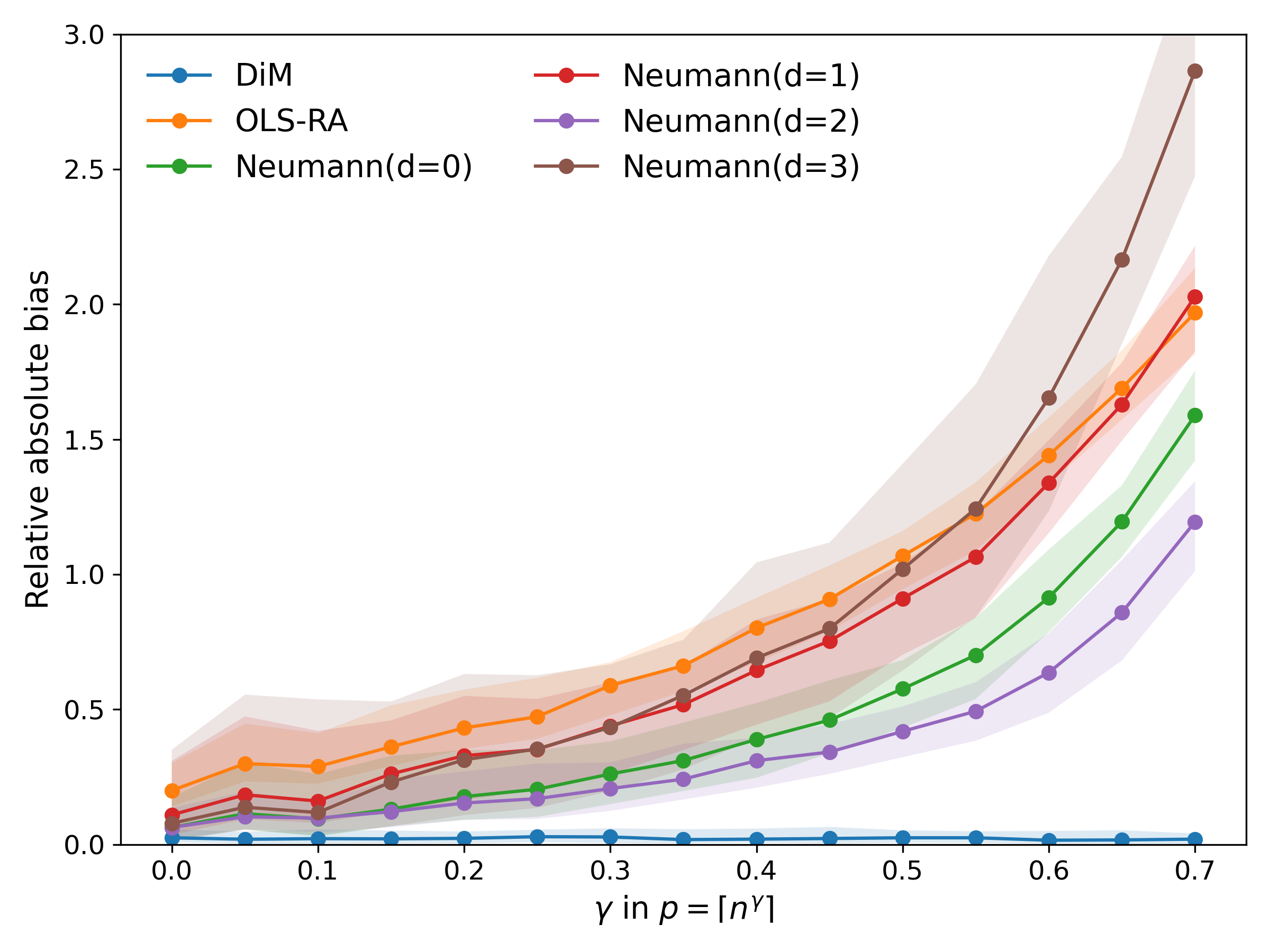}\qquad
    \includegraphics[width=0.42\linewidth]{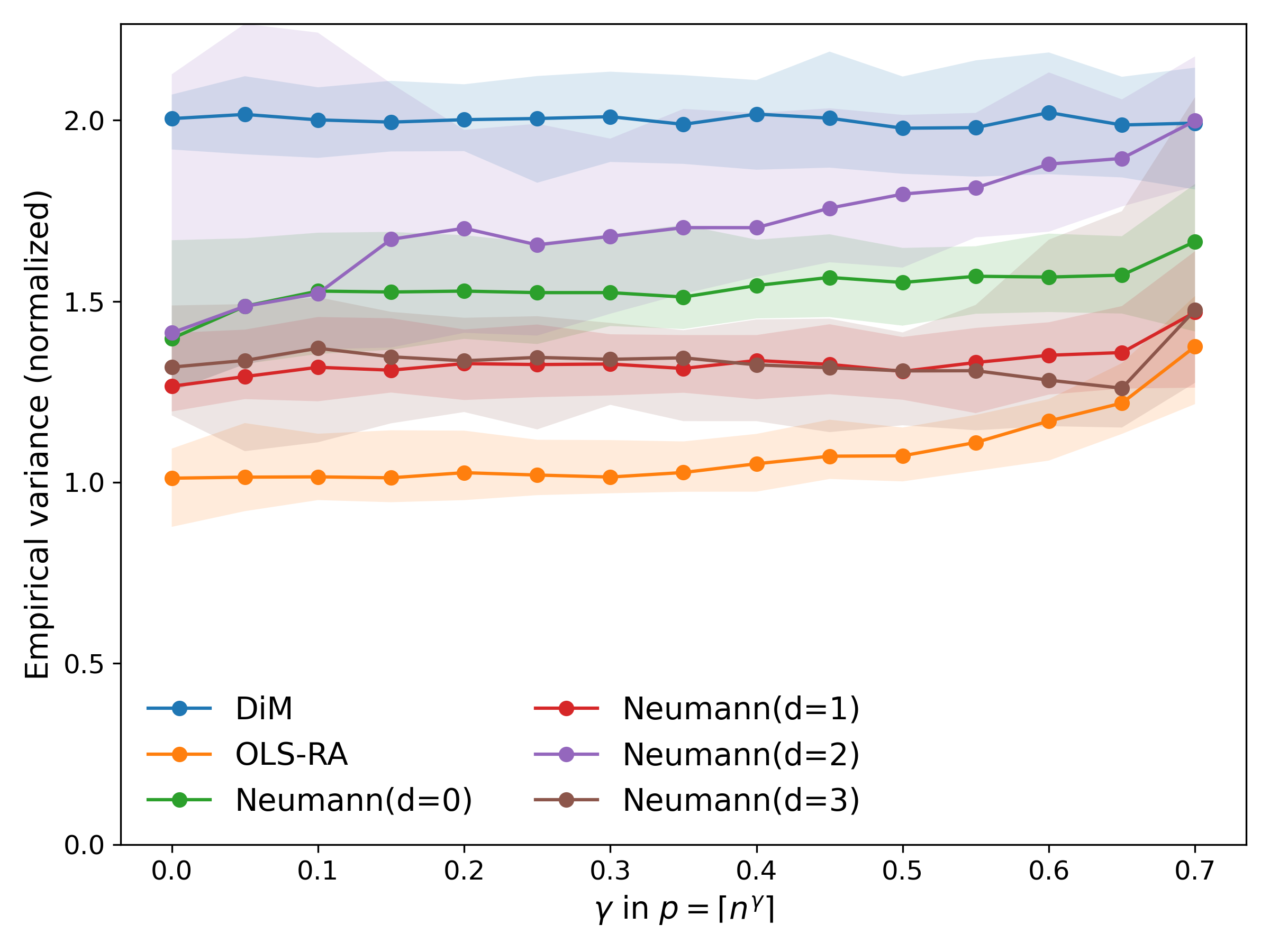}
    \caption{ \small
        Comparison of $\hat\tau_{\mathrm{OLS}}^{[d]}$ for $d \in \{0,1,2,3\}$ against $\hat\tau_{\mathrm{OLS}}$ and $\htaudim$ under $t(2)$ random $X$ and the worst‑case residual model\ref{item:worst} at $n = 500$, $n_1 = 150$, $N = 2000$.  
        The shaded bands indicate $10\%$--$90\%$ regions across the $R=50$ outer repetitions.
        \textbf{(Left):} Normalized absolute bias. 
        \textbf{(Right):} Normalized empirical variance.
        We observe an alternating pattern: even‑degree corrections further reduce bias at the expense of higher variance (relative to $d-1$), while odd‑degree corrections tend to lower variance but can increase absolute bias.
        } 
    \label{figure:experiment.2}
\end{figure}

\paragraph{Experiment 3. Covariate trimming.}
Recall that in randomized experiments, regression adjustment does not require a correctly specified outcome model; even when the linear outcome model is not correct, the RA never hurts the estimator. 
Thus, we may \emph{trim off} covariates via Winsorization, to regularize the data so that the trimmed covariates have smaller $\opnorm{\Delta_{\arm}}$ and more well‑behaved maximum leverage score (closer to $p/n$, cf.\ Assumption~\ref{assump:leverage}), which governs the asymptotic properties of corrected estimators; see \citet[Section~4.5]{lei2021regression} for related discussion. 
For the $t(2)$ design setting considered in the previous paragraph, we trim each of the $p$ covariates at its $5\%$ and $95\%$ quantiles. 
From Figure~\ref{figure:experiment.3}, observe that with covariate trimming, even under heavy‑tailed covariates, Neumann‑corrected estimators reduce bias more effectively, without increasing variance much.

\begin{figure}[t]
    \centering
    \includegraphics[width=0.42\linewidth]{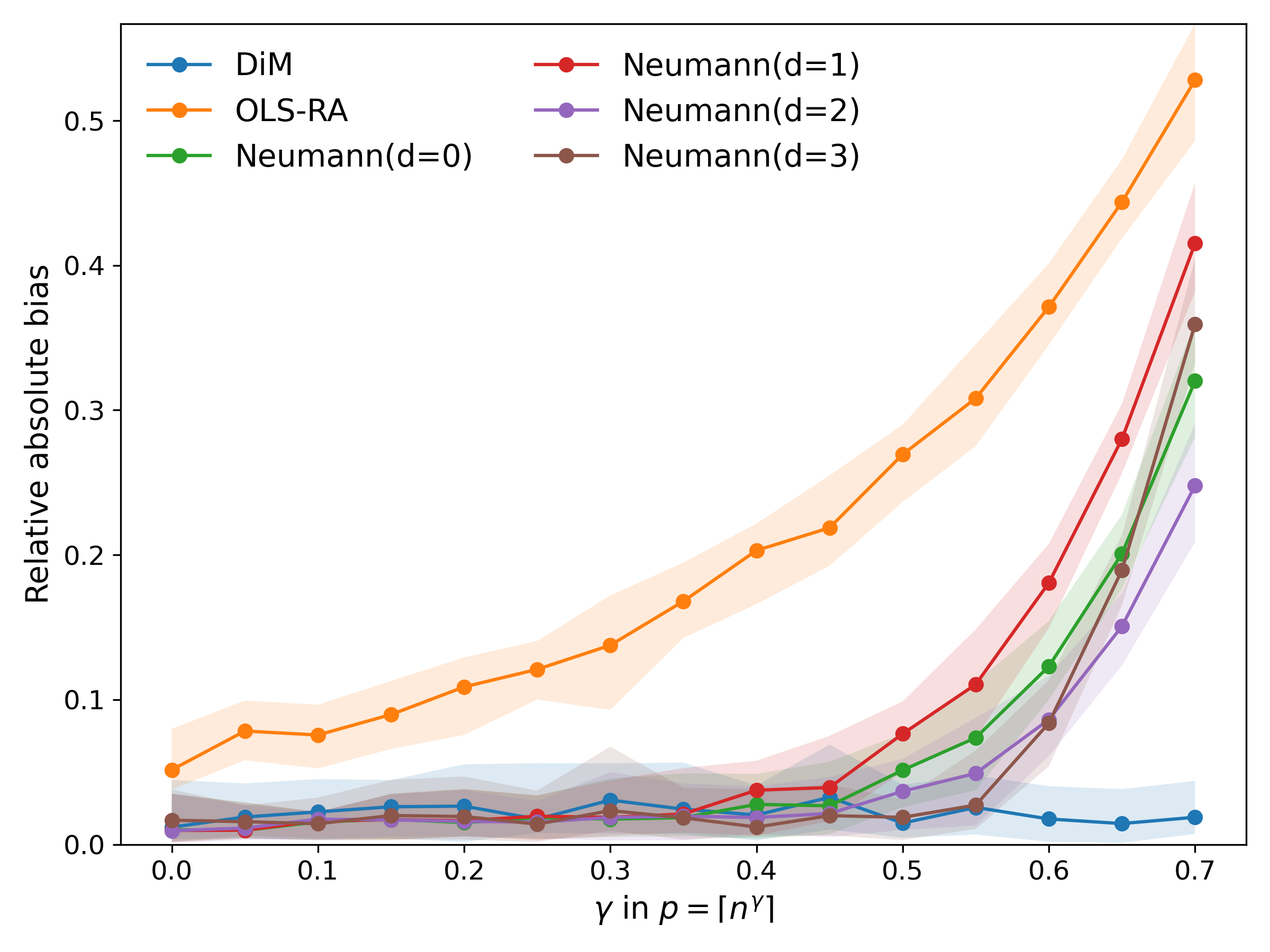}\qquad
    \includegraphics[width=0.42\linewidth]{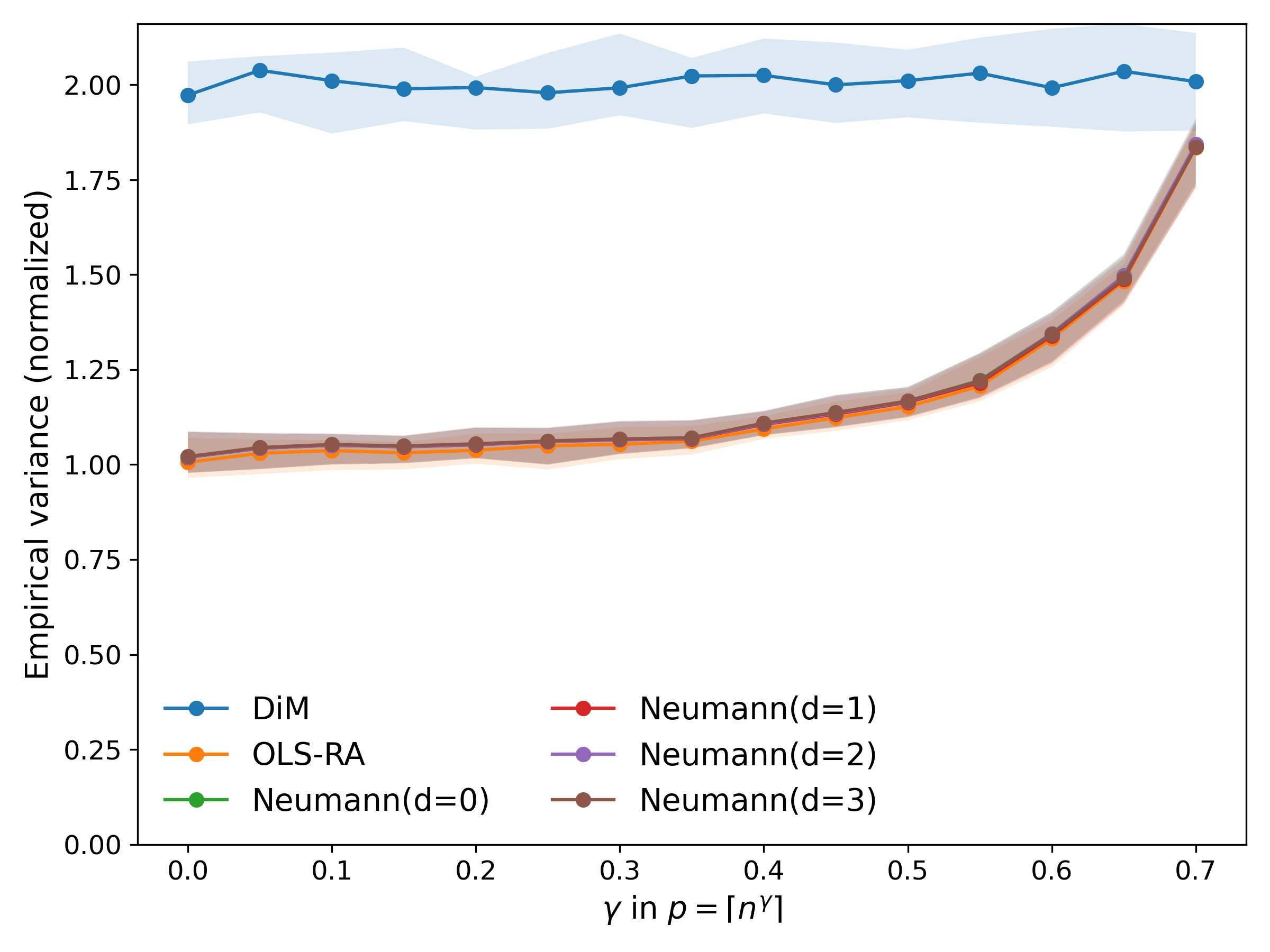}
    \caption{ \small
        Comparison of $\hat\tau_{\mathrm{OLS}}^{[d]}$ for $d \in \{0,1,2,3\}$ against $\hat\tau_{\mathrm{OLS}}$ and $\htaudim$ under $t(2)$ random $X$ with each covariate trimmed at its $5\%$ and $95\%$ quantiles, and the worst‑case residual model~\ref{item:worst} at $n = 500$, $n_1 = 150$, $N = 2000$.  
        The shaded bands indicate $10\%$--$90\%$ regions across the $R=25$ outer repetitions.
        \textbf{(Left):} Normalized absolute bias. 
        \textbf{(Right):} Normalized empirical variance.
        With covariate trimming, even under heavy‑tailed covariates, Neumann‑corrected estimators reduce bias more effectively, without increasing variance much.
        } 
    \label{figure:experiment.3}
\end{figure}

%% file: contents/07_conclusion.tex
\section{Conclusion}\label{sec:conclusion}

We studied regression adjustment for average treatment effect estimation under complete randomization from a finite‑population, design‑based perspective. 
Identifying the random inverse in the arm‑wise OLS remainder as the bottleneck, we introduced a Neumann‑series correction that replaces the inverse with a truncated series and yields simple, arm‑wise correction terms computable from standard OLS fits. 
Under non‑asymptotic envelopes for the centered Gram matrix and residualized means, a Poincar\'e (spectral‑gap) inequality on the Johnson graph (via one‑swap couplings) shows the aggregate estimation error of any fixed number of correction terms is $o_P(n^{-1/2})$. 
Consequently, the degree‑$d$ Neumann‑corrected estimator is asymptotically normal whenever
\[
    p^{\,d+3}(\log p)^{\,d+1}=o(n^{\,d+2}),
\]
strictly expanding the admissible $p$‑growth beyond existing design‑based analyses. 
Simulations corroborate the theory and illustrate a bias–variance trade-off: even degrees mainly reduce bias, while odd degrees tend to temper variance, with the gains most visible under adversarial residual alignment with leverage.

Nevertheless, the analysis in this paper has several limitations, each suggesting a concrete direction for future work. 
First of all, we currently require $p<n$, diffuse leverage, and mild residual regularity; relaxing these (e.g., via leverage trimming or robustification of the envelopes) and tightening constants/log factors are natural targets would sharpen the guarantees. 
Secondly, in this work, we focus on point estimation for a fixed degree $d$; developing data‑driven selection of $d$ and finite‑sample variance estimators enabling valid confidence intervals are important next steps. 
Moreover, the present analysis covers complete randomization; extending the Neumann corrections to stratified, clustered, rerandomized, or covariate‑adaptive assignments would broaden applicability across common experimental settings. 
Further, our theory is focused on OLS in the classical ($n>p$) regime; extending Neumann corrections to stabilized inverses (ridge or pseudoinverses) could bridge to the over‑parameterized OLS‑interpolator regime, potentially unifying design‑based corrections with benign‑overfitting phenomena; connections to learned features or cross‑fitting would also be an exciting direction. 
Finally, developing scalable algorithms to compute higher‑degree Neumann-correction weights would enhance practical utility.

%% file: contents/Appx01_assumptions.tex
\section{Additional discussion on Assumption \ref{assumption:rates} in Section \ref{sec:assumptions}}\label{sec:assumption_rate}

This appendix section collects the minimal concentration tools for sampling without replacement that we use, and derives high‑probability envelopes that instantiate Assumption~\ref{assumption:rates}. 
Section~\ref{sec:primer} records two generic results (scalar and matrix Bernstein inequalities without replacement). 
Section~\ref{sec:envelopes} states three lemmas giving explicit envelopes for the arm‑wise empirical mean $\bar x_{\cS}$, the vector $u_{\cS}(x;r)$, and the operator deviation $\opnorm{ \Sigma_{\cS}-I_p }$, with their proofs recorded in Section~\ref{sec:proof_lemmas}.  
Throughout, for vectors we write $\|\cdot\|$ for the Euclidean norm, and for matrices $\opnorm{ \cdot }$ for the spectral norm.

\subsection{A primer on concentration under sampling without replacement}
\label{sec:primer}

Throughout, let $\cS\subset[n]$ be a uniformly random subset of fixed size $m$. 
We use the following Bernstein‑Serfling-type inequality with finite‑population correction.

\begin{proposition}[{\citet[Theorem 3.5]{bardenet2015concentration}}]
\label{prop:bardenet-maillard}
    Let $\{a_i\}_{i=1}^n\subset\RR$ with mean $\bar a = n^{-1}\sum_i a_i$, range $A\leq  a_i\leq  B$, and population variance $\sigma_a^2 = n^{-1}\sum_i (a_i-\bar a)^2$. If $\cS$ is a size‑$m$ SRSWOR sample, then for all $t>0$ and all $\eta \in (0,1)$,
    \begin{equation}\label{eq:bernstein-srswor}
        \Pr \left( \frac{1}{m}\sum_{i\in\cS} a_i - \bar a \geq  t \right)
            \leq  \exp \left(
                -\frac{m\,t^2 /  2}{ \gamma_{\eta}^2 + \frac{2}{3}(B-A)\,t}\right) +  \eta,
    \end{equation}
    where 
    \[
        \gamma_{\eta}^2 = \frac{n-m+1}{n} \sigma_a^2 +\frac{m-1}{n} \sigma_a (B-A) \sqrt{ \frac{ 2 \log(1/\eta)}{m} }.
    \]
\end{proposition}
The equation display in \eqref{eq:bernstein-srswor} is a one‑sided bound; a two‑sided deviation follows by applying it to both $\{a_i\}$ and $\{-a_i\}$ and taking a union bound.

Similarly, we can control the operator norm of sums of centered Hermitian matrices sampled without replacement; this type of result is commonly referred to as a matrix Bernstein inequality. 
\begin{proposition}[\cite{gross2010note}, Theorem 1]\label{prop:matrix_Bernstein}
    Let $\{M_i\}_{i=1}^n$ be $p\times p$ Hermitian matrices with $\frac{1}{n}\sum_i M_i = 0$. 
     Define
    \[
        v\coloneqq \opnorm{ \, \frac{1}{n}\sum_{i=1}^n M_i^2 \, },\qquad
        c\coloneqq\max_{1\leq i\leq n} \opnorm{ M_i }.
    \]
    If $\cS$ is a size‑$m$ SRSWOR sample, then for all $t>0$,
    \begin{equation}\label{eq:matrix-bernstein-srswor}
        \Pr \left(\left\|\frac{1}{m}\sum_{i\in\cS} M_i\right\|>t\right)
        \ \le\ 2p\;\exp \left(-m\,\min\left\{\frac{t^2}{4v},\ \frac{t}{2c}\right\}\right).
    \end{equation}
\end{proposition}

\citet{gross2010note} show that the trace of matrix moment generating function (mgf) of a sample without replacement is dominated by the trace of the matrix mgf of a sample with replacement, thereby proving the matrix Bernstein tails for i.i.d. samples apply to SRSWOR; this comparison argument, which transfers i.i.d. concentration to SRSWOR, is ultimately due to \citet{hoeffding1963probability}.

For later use, we record an immediate inversion of Propositions~\ref{prop:bardenet-maillard} and \ref{prop:matrix_Bernstein}. 
This inversion is routine and standard; see e.g., \citep{vershynin2018high}.

\begin{corollary}[Convenient $\delta$‑level forms]\label{cor:delta-level}
    \leavevmode
    \begin{enumerate}[label=(\alph*),leftmargin=2em]
        \item\label{item:scalar-invert}
        Let $\{a_i\}$ be as in Proposition~\ref{prop:bardenet-maillard}, and fix $\delta\in(0,1)$. Set $\sfL(\delta) \coloneqq \log(4/\delta)$ and
        \[
            t_{\rm scal}(\delta)\ \coloneqq\ \sqrt{\frac{2\,\Gamma^2(\delta)\,\sfL(\delta)}{m}}\;+\;\frac{2(B-A)}{3m}\,\sfL(\delta),
            \quad\text{where}\quad
            \Gamma^2(\delta)\ \coloneqq\ \frac{n-m+1}{n}\,\sigma_a^2 \;+\; \frac{m-1}{n}\,\sigma_a(B-A)\,\sqrt{\frac{2\sfL(\delta)}{m}}.
        \]
        Then $\Pr \Big(\big|\frac{1}{m}\sum_{i\in\cS}a_i-\bar a\big|>t_{\rm scal}(\delta)\Big)\le \delta$.
        
        \item\label{item:matrix-invert}
        Let $\{M_i\}$ be as in Proposition~\ref{prop:matrix_Bernstein}, fix $\delta\in(0,1)$, set $\sfL_p(\delta) \coloneqq\log \big(2p/\delta\big)$ and
        \[
            t_{\rm mat}(\delta)\ \coloneqq\ \max \left\{\,2\sqrt{\frac{v\,\sfL_p(\delta)}{m}}\,,\ \frac{2c\,\sfL_p(\delta)}{m}\right\}.
        \]
        Then $\Pr \Big(\opnorm{\, \frac{1}{m}\sum_{i\in\cS}M_i \,} >t_{\rm mat}(\delta)\Big)\le \delta$.
    \end{enumerate}
\end{corollary}

Note that both $t_{\rm scal}(\delta)$ and $t_{\rm mat}(\delta)$ are non-decreasing in $1/\delta$.

\begin{proof}
    \begin{itemize}
        \item 
        \emph{Part \ref{item:scalar-invert}:} 
        Apply \eqref{eq:bernstein-srswor} to $\{ a_i \}_{i=1}^n$ with the choice $t=t_{\rm scal}(\delta)$ and $\eta = \delta/4$. 
        Observe that 
        \begin{align*}
            \Pr \left( \frac{1}{m}\sum_{i\in\cS} a_i - \bar a \geq  t \right)
                &\leq  \exp \left(
                    -\frac{m\,t^2 /  2}{ \gamma_{\eta}^2 + \frac{2}{3}(B-A)\,t}\right) +  \eta\\
                &\leq \exp \left(
                    -\frac{m\,t^2 /  2}{ \Gamma^2(\delta) + \frac{2}{3}(B-A)\,t}\right) +  \frac{\delta}{4}\\
                &\leq \frac{\delta}{2},
                    &
                    \because \frac{m\,t_{\rm scal}(\delta)^2 /  2}{ \Gamma^2(\delta) + \frac{2}{3}(B-A)\,t_{\rm scal}(\delta)} \geq \sfL(\delta).
        \end{align*}
        Repeating the same procedure to $\{ - a_i \}_{i=1}^n$ contributes another $\delta/2$. 
        Union bound gives the claim. 

        \item 
        \emph{Part \ref{item:matrix-invert}:} 
        Choose $t=t_{\rm mat}(\delta)$ so that $\min\{t^2/(4v),\,t/(2c)\}\ge \sfL_p(\delta)/m$ in \eqref{eq:matrix-bernstein-srswor}.
    \end{itemize}
\end{proof}

\subsection{High‑probability envelopes}\label{sec:envelopes}

Recall $\cS\subset[n]$ is a uniformly random subset of size $m$. 
For fixed vectors $\{x_i\}_{i=1}^n\subset\RR^p$ and $r=(r_1,\dots,r_n)^\top\in\RR^n$, write
\[
    \bar x_{\cS}\coloneqq \frac{1}{m}\sum_{i\in\cS}x_i,\qquad
    \Sigma_{\cS}\coloneqq \frac{1}{m}\sum_{i\in\cS}(x_i-\bar x_{\cS})(x_i-\bar x_{\cS})^\top,\qquad
    u_{\cS}(x;r)\coloneqq \frac{1}{m}\sum_{i\in\cS}(x_i-\bar x_{\cS})\,r_i.
\]
Also, define
\[
    B_x\coloneqq \max_{1\le i\le n}\|x_i\|,\qquad
    B_r\coloneqq \max_{1\le i\le n}|r_i|,\qquad
    \bar r\coloneqq \frac{1}{n}\sum_{i=1}^n r_i.
\]
Lastly, introduce shorthand notation for the population averages and variances (all deterministic):
\[
    \bar q\coloneqq \frac{1}{n}\sum_{i=1}^n\|x_i\|^2,\quad
    \sigma_q^2\coloneqq \frac{1}{n}\sum_{i=1}^n(\|x_i\|^2-\bar q)^2,\quad
    \bar r_2\coloneqq \frac{1}{n}\sum_{i=1}^n r_i^2,\quad
    \sigma_{r^2}^2\coloneqq \frac{1}{n}\sum_{i=1}^n(r_i^2-\bar r_2)^2.
\]
Note that $\opnorm{x_i x_i^\top}=\|x_i\|^2$ and $\opnorm{n^{-1}\sum_j x_j x_j^\top}\le \bar q$, which we will use later in Lemma \ref{lem:sigma-envelope}.

\begin{remark}
    In the main text $\{x_1, \dots, x_n\}$ will correspond to the covariate vectors and $r$ will be instantiated by population OLS residuals $\res^{(\arm)}$, for which $\bar r=0$ by OLS orthogonality (cf. Section~\ref{sec:reg_adjustment}).
\end{remark}

\paragraph{Three lemmas on high-probability bounds.} 
Here we present high-probability upper bounds for $\| \bar{x}_{\cS} \|$, $\| u_{\cS}(x;r) \|$, and $\opnorm{ \Sigma_{\cS} - I_p }$ in three lemmas. 
Their proofs are based on Propositions \ref{prop:bardenet-maillard}--\ref{prop:matrix_Bernstein} and presented in Appendix \ref{sec:proof_lemmas}.

\begin{lemma}[Mean vector]\label{lem:xbar-envelope}
    Fix $\delta\in(0,1)$ and set $\sfL(\delta) \coloneqq\log(2/\delta)$. Let
    \[
        t_q(\delta)\coloneqq \sqrt{\frac{2\,\Gamma_q^2(\delta)\,\sfL(\delta)}{m}}\ +\ \frac{2B_x^2}{3m}\,\sfL(\delta),
        \qquad\text{where}\qquad
        \Gamma_q^2(\delta)\coloneqq \frac{n-m+1}{n}\,\sigma_q^2\ +\ \frac{m-1}{n}\,\sigma_q\,B_x^2\,\sqrt{\frac{2\sfL(\delta)}{m}}.
    \]
    Then, with probability at least $1-\delta$,
    \begin{equation}\label{eq:xbar-explicit}
        \|\bar x_{\cS}\|\ \le\ \sqrt{\frac{1}{m}\Big(\bar q + t_q(\delta)\Big)}.
    \end{equation}
\end{lemma}

\begin{lemma}[Centered residualized mean]\label{lem:u-envelope}
    Fix $\delta\in(0,1)$ and set $\sfL'(\delta)\coloneqq\log(6/\delta)$. Define
    \[
        t_{r^2}(\delta)\coloneqq \sqrt{\frac{2\,\Gamma_{r^2}^2(\delta)\,\sfL'(\delta)}{m}}\ +\ \frac{2B_r^2}{3m}\,\sfL'(\delta),
        \qquad\text{and}\qquad
        t_{r}(\delta)\coloneqq \sqrt{\frac{2\,\Gamma_{r}^2(\delta)\,\sfL'(\delta)}{m}}\ +\ \frac{4B_r}{3m}\,\sfL'(\delta).
    \]
    where
    \begin{align*}
        \Gamma_{r^2}^2(\delta) &\coloneqq \frac{n-m+1}{n}\,\sigma_{r^2}^2\ +\ \frac{m-1}{n}\cdot \sigma_{r^2}\cdot B_r^2\,\sqrt{\frac{2\sfL'(\delta)}{m}},\\
        \Gamma_{r}^2(\delta) &\coloneqq \frac{n-m+1}{n} \cdot \frac{1}{n}\sum_{i=1}^n(r_i-\bar r)^2\ +\ \frac{m-1}{n}\cdot \bigl( 2B_r \bigr) \cdot \sqrt{\frac{1}{n}\sum_{i=1}^n(r_i-\bar r)^2} \,\sqrt{\frac{2\sfL'(\delta)}{m}}.
    \end{align*}
    Then, with probability at least $1-\delta$,
    \begin{equation}\label{eq:u-explicit}
        \|u_{\cS}(x;r)\|\ \le\ \sqrt{\frac{\bar q+t_q(\delta)}{m}}\ \sqrt{\bar r_2+t_{r^2}(\delta)}\ +\ \|\bar x_{\cS}\|\;\Big(\,|\bar r|+t_{r}(\delta)\,\Big).
    \end{equation}
\end{lemma}

\begin{lemma}[Operator deviation for $\Sigma_{\cS}$]\label{lem:sigma-envelope}
    Fix $\delta\in(0,1)$ and set $\sfL'_p(\delta) \coloneqq\log \big(4p/\delta\big)$. Write
    \[
    M_i\coloneqq x_i x_i^\top - \frac{1}{n}\sum_{j=1}^n x_j x_j^\top,
    \qquad
    v\coloneqq\ \opnorm{\, \frac{1}{n}\sum_{i=1}^n M_i^2 \,},
    \qquad
    c\coloneqq\max_{1\le i\le n} \opnorm{ M_i }\le B_x^2+\bar q.
    \]
    Then, with probability at least $1-\delta$,
    \begin{equation}\label{eq:Delta-explicit}
    \opnorm{ \Sigma_{\cS}-\tfrac{1}{n}\sum_{i=1}^n x_i x_i^\top }
    \ \le\ \max \left\{\,2\sqrt{\frac{v\,\sfL'_p(\delta)}{m}}\,,\ \frac{2c\,\sfL'_p(\delta)}{m}\right\}\ +\ \|\bar x_{\cS}\|^2.
    \end{equation}
    In particular, under the design normalization $\bone_n^\top X=0$ and $X^\top X=nI_p$ (so $n^{-1}\sum_i x_i x_i^\top = I_p$ and $\bar q=p$),
    \begin{equation}\label{eq:Delta-explicit-design}
    \opnorm{ \Sigma_{\cS}-I_p }
    \ \le\ \max \left\{\,2\sqrt{\frac{v\,\sfL'_p(\delta)}{m}}\,,\ \frac{2c\,\sfL'_p(\delta)}{m}\right\}\ +\ \|\bar x_{\cS}\|^2,
    \qquad
    c\le B_x^2+ p.
    \end{equation}
\end{lemma}

\paragraph{Specialization under the design normalization.}
When $\bone_n^\top X=0$ and $X^\top X=nI_p$ we have $\bar q=p$ and $v, c$ above are computed from $\{x_i\}$ as in Lemma~\ref{lem:sigma-envelope}. 
Hence the explicit envelopes from Lemmas~\ref{lem:xbar-envelope}–\ref{lem:sigma-envelope} give admissible choices for the $(\alpha, \beta, \gamma)$ in Assumption~\ref{assumption:rates}:
\begin{align*}
    \alpha_{n,p}(\delta)\ &\coloneqq\ \max \left\{\,2\sqrt{\frac{v\,\log(4p/\delta)}{m}}\,,\ \frac{2c\,\log(4p/\delta)}{m}\right\}\ +\ \gamma_{n,p}(\delta)^2  \quad\text{with }c\le B_x^2+p,\\
    \beta_{n,p}(\delta)\ &\coloneqq\ \sqrt{\frac{p+t_q(\delta)}{m}}\ \sqrt{\bar r_2+t_{r^2}(\delta)}\ +\ \gamma_{n,p}(\delta)\,t_r(\delta),\\
    \gamma_{n,p}(\delta)\ &\coloneqq\ \sqrt{\frac{p+t_q(\delta)}{m}}.
\end{align*}

Furthermore, under diffuse leverage regimes per Assumption \ref{assump:leverage}, (e.g., $B_x^2\lesssim n\kappa_n$ with $\kappa_n\lesssim p/n$), the leading terms (up to logarithms) scale as
\[
    \alpha_{n,p}(\delta)\ \asymp\ \sqrt{\frac{p}{m}}\;+\;\gamma_{n,p}(\delta)^2 \asymp \sqrt{\frac{p}{m}}, \qquad
    \beta_{n,p}(\delta)\ \asymp\ \sqrt{\frac{p}{m}}\ \sqrt{\bar r_2},\qquad
    \gamma_{n,p}(\delta)\ \asymp\ \sqrt{\frac{p}{m}},
\]
matching the canonical rates summarized in Remark~\ref{rem:canonical_rates} (given $m \asymp n$).

\subsection{Proofs of the lemmas in Section \ref{sec:envelopes}}\label{sec:proof_lemmas}
\subsubsection{Proof of Lemma~\ref{lem:xbar-envelope}}
\begin{proof}[Proof of Lemma~\ref{lem:xbar-envelope}]
By Cauchy--Schwarz (or Jensen),
\[
\Big\|\frac{1}{m}\sum_{i\in\cS}x_i\Big\|^2
= \frac{1}{m^2}\Big\|\sum_{i\in\cS}x_i\Big\|^2
\ \le\ \frac{1}{m^2}\, m \sum_{i\in\cS}\|x_i\|^2
= \frac{1}{m}\,\Big(\frac{1}{m}\sum_{i\in\cS}\|x_i\|^2\Big).
\]
Thus
\[
\|\bar x_{\cS}\|
\ \le\ \sqrt{\frac{1}{m}\left(\bar q + \Big[\frac{1}{m}\sum_{i\in\cS}\|x_i\|^2 - \bar q\Big]\right)}.
\]
Apply Corollary~\ref{cor:delta-level}\ref{item:scalar-invert} to the scalars $a_i=\|x_i\|^2$ with $A=0$ and $B=B_x^2$ to obtain, with probability at least $1-\delta$,
\(
\frac{1}{m}\sum_{i\in\cS}\|x_i\|^2 \le \bar q + t_q(\delta).
\)
Substitute this into the previous display to get \eqref{eq:xbar-explicit}.
\end{proof}

\subsubsection{Proof of Lemma~\ref{lem:u-envelope}}
\begin{proof}[Proof of Lemma~\ref{lem:u-envelope}]
Decompose
\[
u_{\cS}(x;r)
= \frac{1}{m}\sum_{i\in\cS}x_i r_i - \bar x_{\cS}\,\bar r_{\cS}
=:\ U_1 - \bar x_{\cS}\,\bar r_{\cS},
\qquad
\bar r_{\cS}\coloneqq \frac{1}{m}\sum_{i\in\cS}r_i.
\]

\emph{Step 1. First term.}
By Cauchy--Schwarz over the sample,
\[
\|U_1\|
\ \le\ \frac{1}{m}\sqrt{\sum_{i\in\cS}\|x_i\|^2}\ \sqrt{\sum_{i\in\cS}r_i^2}
= \sqrt{\frac{1}{m}\Big(\frac{1}{m}\sum_{i\in\cS}\|x_i\|^2\Big)}
\ \sqrt{\frac{1}{m}\sum_{i\in\cS}r_i^2}.
\]
Apply Corollary~\ref{cor:delta-level}\ref{item:scalar-invert} to $\{a_i=\|x_i\|^2\}$ (with $A=0$, $B=B_x^2$) and to $\{a_i=r_i^2\}$ (with $A=0$, $B=B_r^2$) at level $\delta/3$. Then, with probability at least $1-\tfrac{2\delta}{3}$,
\[
\frac{1}{m}\sum_{i\in\cS}\|x_i\|^2 \le \bar q + t_q(\delta/3),
\qquad
\frac{1}{m}\sum_{i\in\cS}r_i^2 \le \bar r_2 + t_{r^2}(\delta/3),
\]
and hence
\[
\|U_1\|
\ \le\ \sqrt{\frac{\bar q+t_q(\delta/3)}{m}}\ \sqrt{\bar r_2+t_{r^2}(\delta/3)}.
\]
Both bounds follow from Corollary~\ref{cor:delta-level}\ref{item:scalar-invert} applied at level $\delta/3$.

\emph{Step 2. Second term.}
By the triangle inequality,
\[
\|\bar x_{\cS}\|\ |\bar r_{\cS}|
\ \le\ \|\bar x_{\cS}\|\ \Big(|\bar r| + |\bar r_{\cS}-\bar r|\Big).
\]
Apply Corollary~\ref{cor:delta-level}\ref{item:scalar-invert} to $\{a_i=r_i\}$ and $\{-a_i\}$
(two-sided) with range $B-A\le 2B_r$ at level $\delta/3$:
\(
|\bar r_{\cS}-\bar r| \le t_r(\delta/3)
\)
with probability at least $1-\delta/3$.
In addition, Lemma~\ref{lem:xbar-envelope} (at level $\delta/3$) controls $\|\bar x_{\cS}\|$.

\emph{Step 3. Combine.}
By a union bound over the three events (for $\|x\|^2$, $r^2$, and $r$), with probability at least $1-\delta$ we have
\[
\|u_{\cS}(x;r)\|
\ \le\ \sqrt{\frac{\bar q+t_q(\delta/3)}{m}}\ \sqrt{\bar r_2+t_{r^2}(\delta/3)}
\ +\ \|\bar x_{\cS}\|\ \Big(|\bar r| + t_r(\delta/3)\Big).
\]
Finally, since $t_q(\cdot)$, $t_{r^2}(\cdot)$, and $t_r(\cdot)$ are nonincreasing in their confidence level (they are non-decreasing in $1/\delta$ through $\log(1/\delta)$), we may replace $\delta/3$ by $\delta$ to obtain a slightly looser but cleaner bound. This yields \eqref{eq:u-explicit}.
\end{proof}

\subsubsection{Proof of Lemma~\ref{lem:sigma-envelope}}
\begin{proof}[Proof of Lemma~\ref{lem:sigma-envelope}]
Write the centered/uncentered decomposition
\[
\Sigma_{\cS}-\frac{1}{n}\sum_{i=1}^n x_i x_i^\top
=
\underbrace{\frac{1}{m}\sum_{i\in\cS}\Big(x_i x_i^\top-\tfrac{1}{n}\sum_{j=1}^n x_j x_j^\top\Big)}_{\eqqcolon\,S_{\cS}}
\ -\ \bar x_{\cS}\bar x_{\cS}^\top.
\]
Hence
\(
\opnorm{ \Sigma_{\cS}-\frac{1}{n}\sum_i x_i x_i^\top }
\le \opnorm{ S_{\cS} } + \|\bar x_{\cS}\|^2.
\)
Apply Corollary~\ref{cor:delta-level}\ref{item:matrix-invert} to $M_i\coloneqq x_i x_i^\top-\frac{1}{n}\sum_j x_j x_j^\top$ at level $\delta/2$, which yields
\[
\opnorm{ S_{\cS} }
\ \le\ \max \left\{\,2\sqrt{\frac{v\,\log(4p/\delta)}{m}}\,,\ \frac{2c\,\log(4p/\delta)}{m}\right\}
\quad\text{with probability at least }1-\delta/2.
\]
Combine this with Lemma~\ref{lem:xbar-envelope} at level $\delta/2$ and take a union bound to obtain \eqref{eq:Delta-explicit}.
Under the design normalization $\bone_n^\top X=0$ and $X^\top X=nI_p$, $n^{-1}\sum_i x_i x_i^\top=I_p$ and \eqref{eq:Delta-explicit-design} follows immediately.
\end{proof}

%% file: contents/Appx02_additional_examples.tex
\section{Additional discussion on examples}\label{sec:additional_examples}

\subsection{More details on Example \ref{example:exp_k.0} in Section \ref{sec:neumann_estimation}}\label{sec:example_design_k.0}

Here we prove equation \eqref{eqn:Neumann_exp_k.0}, which states that under Assumption \ref{assump:covariates} ($\bone_n^\top X = 0$ and $X^\top X = n I_p$), 
\[
    \xi^{[0]}_i(m; X)
        = \frac{(m-1)(n-m)n }{ m^2(n-1)(n-2)} \, \Bigl( \|x_i\|_2^2 - p \Bigr).
\]

\begin{proof}[Proof of \eqref{eqn:Neumann_exp_k.0}]
    Let $\bbE_i[ \cdot ] \equiv \bbE\bigl[ \cdot  \bigm| i \in \cS \bigr]$ (recall $\cS \subset [n]$ with $|\cS| = m$). 
    \begin{align*}
        \xi^{[0]}_i(m; X) 
            &= \bbE_i\bigl[ \bar{x}_{\cS}^{\top} (x_i - \bar{x}_{\cS}) \bigr] \\
            &=  \bbE_i\bigl[ \bar{x}_{\cS} \bigr]^{\top} x_i
                -  \bbE_i\bigl[ \bar{x}_{\cS}^{\top} \bar{x}_{\cS} \bigr].
    \end{align*}
    
    First of all,
    \begin{align*}
        \bbE_i\bigl[ \bar{x}_{\cS} \bigr]
            &= \frac{1}{m} \sum_{j=1}^n \bbE_i\bigl[ x_j \cdot \Ind\{j \in \cS\} \bigr]\\
            &= \frac{1}{m} \sum_{j=1}^n \Pr\bigl( j \in \cS | i \in \cS \bigr) \cdot  x_j\\
            &= \frac{1}{m} \biggl( x_i + \frac{m-1}{n-1}\sum_{j \in [n]\setminus\{i\}} x_j \biggr)\\
            &= \frac{n-m}{m (n-1)} x_i
    \end{align*}
    because $\bone_n^{\top} X = \sum_{j=1}^n x_j = 0$ implies $\sum_{j\neq i}x_j=-x_i$.
    
    Secondly,
    \begin{align*}
        \bbE_i\bigl[ \bar{x}_{\cS}^{\top} \bar{x}_{\cS} \bigr]
            &= \frac{1}{m^2} \sum_{j, k=1}^n \bbE_i\bigl[ \, x_j^{\top} x_{k} \cdot \Ind\{j, k \in \cS\} \, \bigr]\\
            &= \frac{1}{m^2} \sum_{j, k=1}^n \Pr\bigl( j, k \in \cS | i \in \cS \bigr) \cdot x_j^{\top} x_{k}.
    \end{align*}
    Partition the possible configurations of $(j, k)$ pair into
    \begin{itemize}
        \item 
        $j = k = i$: 
        $\Pr\bigl( j, k \in \cS | i \in \cS \bigr) = 1$.
        \item 
        $j = i$ and $k \neq i$, or $j \neq i$ and $k = i$:
        $\Pr\bigl( j, k \in \cS | i \in \cS \bigr) = \frac{m-1}{n-1}$.
        \item 
        $j, k \neq i$ and $j = k$:
        $\Pr\bigl( j, k \in \cS | i \in \cS \bigr) = \frac{m-1}{n-1}$.
        \item 
        $j, k \neq i$ and $j \neq k$:
        $\Pr\bigl( j, k \in \cS | i \in \cS \bigr) = \frac{(m-1)(m-2)}{(n-1)(n-2)}$.
    \end{itemize}
    With $\pi_1 \coloneqq \tfrac{m-1}{n-1}$ and $\pi_2 \coloneqq \tfrac{(m-1)(m-2)}{(n-1)(n-2)}$,
    \begin{align*}
        &\sum_{j, k=1}^n \Pr\bigl( j, k \in \cS | i \in \cS \bigr) \cdot x_j^{\top} x_{k}\\
            &\qquad= x_i^{\top} x_i
                + \pi_1 \cdot \Bigl( x_i^{\top} \sum_{k \neq i} x_{k} + \sum_{j \neq i} x_{j}^{\top} x_i \Bigr)
                + \pi_1 \cdot \sum_{j \neq i} x_j^{\top} x_j
                + \pi_2 \cdot \underbrace{\sum_{j \in [n] \setminus \{i\}} \sum_{k \in [n] \setminus \{i,j\} } x_j^{\top} x_{k}}_{= - \sum_{j \in [n] \setminus\{i\}} (x_j^{\top} x_i + x_j^{\top} x_j)}\\
            &\qquad= x_i^{\top} x_i + \pi_1 \cdot  x_i^{\top} \underbrace{\sum_{k \neq i} x_{k}}_{=-x_i} 
                + \bigl( \pi_1 - \pi_2 \bigr)\cdot \underbrace{\sum_{j \neq i} x_{j}^{\top}}_{=-x_i^{\top}} x_i
                + \bigl( \pi_1 - \pi_2 \bigr)\cdot  \underbrace{\sum_{j \neq i} \| x_j\|^2}_{=n p - \| x_i \|^2 }\\
            &\qquad= \bigl( \pi_1 - \pi_2 \bigr)\cdot \underbrace{\sum_{j=1}^n \|x_j\|_2^2}_{=\tr(XX^{\top}) = np} 
                + \bigl( 1 - 3\pi_1 + 2\pi_2 \bigr) \|x_i\|^2,
    \end{align*}
    where $\pi_1-\pi_2=\frac{(m-1)(n-m)}{(n-1)(n-2)}$ and $1-3\pi_1+2\pi_2=\frac{(n-m)(n-2m)}{(n-1)(n-2)}$.
    
    All in all, 
    \begin{align*}
        \xi^{[0]}_i(m; X)
            &= \frac{(n-m)(m-1)}{m^2} \cdot \frac{ n }{ (n-1)(n-2)} \Bigl( \|x_i\|_2^2 - p \Bigr).
    \end{align*}
\end{proof}

\subsection{Illustration of words, partitions and design monomials in Section \ref{sec:design_expectation}}\label{sec:example_word_partition}

In this section, we provide an example that illustrates the combinatorial notions introduced in Section \ref{sec:design_expectation}.

\begin{example}[Word, partition, and monomial: $d=1$]\label{example:word_partition}
    Suppose that $d=1$, and let $\theta = (U,i)$. 
    Then $L_{\theta} = 2$, $\cV_{\theta} = \{1,2\}$, and $\cE_{\theta} = \{ (1,2) \}$. 
    There are two possible partitions:
    \[
        \pi_1 = \bigl\{ \{1\}, \{2\} \bigr\} ~(|\pi_1|=2),
        \qquad
        \pi_2 = \bigl\{ \{1, 2\} \bigr\} ~(|\pi_2|=1).
    \]

    For $\pi_1$, let $\omega = (\omega_1, \omega_2)$ with $\omega_1 \neq \omega_2$. 
    Then the corresponding design monomial
    \[
        \varphi_{\omega}(X; i) = ( x_{\omega_1}^{\top} x_{\omega_2} ) ( x_{\omega_2}^{\top} x_i ).
    \]
    For $\pi_2$, there is one block, $\omega = (\omega_1)$, and
    \[
        \varphi_{\omega}(X; i) = ( x_{\omega_1}^{\top} x_{\omega_1} ) ( x_{\omega_1}^{\top} x_i ).
    \]
\end{example}

The rest of this section elaborates on the combinatorial notions considering degree $d=1$, including Example~\ref{example:word_partition}, in full detail.

\paragraph{Recollecting definitions.}
We begin by recalling the combinatorial ingredients from Section~\ref{sec:design_expectation}; see Definitions~\ref{defn:word}--\ref{defn:allocation} and Theorem~\ref{thm:xi_formula}. 
Recall that the SRSWOR weights (Lemma \ref{lem:SRSWOR_inclusion}) are
\[
    \psi_0(\pi)=\frac{(m-1)_{|\pi|}}{(n-1)_{|\pi|}},
    \qquad
    \psi_1(\pi)=\frac{(m-1)_{|\pi|-1}}{(n-1)_{|\pi|-1}},
\]
where $(a)_\ell = a(a-1)\cdots(a-\ell+1)$ denotes the falling factorial. 
The sign and multiplicity of a word $\theta=(\phi,\chi)$ are
\[
    \operatorname{sgn}_\theta=(-1)^{u(\phi)+\Ind\{\chi=S\}},
    \qquad
    L_\theta=1+u(\phi)+2v(\phi)+\Ind\{\chi=S\},
\]
where $u(\phi) \coloneqq |\{a:\phi_a=U\}|$ and $v(\phi) \coloneqq |\{a:\phi_a=V\}|$ (cf. Definition~\ref{defn:word}). 
The design monomial (Definition~\ref{defn:allocation}) for a word $\theta$ and partition $\pi$ is
\[
    \varphi_\omega(X;i)=\Bigg(\prod_{(s,t)\in \cE_\theta}  \langle x_{\omega(\iota_\pi(s))},\,x_{\omega(\iota_\pi(t))}\rangle\Bigg)\cdot
    \big\langle x_{\omega(\iota_\pi(L_\theta))},\,x_i\big\rangle^{\Ind\{\chi=i\}},
\]
with $\iota_\pi:\cV_\theta\to[|\pi|]$ the block index map (Definition~\ref{defn:partition}).
Summing over Class~0/1 allocations (avoid/contain $i$ in exactly one block) gives $\Phi_{\pi,0}(X;i)$ and $\Phi_{\pi,1}(X;i)$ (Eq.~\eqref{eqn:diagram_aggregates}), which enter Theorem~\ref{thm:xi_formula}:
\[
    \xi^{[d]}_i(m;X)=  \sum_{\theta\in\{I,U,V\}^d\times\{i,S\}}  
    \frac{\operatorname{sgn}_\theta}{m^{L_\theta}}
    \sum_{\pi\in\Part(\cV_\theta)}\Big\{ \psi_0(\pi)\Phi_{\pi,0}(X;i)+\psi_1(\pi)\Phi_{\pi,1}(X;i)\Big\}.
\]
All of the notions above are specialized below to $d=1$, where $\varphi\in\{I,U,V\}$ and $\chi\in\{i,S\}$.%

\paragraph{Degree $d=1$: Enumeration of words, signs and multiplicities.}
There are six possible words when $d=1$:
\[
    (I,i),\ (I,S),\ (U,i),\ (U,S),\ (V,i),\ (V,S).
\]
Their $(\operatorname{sgn}_\theta,L_\theta)$ and position graphs $\cG_\theta=(\cV_\theta,\cE_\theta)$ follow from Definition~\ref{defn:position_graph}:
\begin{itemize}
    \item $\theta=(I,i)$: $\operatorname{sgn}_\theta=+1$, $L_\theta=1$, $\cV_\theta=\{1\}$, $\cE_\theta=\varnothing$.
    \item $\theta=(I,S)$: $\operatorname{sgn}_\theta=-1$, $L_\theta=2$, $\cV_\theta=\{1,2\}$, $\cE_\theta=\{(1,2)\}$.
    \item $\theta=(U,i)$: $\operatorname{sgn}_\theta=-1$, $L_\theta=2$, $\cV_\theta=\{1,2\}$, $\cE_\theta=\{(1,2)\}$.
    \item $\theta=(U,S)$: $\operatorname{sgn}_\theta=+1$, $L_\theta=3$, $\cV_\theta=\{1,2,3\}$, $\cE_\theta=\{(1,2),(2,3)\}$.
    \item $\theta=(V,i)$: $\operatorname{sgn}_\theta=+1$, $L_\theta=3$, $\cV_\theta=\{1,2,3\}$, $\cE_\theta=\{(1,2)\}$.
    \item $\theta=(V,S)$: $\operatorname{sgn}_\theta=-1$, $L_\theta=4$, $\cV_\theta=\{1,2,3,4\}$, $\cE_\theta=\{(1,2),(3,4)\}$.
\end{itemize}
For $|\cV_\theta|=1,2,3,4$ the number of set partitions is $1,2,5,15$ (Bell numbers), respectively. 
Now we illustrate a few example cases.

\begin{enumerate}
    \item 
    \textbf{Worked case A: $\theta=(U,i)$ (the case for Example \ref{example:word_partition}).}
    Here $L_\theta=2$, $\cE_\theta=\{(1,2)\}$, so $\cV_\theta=\{1,2\}$ has two partitions:
    \[
        \pi_1=\big\{\{1\},\{2\}\big\}\quad(\,|\pi_1|=2\,),\qquad
        \pi_2=\big\{\{1,2\}\big\}\quad(\,|\pi_2|=1\,).
    \]
    Let $\omega=(\omega_1,\omega_2)$ be injective for $\pi_1$ and $\omega=(\omega_1)$ for $\pi_2$.
    
    \emph{(a) Design monomials.}  
    Since $\chi=i$, there is a terminal factor $\langle x_{\omega(\iota_\pi(L_\theta))},x_i\rangle$ with $L_\theta=2$ and $\iota_\pi(2)$ the block of position $2$:
    \begin{align*}
        \pi_1=\{\{1\},\{2\}\}:&\quad
        \varphi_\omega(X;i)=\underbrace{\langle x_{\omega_1},x_{\omega_2}\rangle}_{(1,2)\text{ cross-edge}}\cdot
        \underbrace{\langle x_{\omega_2},x_i\rangle}_{\text{terminal}};\\
        \pi_2=\{\{1,2\}\}:&\quad
        \varphi_\omega(X;i)=\underbrace{\langle x_{\omega_1},x_{\omega_1}\rangle}_{(1,2)\text{ self-edge}=\|x_{\omega_1}\|_2^2}\cdot
        \underbrace{\langle x_{\omega_1},x_i\rangle}_{\text{terminal}}.
    \end{align*}
    Observe that these correspond to the design monomials shown in Example~\ref{example:word_partition}.
    
    \emph{(b) Class aggregates.}  
    Class~0 excludes $i$ from all blocks; Class~1 includes $i$ in exactly one block.
    \begin{align*}
        \Phi_{\pi_1,0}(X;i)
            &=\sum_{\substack{\omega_1\neq \omega_2\\ \omega_1\neq i,\ \omega_2\neq i}}
            \big(x_{\omega_1}^\top x_{\omega_2}\big)\big(x_{\omega_2}^\top x_i\big),\\
        \Phi_{\pi_1,1}(X;i)
            &=\sum_{\substack{\omega_2=i\\ \omega_1\neq i}}
            \big(x_{\omega_1}^\top x_i\big)\|x_i\|_2^2
            \;+\;
            \sum_{\substack{\omega_1=i\\ \omega_2\neq i}}
            \big(x_i^\top x_{\omega_2}\big)^2,\\[3pt]
        \Phi_{\pi_2,0}(X;i)
            &=\sum_{\omega_1\neq i}\|x_{\omega_1}\|_2^2\big(x_{\omega_1}^\top x_i\big),\\
        \Phi_{\pi_2,1}(X;i)
            &=\|x_i\|_2^4.
    \end{align*}
    
    \emph{(c) Weights and prefactor.}  
    Here $|\pi_1|=2$, $|\pi_2|=1$, $\operatorname{sgn}_\theta=-1$, $L_\theta=2$.  Hence the contribution of $\theta=(U,i)$ to $\xi^{[1]}_i(m;X)$ is
    \[
        \xi^{[1]}_i(m;X;\theta{=}(U,i))
            =
            -\frac{1}{m^2} \left[
            \frac{(m-1)_2}{(n-1)_2}\,\Phi_{\pi_1,0}
            +\frac{(m-1)_1}{(n-1)_1}\,\Phi_{\pi_1,1}
            +\frac{(m-1)_1}{(n-1)_1}\,\Phi_{\pi_2,0}
            +1\cdot\Phi_{\pi_2,1}
            \right] ,
    \]
    where the last weight equals $1$ since $(m-1)_0/(n-1)_0=1$.

    \item 
    \textbf{Worked case B: $\theta=(V,i)$ (a degree–1 ``$V$'' word).} 
    Here $\operatorname{sgn}_\theta=+1$, $L_\theta=3$, $\cE_\theta=\{(1,2)\}$.
    Thus $\cV_\theta=\{1,2,3\}$ and we illustrate three representative partitions:
    \[
        \pi_a=\big\{\{1,2\},\{3\}\big\},\qquad
        \pi_b=\big\{\{1\},\{2,3\}\big\},\qquad
        \pi_c=\big\{\{1\},\{2\},\{3\}\big\}.
    \]

    \emph{(a) Design monomials.}  
    With injective $\omega$, the monomials are:
    \begin{align*}
        \pi_a:\quad & \varphi_\omega(X;i)=\underbrace{\langle x_{\omega_1},x_{\omega_1}\rangle}_{(1,2)\text{ self-edge}=\|x_{\omega_1}\|_2^2}\cdot \underbrace{\langle x_{\omega_2},x_i\rangle}_{\text{terminal (pos.~3)}},\\
        \pi_b:\quad & \varphi_\omega(X;i)=\underbrace{\langle x_{\omega_1},x_{\omega_2}\rangle}_{(1,2)\text{ cross-edge}}\cdot \underbrace{\langle x_{\omega_2},x_i\rangle}_{\text{terminal (pos.~3)}},\\
        \pi_c:\quad & \varphi_\omega(X;i)=\underbrace{\langle x_{\omega_1},x_{\omega_2}\rangle}_{(1,2)\text{ cross-edge}}\cdot \underbrace{\langle x_{\omega_3},x_i\rangle}_{\text{terminal (pos.~3)}}.
    \end{align*}

    \emph{(b) Class aggregates.}  
    The corresponding Class~0/1 aggregates read
    \begin{align*}
        \Phi_{\pi_a,0}
            &=\sum_{\substack{\omega_1\neq\omega_2\\ \omega_1,\omega_2\neq i}}
            \|x_{\omega_1}\|_2^2\,\big(x_{\omega_2}^\top x_i\big),\\
        \Phi_{\pi_a,1}
            &=\sum_{\substack{\omega_2\neq i}}
            \|x_i\|_2^2\,\big(x_{\omega_2}^\top x_i\big)
            +\sum_{\substack{\omega_1\neq i}}
            \|x_{\omega_1}\|_2^2\,\|x_i\|_2^2,\\
        \Phi_{\pi_b,0}
            &=\sum_{\substack{\omega_1\neq\omega_2\\ \omega_1,\omega_2\neq i}}
            \big(x_{\omega_1}^\top x_{\omega_2}\big)\big(x_{\omega_2}^\top x_i\big),\\
        \Phi_{\pi_b,1}
            &=\sum_{\substack{\omega_2\neq i}}
            \big(x_i^\top x_{\omega_2}\big)^2
            +\sum_{\substack{\omega_1\neq i}}
            \big(x_{\omega_1}^\top x_i\big)\|x_i\|_2^2,\\
        \Phi_{\pi_c,0}
            &=  \sum_{\substack{\omega_1,\omega_2,\omega_3\ \text{all distinct}\\ \omega_1,\omega_2,\omega_3\neq i}}
            \big(x_{\omega_1}^\top x_{\omega_2}\big)\big(x_{\omega_3}^\top x_i\big),\\
        \Phi_{\pi_c,1}
            &=\sum_{\substack{\omega_2,\omega_3\neq i\\ \omega_2\neq \omega_3}}
            \big(x_i^\top x_{\omega_2}\big)\big(x_{\omega_3}^\top x_i\big)
            +\sum_{\substack{\omega_1,\omega_3\neq i\\ \omega_1\neq \omega_3}}
            \big(x_{\omega_1}^\top x_i\big)\big(x_{\omega_3}^\top x_i\big)
            +\sum_{\substack{\omega_1,\omega_2\neq i\\ \omega_1\neq \omega_2}}
            \big(x_{\omega_1}^\top x_{\omega_2}\big)\|x_i\|_2^2.
    \end{align*}

    \emph{(c) Weights and prefactor.}
    The weights are $\psi_0(\pi_a)=\psi_0(\pi_b)=(m - 1)_2/(n - 1)_2$, $\psi_1(\pi_a)=\psi_1(\pi_b)=(m - 1)_1/(n - 1)_1$, and
    $\psi_0(\pi_c)=(m - 1)_3/(n - 1)_3$, $\psi_1(\pi_c)=(m - 1)_2/(n - 1)_2$.  With $\operatorname{sgn}_\theta=+1$ and $L_\theta=3$, the total contribution of $\theta=(V,i)$ is
    \[
        \xi^{[1]}_i(m;X;\theta{=}(V,i))
            =
            \frac{1}{m^3}\sum_{\pi\in\{\pi_a,\pi_b,\pi_c\}}
            \Big\{\psi_0(\pi)\Phi_{\pi,0}+\psi_1(\pi)\Phi_{\pi,1}\Big\}.
    \]

    \item 
    \textbf{Worked case C: $\chi=S$.}
    When $\chi=S$ there is no terminal factor (Definition~\ref{defn:allocation}), but the position graph contains an extra edge $(s,s{+}1)$ at the end by construction (Definition~\ref{defn:position_graph}). 
    Thus all monomials are pure products of inner products along edges, and the sign flips by $(-1)$ relative to the corresponding $\chi=i$ case. 
    For example, with $\theta=(U,S)$ we have $L_\theta=3$ and $\cE_\theta=\{(1,2),(2,3)\}$; if $\pi=\{\{1,2,3\}\}$,
    \[
        \varphi_\omega(X;i)\;=\;\langle x_{\omega_1},x_{\omega_1}\rangle\cdot \langle x_{\omega_1},x_{\omega_1}\rangle
            \;=\;\|x_{\omega_1}\|_2^4,
    \]
    and the weight pair is $(\psi_0,\psi_1)=\big((m - 1)_1/(n - 1)_1,\ (m-1)_0/(n-1)_0\big)=\big((m - 1)/(n - 1),\,1\big)$ since $|\pi|=1$. 
    All other partitions are handled analogously by classifying each edge as self vs.\ cross at the partition level.
\end{enumerate}

\paragraph{Assembling the degree–1 coefficient $\xi^{[1]}_i(m;X)$.}
Summing the contributions of the six words with their $(\operatorname{sgn}_\theta,L_\theta)$ and $(\psi_0,\psi_1)$ yields the degree–1 design coefficient
\[
    \xi^{[1]}_i(m;X)
        =
        \sum_{\theta\in\{(I,i),(I,S),(U,i),(U,S),(V,i),(V,S)\}}
        \frac{\operatorname{sgn}_\theta}{m^{L_\theta}}
        \sum_{\pi\in\Part(\cV_\theta)}\Big\{\psi_0(\pi)\Phi_{\pi,0}(X;i)+\psi_1(\pi)\Phi_{\pi,1}(X;i)\Big\}.
\]
In applications we evaluate these sums efficiently using the Folding–M\"{o}bius procedure of Appendix~\ref{sec:computation_expectation}, which avoids explicit enumeration over injective allocations $\omega$.
\medskip

\begin{remark}[Sanity check under design normalization]
    Under $\bone_n^\top X=0$ and $X^\top X=nI_p$, self–edges contribute $\|x_j\|_2^2=n h_j$ where $h_j$ is the leverage; cross–edges are entries of the Gram $G=XX^\top$.  
    Therefore, degree–1 terms combine quadratic products of $G$ (or of the leverage vector $h$) with linear terms in $G(:,i)$, cf. Section~\ref{sec:design_expectation}.
\end{remark}

%% file: contents/Appx03_algorithm.tex
\section{Algorithms to compute the design expectations in Theorem \ref{thm:xi_formula}}\label{sec:computation_expectation}

This section describes implementable procedures to compute the Neumann weights $\xi^{[d]}_i(m;X)$ in Theorem~\ref{thm:xi_formula} and the expectations $\tfrac{1}{n}\langle \xi^{[d]}, r^{(\arm)}\rangle=\mathbb{E}[\sfR^{[d]}_{\arm}]$. 
We proceed in two steps: we first evaluate \emph{all‑assignment} aggregates at the Gram level via linear contractions; then recover the injective sums required by Theorem~\ref{thm:xi_formula} using M\"{o}bius inversion on the partition lattice.
Section \ref{sec:auxiliary_definitions} introduces the necessary auxiliary definitions to formally describe the procedure. 
Section \ref{sec:folding_algorithms} presents two algorithms: a vector-form \emph{Folding–M\"{o}bius} algorithm for $\xi^{[d]}$ (Algorithm \ref{alg:folding_mobius}) and a \emph{scalar} variant that computes $\mathbb{E}[\sfR^{[d]}_{\arm}]$ without forming $\xi^{[d]}$ explicitly (Algorithm \ref{alg:scalar_folding_mobius}).


\subsection{Auxiliary definitions}\label{sec:auxiliary_definitions}
Recall the essential definitions and notation from Section \ref{sec:design_expectation}: 
words $\theta=(\phi,\chi)$, position graph $\cG_{\theta} = (\cV_{\theta}, \cE_{\theta})$, partitions
$\pi=\{P_1,\dots,P_\ell\}\in\Part(\cV_\theta)$ with index map $\iota_\pi:\cV_\theta\to[\ell]$, and the design monomials $\varphi_{\omega}(X; i)$. 
Directly summing $\varphi_\omega$ over all \emph{injective} $\omega$ directly requires a combinatorial enumeration, which can be computationally prohibitive for large $n$ and $d$. 
Allowing multi-allocations collapses this sum into a sequence of simpler \emph{Gram-level} linear operations, such as matrix–vector products and element-wise (Hadamard) products, which can be computed much more efficiently. 
However, relaxing the injectivity leads to multiple counting of certain multi-index patterns; these over-counted multi-indices can be corrected back via a M\"{o}bius inversion (inclusion–exclusion principle) on the block-partition lattice and recover the injective sums. 
In this section, we introduce three auxiliary objects that make the algorithm precise and efficient: (i) a quotient position multigraph encoding block‑level edge multiplicities, (ii) component‑restricted all‑assignment aggregates, and (iii) a component-wise folding operator that evaluates these aggregates via Gram‑level linear operations.

\subsubsection{Multi-allocations, quotient multigraphs, and component-restricted aggregates}

Directly summing design monomials $\varphi_\omega$ over injective block assignments requires a combinatorial enumeration, which is prohibitive for large $n$ and $d$. 
We (i) relax injectivity to form \emph{all-assignment} sums that can be evaluated by linear contractions in the Gram $G=XX^\top$, and (ii) record the block-level inner-product pattern via a quotient multigraph. 
The result factorizes across connected components, enabling component-wise folds, however, this relaxation over‑counts multi‑index patterns. 
Ultimately, a M\"{o}bius inversion on the block‑partition lattice corrects the over‑counting and recovers the injective sums.

\paragraph{Multi-allocations.}
We extend the definition of allocations admissible to a word-partition pair by relaxing the injectivity requirement.

\begin{definition}[Multi-allocation]\label{defn:multi_allocation}
    For a word $\theta$ and a partition $\pi = \{ P_1, \dots, P_{\ell}\} \in \Part(\cV_{\theta})$, 
    a \emph{multi-allocation} admissible to $(\theta, \pi)$ is a map $\omega: [\ell] \to [n]$. 
    Denote by $\Omega^{\all}(\theta, \pi) \coloneqq [n]^{\ell}$ the set of all multi-allocations. 
\end{definition}

Note that an allocation admissible to $(\theta, \pi)$ (Definition \ref{defn:allocation}) is a multi-allocation with injectivity constraint. 
We define the aggregate of the design monomial akin to \eqref{eqn:diagram_aggregates}, but over all multi-allocations (not necessarily injective).

\begin{definition}[All-assignment aggregate]\label{defn:all_assignments}
    Let $X \in \RR^{n \times p}$. 
    For a word $\theta$ and a partition $\pi = \{ P_1, \dots, P_\ell\} \in \Part(\cV_{\theta})$, the $i$-th \emph{design monomial} induced by $X$ with respect to $\omega \in \Omega^{\all}(\theta, \pi)$ is    
    \begin{equation}\label{eqn:design_monomial_all}
        \varphi_{\omega}(X; i) 
            \coloneqq \biggl( \prod_{(s,t) \in \cE_{\theta}} x_{ \omega_{\iota_{\pi}(s)}}^{\top} x_{ \omega_{\iota_{\pi}(t)}} \biggr) 
            \times \bigl( x_{ \omega_{\iota_{\pi}(L_{\theta})}}^{\top} x_i \bigr)^{\Ind\{ \chi = i \}}.
    \end{equation}
    The \emph{all-assignment aggregate} of $X$ with respect to $\pi$ is $\Phi_{\theta, \pi}^{\all}(X) = \bigl( \Phi_{\theta, \pi}^{\all}(X; 1), \dots, \Phi_{\theta, \pi}^{\all}(X; n) \bigr) \in \RR^n$ such that
    \begin{equation}\label{eqn:all_assignments}
        \Phi_{\theta, \pi}^{\all}(X;i) = \sum_{\omega \in \Omega^{\all}(\theta, \pi)} \varphi_{\omega}(X; i),
        \qquad\forall i \in [n].
    \end{equation}
\end{definition}

The sum $\sum_{\omega\in\Omega^{\all}(\theta,\pi)}$ \emph{allows collisions} across blocks. 
Recall the definition of the position graph of a word $\theta = (\phi, \chi)$ (Definition \ref{defn:position_graph}). 
By construction: (i) $|\cE_\theta|=m_\theta\coloneqq u(\phi)+v(\phi)+\Ind\{\chi=\avg\}$, and (ii) its edges admit a canonical enumeration $(s_1,t_1),\dots,(s_{m_\theta},t_{m_\theta})$ with $s_1<\cdots<s_{m_\theta}$ and $s_a<t_a$ for all $a$, matching the construction order.
Throughout, positions are enumerated $1,\dots,L_\theta$; thus $L_\theta$ is both the length of the word and the index of the last position (the anchor position when $\chi=i$).

\paragraph{Quotient position multigraph.}
To summarize inner‑product factors at the block level, we collapse positions within each block and retain multiplicities of self- and cross-edges.

\begin{definition}[Quotient position multigraph]\label{def:quotient_graph}
    Let $\theta=(\phi,\chi)$ be a word with position graph $\cG_\theta=(\cV_\theta,\cE_\theta)$ and let $\pi=\{P_1,\dots,P_\ell\}\in\Part(\cV_\theta)$ with index map $\iota_\pi:\cV_\theta\to[\ell]$. 
    The \emph{quotient position multigraph} is
    \[
        \cQ_{\theta,\pi}=\big([\ell],\,\{m_a\}_{a=1}^\ell,\,\{c_{ab}\}_{1\le a<b\le \ell}\big),
    \]
    where
    \begin{align*}
        m_a &\coloneqq \bigl|\{(s,t)\in\cE_\theta:\ \iota_\pi(s)=\iota_\pi(t)=a\}\bigr|   &&\text{(self-edge multiplicity of block $a$)},\\
        c_{ab} &\coloneqq \bigl|\{(s,t)\in\cE_\theta:\ \{\iota_\pi(s),\iota_\pi(t)\}=\{a,b\}\}\bigr|  &&\text{(cross-edge multiplicity between $a$ and $b$)}.
    \end{align*}
    Its underlying simple graph is $\cQ_{\theta,\pi}^{\rm simp}=([\ell],\,\cE^{\rm simp})$ with $\cE^{\rm simp} \coloneqq \bigl\{\{a,b\}:\ c_{ab}>0\bigr\}$. 
    A vertex $a$ \emph{has a self-edge} if $m_a>0$, and a pair $\{a,b\}$ is a \emph{cross-edge} if $c_{ab}>0$. 
    When $\chi=i$, the \emph{anchor block} is $a^*\coloneqq \iota_\pi(L_\theta)$, and the \emph{anchor component} $C^*$ is the connected component of $\cQ_{\theta,\pi}^{\rm simp}$ containing $a^*$; when $\chi=\avg$, no anchor is designated.
\end{definition}

By construction, $m_a$ and $c_{ab}$ equal the numbers of (possibly repeated) self‑edges and cross‑edges of $\cE_\theta$ whose endpoints collapse to $a$ and $\{a,b\}$ under $\iota_\pi$, respectively.

\paragraph{Component-restricted aggregates.} 
Because the quotient’s underlying simple graph may be disconnected, aggregates factor across its connected components. 
Thus, we define per‑component contributions.

\begin{definition}[Restricted all-assignment aggregates]\label{def:component_aggregates}
    Let $\cQ_{\theta,\pi}=\big([\ell],\{m_a\},\{c_{ab}\}\big)$. 
    For $C\subseteq[\ell]$ (e.g., a connected component of $\cQ_{\theta,\pi}^{\rm simp}$), the \emph{all-assignment aggregate restricted to $C$} is 
    \[
        \Phi^{\all}_{\theta,\pi}[C](X;i)
            ~\coloneqq~
            \sum_{\omega_C\in [n]^C}\,
            \Biggl(\prod_{a\in C} g_{\omega_C(a)}^{\,m_a}\Biggr)
            \Biggl(\prod_{\substack{\{a,b\}\subseteq C:\\ a<b}} G_{\omega_C(a),\,\omega_C(b)}^{\,c_{ab}}\Biggr)
            \times
            \begin{cases}
                G_{\omega_C(a^*),\,i}, & \text{if }\chi=i\ \text{and } a^*\in C,\\
                1, & \text{otherwise,}
            \end{cases}
    \]
    where $G=XX^\top$ and $g=\diag(G)$.
\end{definition}

\begin{lemma}[Component factorization]\label{lem:component_factorization}
    Let $C_1,\dots,C_K$ be the connected components of $\cQ_{\theta,\pi}^{\rm simp}$. 
    Then, for any $X$ and any $i\in[n]$,
    \[
        \Phi^{\all}_{\theta,\pi}(X;i)
        ~=~
        \prod_{k=1}^K \Phi^{\all}_{\theta,\pi}[C_k](X;i).
    \]
\end{lemma}

\begin{proof}[Proof of Lemma \ref{lem:component_factorization}]
    By Definition~\ref{defn:all_assignments} and Definition~\ref{def:quotient_graph}, grouping position-level factors by blocks and accounting for multiplicities $(m_a,c_{ab})$ in $\cQ_{\theta,\pi}$ rewrites the integrand (summand) as
    \[
        \Phi^{\all}_{\theta,\pi}(X;i)
            =\sum_{\omega\in[n]^\ell}\,\prod_{a=1}^{\ell} g_{\omega(a)}^{\,m_a}\,
            \prod_{1\le a<b\le \ell} G_{\omega(a),\,\omega(b)}^{\,c_{ab}}
            \times
            \begin{cases}
                G_{\omega(a^*),\,i}, & \chi=i,\\
                1, & \chi=\avg.
            \end{cases}
    \]
    Note that repeated cross‑edges and self‑edges in $\cE_\theta$ contribute multiplicatively and are equivalent to the powers $G^{c_{ab}}$ and $g^{m_a}$ used in Definition~\ref{def:component_aggregates}. 
    In particular, the induced list $\cE_\theta[C]$ includes repeated entries according to these multiplicities; each repetition triggers one contraction in the fold.

    If $a$ and $b$ lie in different components of $\cQ_{\theta,\pi}^{\rm simp}$, then $c_{ab}=0$, so there are no cross terms between different components. 
    Hence the integrand factorizes as a product of $K$ functions, each depending only on $\omega_{C_k}\in [n]^{C_k}$, with the anchor factor appearing only in the anchor component when $\chi=i$. 
    By Fubini’s theorem (finite sums), the sum over $\omega\in[n]^\ell$ factors into a product of sums over $\omega_{C_k}\in[n]^{C_k}$, which are exactly the $\Phi^{\all}_{\theta,\pi}[C_k](X;i)$ in Definition~\ref{def:component_aggregates}.
\end{proof}

\paragraph{Component folds (Gram contractions).}
The component aggregates in Definition~\ref{def:component_aggregates} can be evaluated without enumerating assignments: each edge in the induced subgraph corresponds to a single linear contraction at the Gram level. 
We formalize this as a \emph{fold} that scans the induced edges and alternates two primitive updates—$v\mapsto Gv$ for cross‑edges and $v\mapsto g\odot v$ for self‑edges—followed by a final summation $\bone^\top(\cdot)$. 
This yields exactly the component aggregate (Lemma~\ref{lem:restricted_fold_correct}).

\begin{definition}[Induced edge list]\label{def:induced-edge-list}
    Let $\{(s_a,t_a)\}_{a=1}^{m_\theta}$ be the canonical enumeration of $\cE_\theta$ (with $s_1<\cdots<s_{m_\theta}$ and $s_a<t_a$). 
    For $C\subseteq[\ell]$, define the \emph{induced edge list} on $C$ by
    \[
        \cE_\theta[C] ~\coloneqq~ \{\, (s_a,t_a)\in\cE_\theta :~ \iota_\pi(s_a)\in C\ \text{and}\ \iota_\pi(t_a)\in C \,\},
    \]
    ordered by the inherited canonical order of $\{(s_a,t_a)\}_{a=1}^{m_\theta}$.
\end{definition}

\begin{definition}[Component fold]\label{def:component_fold}
    Fix a word $\theta$, a partition $\pi$, and a component $C\subseteq[\ell]$ of $\cQ_{\theta,\pi}^{\rm simp}$.
    Set $M:=|\cE_\theta[C]|$ and enumerate $\cE_\theta[C]=\{(s_{a(m)},t_{a(m)})\}_{m=1}^M$ in increasing $a(m)$; write $b(u):=\iota_\pi(u)$.
    Let $G\in\RR^{n\times n}$ be symmetric and $v_{\init} \in \RR^n$. 
    The \emph{component fold} with respect to $G$ and $v_{\init}$ is
    \[
        \CompFold{C}(\theta,\pi;G, v_{\init}) \coloneqq \bone_n^\top v^{(0)}.
    \]
    with the vector $v^{(0)}\in\RR^n$ defined recursively by
    \begin{align*}
        v^{(M)} &= v_{\init},\\
        v^{(m-1)}
        &\coloneqq
        \begin{cases}
            g\odot v^{(m)}, & \text{if } b(s_{a(m)})=b(t_{a(m)}) \quad\text{(self-edge in $C$)},\\
            G\,v^{(m)},     & \text{if } b(s_{a(m)})\neq b(t_{a(m)}) \quad\text{(cross-edge in $C$)},
        \end{cases}
        \qquad
        \text{for }m=M,M-1,\dots,1,
    \end{align*}
    where $g = \diag(G)$, and $\odot$ denotes the Hadamard product. 
    
\end{definition}

\begin{remark}[Symmetry after masking]\label{rem:symm_mask}
    Since $P^{(-i)}$ is symmetric and idempotent, $G^{(-i)}=P^{(-i)} G P^{(-i)}$ is symmetric whenever $G$ is. 
    Hence the symmetry assumption on $G$ in Definition~\ref{def:component_fold} is fully compatible with masking.
\end{remark}

\begin{lemma}[Restricted folding equals component aggregate]\label{lem:restricted_fold_correct}
    For any word $\theta$, partition $\pi$, component $C\subseteq[\ell]$, matrix $X \in \RR^{n \times p}$, and index $i\in[n]$,
    \[
        \CompFold{C}(\theta,\pi; G^\star, v_{\init}^\star) = \Phi^{\all}_{\theta,\pi}[C](X;i),
    \]
    where $G^\star = XX^{\top}$ and 
    \[
        v_{\init}^\star =
        \begin{cases}
            G(:,i), & \text{if }\chi=i\ \text{ and } a^*\in C,\\
            \bone_n, & \text{otherwise}.
        \end{cases}
    \]
\end{lemma}

Here the fold uses $g^\star\coloneqq\diag(G^\star)$ implicitly in the Hadamard steps.

\begin{remark}[Order independence]
    Only the \emph{set} $\cE_\theta[C]$ matters. Each fold step eliminates one summation index via a linear contraction. 
    Changing the scan order permutes the order of summations but not their value; by Fubini's theorem, $\CompFold{C}$ is invariant to the order chosen.
\end{remark}

\begin{proof}[Proof of Lemma \ref{lem:restricted_fold_correct}]
    Fix $C$ and $i$. Let $\{(s_{a(m)},t_{a(m)})\}_{m=1}^M$ be the induced edge list $\cE_\theta[C]$ in increasing order of $a(\cdot)$. 
    We prove by backward induction on $m$ that $v^{(m)}\in\RR^n$ produced by Definition~\ref{def:component_fold} satisfies
    \[
        v^{(m)}_j
            ~=~ \sum_{\omega_C\in[n]^C}\;
            \prod_{\substack{r>m\\(s_{a(r)},t_{a(r)})\in\cE_\theta[C]}}
            \bigl\langle x_{\omega_C(b(s_{a(r)}))},\,x_{\omega_C(b(t_{a(r)}))}\bigr\rangle\,
            \times
            \begin{cases}
                \langle x_{\omega_C(a^*)},\,x_i\rangle, & \chi=i\ \text{and } a^*\in C,\\
                1,& \text{otherwise,}
            \end{cases}
    \]
    where, at step $m\ge 1$, the coordinate $j$ indexes the label of the (unique) \emph{currently free block} $b \big(t_{a(m)}\big)$, and at $m=0$ it indexes the last remaining free block determined by the updates. 
    The terminal sum $\bone_n^\top v^{(0)}$ then completes the summation over that block.

    \begin{itemize}
        \item 
        For the base case $m=M$: if $\chi=i$ and $a^*\in C$, then the only factor present is $\langle x_{\omega_C(a^*)},x_i\rangle$, so $v^{(M)}_j=\langle x_j,x_i\rangle$ agrees with the initialization $G(:,i)$; otherwise no factor is present and $v^{(M)}=\bone_n$.

        \item 
        Suppose the claim holds for $m$. Consider edge $(s_{a(m)},t_{a(m)})$. 
        If $b(s_{a(m)})=b(t_{a(m)})$, then introducing this self-edge multiplies the current partial sum pointwise by $g_j=\langle x_j,x_j\rangle$, which is exactly the Hadamard update $v^{(m-1)}=g\odot v^{(m)}$. 
        If $b(s_{a(m)})\neq b(t_{a(m)})$, then adding this cross-edge introduces a fresh label $k$ at $b(s_{a(m)})$ to be summed out against the current free label $j$ at $b(t_{a(m)})$, yielding
        \[
            v^{(m-1)}_j  =  \sum_{k=1}^n \langle x_k,x_j\rangle\, v^{(m)}_k  =  (G\,v^{(m)})_j,
        \]
        matching the update in the definition. 
        This proves the induction.
    \end{itemize}
    
    Consequently, at $m=0$, all edges in $\cE_\theta[C]$ have been accounted for; summing over the remaining free block label gives $\bone_n^\top v^{(0)}$, which equals $\Phi^{\all}_{\theta,\pi}[C](X;i)$ by Definition~\ref{def:component_aggregates}.
\end{proof}

\begin{remark}
    Under the covariate normalization $X^{\top} X = n I_p$ (Assumption \ref{assump:covariates} in Section \ref{sec:reg_adjustment}), the hat matrix $H = X(X^{\top}X)^{-1} X^{\top} = \frac{1}{n} G$. 
    If one prefers to work with $(H,h)$ instead of $(G,g)$, the fold remains valid provided one multiplies the final scalar by $n^{\,m_\theta}$ and, when $\chi=i$, initializes with $n\,H(:,i)$ in place of $G(:,i)$. 
    This exactly compensates for the $1/n$ factors introduced at each edge (and at initialization when $\chi=i$).
\end{remark}

\subsubsection{Index‑exclusion masking and Class 0/1 separation}
To separate injective allocations that \emph{avoid} $i$ (Class~0) from those that \emph{use} $i$ exactly once (Class~1), we evaluate the all‑assignment sums twice: unmasked (over $[n]$) and masked (over $[n]\setminus\{i\}$). 
Masking is implemented by zeroing the $i$‑th coordinate at every contraction step; algebraically this corresponds to replacing $(G,g)$ by $(G^{(-i)},g^{(-i)})$ and pre‑multiplying initial messages by $P^{(-i)}$.

Recall the Class 0/1 monomial aggregates, cf. \eqref{eqn:diagram_aggregates}:
\[
    \Phi_{\pi,0}(X;i)=\sum_{\omega\in\Omega(\theta,\pi)_0}\varphi_\omega(X;i),
    \qquad
    \Phi_{\pi,1}(X;i)=\sum_{\omega\in\Omega(\theta,\pi)_1}\varphi_\omega(X;i),
\]
where $\Omega(\theta,\pi)_0$ collects injective allocations that avoid $i$ and $\Omega(\theta,\pi)_1$ have exactly one block equal to $i$. 
We define the notion of mask to connect these Class 0/1 monomial aggregates to the all-assignment aggregates.

\begin{definition}[Mask at $i$]\label{def:mask}
    Let $X \in \RR^{n \times p}$. 
    For any $i \in [n]$, the \emph{$i$-th masked Gram objects} are defined as
    \begin{equation}\label{eq:mask-gram}
        G^{(-i)} \coloneqq P^{(-i)} XX^{\top} P^{(-i)}
        \qquad\text{and}\qquad
        g^{(-i)} = \diag( G^{(-i)})
    \end{equation}
    where $P^{(-i)} = I_n - e_i e_i^{\top}$.
\end{definition}

In implementations, one may apply the mask “on the fly’’ by performing the masked cross-edge contraction as
\[
    v ~\longmapsto~ P^{(-i)} \big(G\,(P^{(-i)} v)\big)
\]
at every cross‑edge step, and by replacing $g$ with the masked diagonal $g^{(-i)}=\diag \big(P^{(-i)} G P^{(-i)}\big)$ at self‑edge steps. 
This avoids forming $P^{(-i)}GP^{(-i)}$ explicitly while exactly enforcing the exclusion of index $i$.

\begin{definition}[Masked component-restricted aggregates]\label{def:masked_component_aggregates}
    With the same notation as Definition~\ref{def:component_aggregates}, for $i \in [n]$, define
    \[
        \Phi^{\all,(-i)}_{\theta,\pi}[C](X;i)
            ~\coloneqq~
            \sum_{\omega_C\in ([n]\setminus\{i\})^{C}}\,
            \Biggl(\prod_{a\in C} g_{\omega_C(a)}^{\,m_a}\Biggr)
            \Biggl(\prod_{\substack{\{a,b\}\subseteq C:\\ a<b}}
            (G_{\omega_C(a),\,\omega_C(b)})^{\,c_{ab}}\Biggr)
            \times
            \begin{cases}
                G_{\omega_C(a^*),\,i}, & \text{if }\chi=i\ \text{and } a^*\in C,\\
                1, & \text{otherwise.}
            \end{cases}
    \]
\end{definition}

Recall from Definition \ref{defn:all_assignments} that for any word $\theta$ and any $\pi \in \Part(\cV_{\theta})$, the \emph{all-assignment} aggregate
\[
    \Phi^{\all}_{\theta,\pi}(X;i) = \sum_{\omega\in \Omega^{\all}(\theta, \pi)}\ \varphi_\omega(X;i).
\]
Observe that $\Omega^{\all}(\theta, \pi) = [n]^{|\pi|}$, which allows collisions across blocks, and that $\Phi^{\all}_{\theta,\pi}(X;i)$ can be exactly evaluated by the folding procedure with $(G, g)$ as in Lemma \ref{lem:restricted_fold_correct}. 
Carrying out the fold with $(G^{(-i)},g^{(-i)})$, instead of $(G, g)$, and initializing with $v^{(M)}:=P^{(-i)}G e_i$ when $\chi=i$ and $a^*\in C$, and $v^{(M)}:=P^{(-i)}\bone_n$ otherwise, evaluates the all-assignment aggregate \emph{excluding} the label $i$, namely,
\begin{equation}\label{eqn:aggregate_excluding_i}
    \Phi^{\all,(-i)}_{\theta,\pi}(X;i)\ =\ \sum_{\omega\in ([n]\setminus\{i\})^{|\pi|}}\ \varphi_\omega(X;i),
\end{equation}
because each contraction step suppresses any path that would use the index $i$. 

\begin{lemma}[Masked restricted folding equals masked component aggregate]\label{lem:masked_fold_correct}
    Fix $i \in [n]$. 
    Let $G=XX^\top$, $P^{(-i)}=I-e_ie_i^\top$, and $G^{(-i)}=P^{(-i)} G P^{(-i)}$. 
    For any component $C\subseteq[\ell]$,
    \[
        \CompFold{C}\big(\theta,\pi;\,G^{(-i)}, v_{\mathrm{init}}^{(-i)}\big)
        = \Phi^{\all,(-i)}_{\theta,\pi}[C](X;i),
    \]
    where
    \[
        v_{\mathrm{init}}^{(-i)}  = 
        \begin{cases}
            P^{(-i)}G e_i, & \text{if }\chi=i\ \text{and } a^*\in C,\\
            P^{(-i)}\bone_n, & \text{otherwise.}
        \end{cases}
    \]
\end{lemma}

\begin{remark}
    Masked evaluations restrict indices to $([n]\setminus\{i\})$, so Hadamard steps use the masked diagonal $g^{(-i)}\coloneqq \diag(G^{(-i)})$ (cf. Definition \ref{def:component_fold}). 
    Note that the initial vector for the anchor component $P^{(-i)}G e_i \neq G^{(-i)} e_i = 0$.
\end{remark}

\begin{proof}
    Identical to the proof of Lemma~\ref{lem:restricted_fold_correct}, with the following two observations:
    (i) replacing $(G,g)$ by $(G^{(-i)},g^{(-i)})$ forbids the index $i$ from appearing in any contraction step, so summation is over $([n]\setminus\{i\})^C$;
    (ii) multiplying the initialization by $P^{(-i)}$ zeros the $i$‑th coordinate, which forbids the remaining free block from taking label $i$. 
    The terminal factor $G_{\omega_C(a^*),i}$ is preserved by initializing with $P^{(-i)}G e_i$ (not $G^{(-i)}e_i$).
\end{proof}

See Algorithm \ref{alg:fold-and-sum} for precise description of the procedure.

\begin{algorithm}[H]
    \caption{\textsc{FoldAndSum}$(\theta,\pi;\,X;\,i;\,\mathrm{mask})$ — evaluates $\Phi^{\all}_{\theta,\pi}(X;i)$
    (or $\Phi^{\all,(-i)}_{\theta,\pi}(X;i)$ if $\mathrm{mask}=\texttt{True}$)}
    \label{alg:fold-and-sum}
    \begin{algorithmic}[1]
        \Require    $\theta$, $\pi\in\Part(\cV_\theta)$, $X\in\RR^{n\times p}$, $i\in[n]$, $\mathrm{mask}\in\{\texttt{True},\texttt{False}\}$
        \Ensure A scalar equal to $\Phi^{\all}_{\theta,\pi}(X;i)$ if $\mathrm{mask}=\texttt{False}$, and to $\Phi^{\all,(-i)}_{\theta,\pi}(X;i)$ if $\mathrm{mask}=\texttt{True}$
        \State 
            Build $\cQ_{\theta,\pi}$ and its simple graph $\cQ^{\rm simp}_{\theta,\pi}$; let $\{C_1,\dots,C_K\}$ be the connected components.
        \If{$\mathrm{mask}=\texttt{True}$}
            \State  $P\gets I_n-e_ie_i^\top$;\quad $G^\star\gets PXX^\top P$
        \Else
            \State  $G^\star\gets XX^\top$
        \EndIf
        \State  $Z\gets 1$
        \For{$k=1$ \textbf{to} $K$}
            \If{$\chi=i$ \textbf{and} $a^*\in C_k$}
                \State $v_{\init}\gets \begin{cases}
                    P\,XX^{\top} e_i, & \mathrm{mask}=\texttt{True} \\
                    XX^{\top} e_i, & \mathrm{mask}=\texttt{False}
                \end{cases}$
            \Else
                \State $v_{\init}\gets \begin{cases}
                    P\,\bone_n, & \mathrm{mask}=\texttt{True} \\
                    \bone_n, & \mathrm{mask}=\texttt{False}
                \end{cases}$
            \EndIf
            \State  $Z\gets Z \times \CompFold{C_k}(\theta,\pi;~G^\star, v_{\init})$  \Comment{See Definitions~\ref{def:component_fold} \& \ref{def:mask} for $\CompFold{C_k}(\theta,\pi;~G^\star, v_{\init})$}
        \EndFor
        \State \Return $Z$
    \end{algorithmic}
\end{algorithm}

\subsubsection{Partition lattice and M\"{o}bius correction}
Next, we develop machinery to resolve over-counting in the all-assignment aggregation via M\"{o}bius inclusion-exclusion principle.
Recall the definition of a partition of a set and the blocks (Definition \ref{defn:partition}). 
The set of partitions of a set $\cT$ form a lattice $\bigl( \Part(\cT), \preceq \bigr)$ by the refinement order $\preceq$: for any set $\cT$ and any $\pi_1, \pi_2 \in \Part(\cT)$,
\[
    \pi_1 \preceq \pi_2
    \quad\text{if and only if}\quad
    \text{every block of }\pi_1\text{ is contained in a block of }\pi_2.
\]
Let $\disc_{\cT} \coloneqq \bigl\{ \{ z\}: z \in \cT \bigr\}$ be the discrete (finest) partition of $\cT$.

\begin{definition}[Coarsening of partitions]
    Let $\cT \subset \ZZ$. 
    For a partition $\pi \in \Part(\cT)$, for any $\rho \in \Part( [|\pi|] )$, the \emph{coarsening} (or composition) of $\pi$ by $\rho$ is the partition
    \[
        \rho \circ \pi \coloneqq \left\{ P'_b = \bigcup_{a \in Q_b} P_a: Q_b \in \rho \right\} \in \Part(\cT),
    \]
    enumerated by increasing block minima $\min P'_1 < \dots < \min P'_{|\rho|}$. 
    Letting $\tau_\rho:[|\pi|]\to[|\rho|]$ map $a\mapsto b$ if and only if $a\in Q_b$, the index map of $\rho\circ\pi$ is $\iota_{\rho\circ\pi}=\tau_\rho\circ\iota_\pi$.
\end{definition}

All sums depend only on block membership, not on block labels. 
Thus, the enumeration convention in coarsening only fixes labels and does not affect any of the sums.

Now we define the M\"{o}bius correction operator.
\begin{definition}[M\"{o}bius correction]\label{defn:mobius_correction}
    Let $\cT \subset \ZZ$. 
    The \emph{M\"{o}bius correction} of a functional $F: \Part(\cT) \to \RR^n$ is $\Mob[F]: \Part(\cT) \to \RR^n$ such that
    \[
        \Mob[F](\pi) = \sum_{ \rho \in \Part([\ell])} \lambda(\rho) \cdot F\bigl( \rho \circ \pi \bigr)
    \]
    where $\ell = |\pi|$ and
    \[
        \lambda(\rho) = (-1)^{\ell - |\rho|} \prod_{Q \in \rho} \bigl(|Q|-1 \bigr)!
    \]
    is the M\"{o}bius weight of $\rho \in \Part([\ell])$.
\end{definition}

Now we are ready to present the main theorem in this section. 

\begin{theorem}[M\"{o}bius reduction]\label{thm:mobius_reduction}
    Let $\theta$ be a word and $\pi \in \Part(\cV_{\theta})$. 
    For any $X \in \RR^{n \times p}$ and any $i \in [n]$,
    \begin{align}
        \Phi_{\pi,0}(X;i)
            &= \sum_{\rho\in\Part([\ell])}\ \lambda(\rho) \cdot \Phi^{\all,(-i)}_{\theta,\ \rho\circ\pi}(X;i),
                \label{eq:mob-class0}\\
        \Phi_{\pi,1}(X;i)
            &= \sum_{\rho\in\Part([\ell])}\ \lambda(\rho) \cdot \Big[\Phi^{\all}_{\theta,\ \rho\circ\pi}(X;i)\ -\ \Phi^{\all,(-i)}_{\theta,\ \rho\circ\pi}(X;i)\Big],
                \label{eq:mob-class1}
    \end{align}
    where $\ell = |\pi|$ and $\Phi^{\all,(-i)}_{\theta,\ \pi}$ is defined in \eqref{eqn:aggregate_excluding_i}.
\end{theorem}

\begin{remark}[Algorithmic implication]
    For each coarsening $\rho\circ\pi$, computing two folds suffices: (i) an unmasked fold computes $\Phi^{\all}_{\theta,\rho\circ\pi}(X;i)$; (ii) a masked fold with $(G^{(-i)},g^{(-i)})$ computes $\Phi^{\all,(-i)}_{\theta,\rho\circ\pi}(X;i)$. Plugging into \eqref{eq:mob-class0}–\eqref{eq:mob-class1} with weights $\lambda(\rho)$ yields $\Phi_{\pi,0},\Phi_{\pi,1}$, which enter Theorem~\ref{thm:xi_formula}.
\end{remark}

\begin{proof}[Proof of Theorem \ref{thm:mobius_reduction}]
    For $\rho\in\Part([\ell])$, let $\mathsf{A}(\rho) \coloneqq \Phi^{\all}_{\theta,\ \rho\circ\pi}(X;i)$ be the all-assignment sum with blocks identified inside parts of $\rho$, and let $\mathsf{I}(\rho)$ be the same sum restricted to maps injective across the parts of $\rho$. 
    Grouping by the finest collision pattern gives the incidence relation on $(\Part([\ell]),\preceq)$:
    \[
        \mathsf{A}(\rho)=\sum_{\sigma\preceq\rho}\mathsf{I}(\sigma),    \qquad \rho \in \Part([\ell]).
    \]
    Recall $\disc_{[\ell]}=\{\{1\},\dots,\{\ell\}\}$ denotes the discrete partition. 
    M\"{o}bius inversion on $(\Part([\ell]),\preceq)$ yields 
    \[
    \mathsf{I}(\disc_{[\ell]})=\sum_{\rho\in\Part([\ell])}\lambda(\rho)\,\mathsf{A}(\rho).
    \]
    Evaluating $\mathsf{I}(\disc_{[\ell]})$ over $[n]$ gives the injective aggregate over all labels, that is,
    \begin{equation}\label{eq:mob-all}
        \Phi_{\pi}(X;i) 
            \coloneqq \sum_{\omega \in \Omega(\theta, \pi) }\varphi_\omega(X;i)
            = \sum_{\rho\in\Part([\ell])}\lambda(\rho)\cdot \Phi^{\all}_{\theta,\ \rho\circ\pi}(X;i).
    \end{equation}
    Similarly, evaluating $\mathsf{I}(\disc_{[\ell]})$ with the masked fold \((G^{(-i)},g^{(-i)})\), cf. Definition~\ref{def:mask}, gives the injective aggregate over $[n]\setminus\{i\}$: 
    \[
        \Phi_{\pi,0}(X;i)=\sum_{\rho\in\Part([\ell])}\lambda(\rho) \cdot\Phi^{\all,(-i)}_{\theta,\ \rho\circ\pi}(X;i).
    \]
    This proves \eqref{eq:mob-class0}.

    Finally, under injectivity across blocks, each label can appear at most once; therefore, the injective sum over all labels minus the injective sum that excludes $i$ isolates the case where $i$ appears exactly once. 
    Therefore, $\Phi_{\pi,1}(X;i)=\Phi_{\pi}(X;i)-\Phi_{\pi,0}(X;i)$. 
    Substituting the two inversion formulas from \eqref{eq:mob-class0} and \eqref{eq:mob-all} yields \eqref{eq:mob-class1}.

\end{proof}

\subsection{Algorithm description}\label{sec:folding_algorithms}
\subsubsection{Folding-M\"{o}bius algorithm}\label{sec:folding_mobius}

\begin{algorithm}[H]
    \caption{Folding–M\"{o}bius algorithm to compute $\xi^{[d]}_i(m;X)$ (for fixed $i \in [n]$).}
    \label{alg:folding_mobius}
    \begin{algorithmic}[1]
        \Require $X\in\RR^{n\times p}$, subset size $m \in [n]$, degree $d \in \ZZnn$, index $i \in [n]$
        \Ensure $\xi^{[d]}_i(m;X)\in\RR$
        \State $\xi\gets 0$;\quad $G\gets XX^\top$;\quad $g\gets\diag(G)$
        \ForAll{$\theta=(\phi,\chi)\in\{I,U,V\}^d\times\{i,\avg\}$}
            \State 
                Build $\cG_\theta=(\cV_\theta,\cE_\theta)$; enumerate edges $(s_a,t_a)_{a=1}^{m_\theta}$
            \ForAll{$\pi\in\Part(\cV_\theta)$}
                \State 
                    $\Phi_{\pi,0}\gets 0$; $\Phi_{\pi,1}\gets 0$; $\ell \gets |\pi|$
                \ForAll{$\rho\in\Part([\ell])$}
                    \State $\widetilde{\Phi} \gets$ \textsc{FoldAndSum}$(\theta,\rho\circ\pi;~X; ~i; ~\texttt{False})$
                    \State $\widetilde{\Phi}^{(-i)}\gets$ \textsc{FoldAndSum}$(\theta,\rho\circ\pi;~X; ~i; ~\texttt{True})$
                    \State $\Phi_{\pi,0}\gets \Phi_{\pi,0}+\lambda(\rho)\cdot \widetilde{\Phi}^{(-i)}$
                    \State $\Phi_{\pi,1}\gets \Phi_{\pi,1}+\lambda(\rho)\cdot (\widetilde{\Phi} - \widetilde{\Phi}^{(-i)})$
                \EndFor
                \State $\xi\gets \xi+\dfrac{\mathrm{sgn}_\theta}{m^{L_\theta}}\Big(\psi_0(\pi)\,\Phi_{\pi,0}+\psi_1(\pi)\,\Phi_{\pi,1}\Big)$
            \EndFor
        \EndFor
        \State \Return $\xi$
    \end{algorithmic}
\end{algorithm}

\begin{enumerate}
    \item 
    \emph{Edges.} Component folds (Definition~\ref{def:component_fold}) use the induced edge lists $\cE_\theta[C]$; only whether an edge is self vs.\ cross \emph{within each component of} $\cQ_{\theta,\pi}^{\rm simp}$ matters.
    \item 
    \emph{Fold evaluation.} \textsc{FoldAndSum}(\,$\theta,\rho\circ\pi;X,i$\,) multiplies the per‑component folds (Algorithm~\ref{alg:fold-and-sum}); Lemmas~\ref{lem:restricted_fold_correct} and \ref{lem:masked_fold_correct} justify the unmasked/masked cases.
    See Algorithm \ref{alg:fold-and-sum} for a complete pseudocode.
    \item 
    \emph{M\"{o}bius correction.} For each coarsening $\rho\circ\pi$, combine unmasked/masked values with weight $\lambda(\rho)$ to obtain the injective aggregates $\Phi_{\pi,0},\Phi_{\pi,1}$ via \eqref{eq:mob-class0}–\eqref{eq:mob-class1}. 
    \item 
    \emph{Final accumulation.} Substitute $\Phi_{\pi,0}$ and $\Phi_{\pi,1}$ into Theorem~\ref{thm:xi_formula} with prefactor $\mathrm{sgn}_\theta/m^{L_\theta}$ and SRSWOR weights $\psi_0(\pi)=\frac{(m-1)_\ell}{(n-1)_\ell}$, $\psi_1(\pi)=\frac{(m-1)_{\ell-1}}{(n-1)_{\ell-1}}$.
\end{enumerate}

\begin{remark}[Complexity]
    Fix a word $\theta$ and a partition $\pi$ with $\ell=|\pi|$ and total edge count $m_\theta$. For each coarsening $\rho\in\Part([\ell])$, \textsc{FoldAndSum} evaluates two products of component folds (masked and unmasked). 
   If a component $C$ has $m_C^{\rm cross}$ cross‑edges and $m_C^{\rm self}$ self‑edges, one fold costs 
    \[
        O \left(m_C^{\rm cross}\cdot (np) + m_C^{\rm self}\cdot n\right)
    \]
    using $Gv=X(X^\top v)$. 
    Summing over components, each masked or unmasked fold per each coarsening costs $O \left(m_\theta\,np + m_\theta\,n\right)$, i.e., $O(m_\theta\,np)$. 
    Summing over the $B_\ell$ (Bell number) coarsenings and doubling for masked/unmasked gives a total of $O \big(B_\ell\,m_\theta\,np\big)$ per word–partition pair $(\theta,\pi)$ when $Gv$ is implemented as $X(X^\top v)$; memory requirement is $O(np)$ for $X$ and $O(n)$ for working vectors (no $n^2$ storage of $G$ is required).
\end{remark}

\begin{remark}[Comparison vs.\ brute-force enumeration]
    Brute-force injective summation over $\Omega(\theta,\pi)$ costs $n_{(\ell)}=n(n-1)\cdots(n-\ell+1)$ terms. The Folding--M\"{o}bius evaluation replaces this by: (\emph{a}) $B_\ell$ coarsenings (Bell number; depends only on $\ell\le L_\theta\le 2d+2$), and (\emph{b}) for each coarsening, two folds of length $m_\theta$; each fold costs $m_\theta$ applications of $v\mapsto Gv=X(X^\top v)$ (i.e., $O(np)$) or $v\mapsto g\odot v$ (i.e., $O(n)$). Thus the cost is independent of $n$ combinatorially and typically far smaller than $n_{(\ell)}$ for moderate $d$ and large $n$.
\end{remark}

\subsubsection{Scalar Folding–M\"{o}bius algorithm}\label{sec:scalar_folding_mobius}

\begin{algorithm}[ht!]
    \caption{Scalar Folding–M\"{o}bius algorithm for $\bbE[\sfR^{[d]}_{\arm}]$.}
    \label{alg:scalar_folding_mobius}
    \begin{algorithmic}[1]
        \Require $X\in\RR^{n\times p}$, subset size $m \in [n]$, degree $d \in \ZZnn$, residual $r^{(\arm)}\in\RR^n$
        \Ensure $\bbE[\sfR^{[d]}_{\arm}]$
        \State $S \gets 0$;\quad $G\gets XX^\top$
        \ForAll{$\theta=(\phi,\chi)\in\{I,U,V\}^d\times\{i,\avg\}$}
            \ForAll{$\pi\in\Part(\cV_\theta)$}
                \State $S_0\gets 0$; $S_1\gets 0$; $\ell\gets|\pi|$
                \ForAll{$\rho\in\Part([\ell])$}
                    \State Build $\cQ_{\theta,\rho\circ\pi}$ and components $\{C_1,\dots,C_K\}$; set $G^\star\gets G$
                    \State \emph{// Unmasked part $R_{\rm all}$}
                    \If{$\chi=i$}
                        \State $\text{const}\gets 1$
                        \For{$k=1$ \textbf{to} $K$}
                            \If{$a^*\notin C_k$} \ $\text{const}\gets \text{const}\times \CompFold{C_k}(\theta,\rho\circ\pi;\,G^\star,\,\bone_n)$
                            \Else \ $\text{anchor}\gets k$
                            \EndIf
                        \EndFor
                        \State $R_{\rm all}\gets \text{const}\times \CompFold{C_{\text{anchor}}}(\theta,\rho\circ\pi;\,G^\star,\,X(X^\top r^{(\arm)}))$
                    \Else
                        \State $R_{\rm all}\gets 0$ \Comment{Because $\sum_i r^{(\arm)}_i\,\Phi^{\all}_{\theta,\rho\circ\pi}(X;i)=\Phi^{\all}\,\bone^\top r^{(\arm)}=0$ by \eqref{eqn:population_OLS}}
                    \EndIf
                    \State \emph{// Masked part $R_{\rm ex}$}
                    
                    \State $R_{\rm ex}\gets 0$
                    \For{$i=1$ \textbf{to} $n$}
                        \State $P\gets I_n-e_ie_i^\top$
                        \State $\text{const}\gets 1$
                        \For{$k=1$ \textbf{to} $K$}
                            \If{$\chi=i$ \textbf{and} $a^*\in C_k$}
                                \State $v_{\init}\gets P\,G e_i$ \Comment{Preserves the anchor factor $G_{\omega(a^*),i}$}
                            \Else
                                \State $v_{\init}\gets P\,\bone_n$
                            \EndIf
                            \State $G^\star\gets PGP$ 
                            \State $\text{const}\gets \text{const}\times \CompFold{C_k}(\theta,\rho\circ\pi;\,G^\star,\, v_{\init})$
                        \EndFor
                        \State $R_{\rm ex}\gets R_{\rm ex}+ r^{(\arm)}_i\,\text{const}$
                    \EndFor
                    \State $S_0\gets S_0+\lambda(\rho)\,R_{\rm ex}$;\quad $S_1\gets S_1+\lambda(\rho)\,(R_{\rm all}-R_{\rm ex})$
                \EndFor
                \State $S \gets S +\dfrac{\mathrm{sgn}_\theta}{m^{L_\theta}}\Big(\psi_0(\pi)\,S_0+\psi_1(\pi)\,S_1\Big)$
            \EndFor
        \EndFor
        \State \Return $\frac{1}{n}\,S$
    \end{algorithmic}
\end{algorithm}

The scalar routine (Algorithm \ref{alg:scalar_folding_mobius}) evaluates $\frac{1}{n}\sum_i \xi^{[d]}_i r^{(\arm)}_i$ without forming $\xi^{[d]}$.
\begin{itemize}
    \item 
    \emph{Unmasked part.} Only the anchor component (when $\chi=i$) depends on $i$. Linearity yields
    \[
        \sum_{i=1}^n r^{(\arm)}_i\cdot \CompFold{C^*}(\cdot;\,G,\,G e_i)
         =  \CompFold{C^*}(\cdot;\,G,\,X(X^\top r^{(\arm)})).
    \]
    which is multiplied by the product of non‑anchor components evaluated at $\bone_n$. 
    When $\chi=\avg$, $\Phi^{\all}_{\theta,\rho\circ\pi}(X;i)$ is independent of $i$, so $\sum_i r^{(\arm)}_i\,\Phi^{\all}_{\theta,\rho\circ\pi}(X;i)=\Phi^{\all}_{\theta,\rho\circ\pi}(X)\,\bone^\top r^{(\arm)}=0$ by the population normal equations; cf.\ \eqref{eqn:population_OLS}.

    \item 
    \emph{Masked part.} The mask depends on $i$, so $R_{\rm ex}$ necessarily uses an $i$‑loop, with per‑component initializations $P\,G e_i$ (anchor) or $P\bone_n$ (non‑anchor). 
    Each masked cross‑edge contraction (Lines 31--32) is applied as $v \leftarrow P^{(-i)} \big(G(P^{(-i)} v)\big)$, which has the same $O(np)$ cost as an unmasked $Gv$ when implemented via $Gv=X(X^\top v)$, and avoids forming $P^{(-i)} G P^{(-i)}$ (which would require $O(n^2)$ memory).
\end{itemize}

\paragraph{Why the scalar routine can be preferable.}
The vector routine forms $\xi^{[d]}\in\mathbb{R}^n$ by evaluating all (coarsening, component) folds for every $i$. 
In contrast, the scalar routine pushes $\frac{1}{n}\langle \xi^{[d]},r^{(\arm)}\rangle$ inside the loops and uses linearity to replace $\sum_i r^{(\arm)}_i\CompFold{C^*}(\cdot;\,G,Ge_i)$ by a \emph{single} anchor fold with initialization $X(X^\top r^{(\arm)})$. 
Thus the unmasked part avoids the $i$-loop entirely, and only the masked part retains the $i$-loop due to the $i$-dependent mask. When $n$ is large and the number of masked coarsening–component folds is moderate (i.e., for moderate $d$), this reduces
both memory (never storing $\xi^{[d]}$) and time (fewer matrix–vector products).

\begin{remark}[Complexity comparison]
    Let $T_{\theta,\pi}\coloneqq B_{|\pi|}\cdot m_\theta$ denote the number of folds (coarsenings times edges) for a fixed $(\theta,\pi)$. 
    The vector routine computes, for each $i$, \emph{both} an unmasked and a masked fold per coarsening: its cost is 
    \[
        O \big(n\cdot T_{\theta,\pi}\cdot (np)\big).
    \]
    The scalar routine eliminates the entire $i$‑loop for the unmasked part: its cost is
    \[
        O \big(T_{\theta,\pi}\cdot (np)\big) + O \big(n\cdot T_{\theta,\pi}\cdot (np)\big)
    \]
    (masked part unchanged). 
    Therefore, for those words with $\chi=i$ the scalar routine halves the unmasked work; for $\chi=\avg$ it eliminates it entirely, yielding a larger reduction. 
    In all cases, the scalar routine requires only $O(n)$ memory since $\xi^{[d]}$ is never materialized.
\end{remark}

%% file: contents/Appx04_proof_main_thm1.tex
\section{Deferred proof of Theorem \ref{thm:main-OLS} in Section \ref{sec:theory-OLS}}\label{sec:proof_main_theorem}

\subsection{Helper lemma: Row-swap coupling and Poincar\'{e} bound}\label{sec:Poincare_Johnson}

Let $\cS\subset[n]$ be a random subset of fixed size $m$ drawn by simple random sampling without replacement (SRSWOR), i.e., $\cS$ is drawn uniformly at random from $\binom{[n]}{m}$.
For any set $\cS \in \binom{[n]}{m}$ and any ordered pair $(j,k) \in \cS \times \cS^c$, where $\cS^c \coloneqq [n] \setminus \cS$, define the $(j,k)$-swapped set of $\cS$:
\[
    \cS^{(j\leftrightarrow k)} \coloneqq \big( \cS\setminus\{j\} \big) \cup \{k\},
\]
Given a function $F: \binom{[n]}{m} \to\RR$, define its \emph{one-swap increment} at $(\cS;~j,k)$ as
\begin{equation}\label{eqn:swap_operator}
    \DeltaJK F(\cS) \coloneqq F\bigl(\cS^{(j\leftrightarrow k)}\bigr) - F(\cS).
\end{equation}
Unless otherwise stated, expectations in the rest of this subsection are taken with respect to the joint law where
\[
    \cS\sim\Unif\big(\textstyle\binom{[n]}{m}\big),
    \qquad
    J\sim\Unif(\cS),
    \qquad 
    K\sim\Unif(\cS^c),
\]
with $J,K$ sampled conditionally independently given $\cS$.

\begin{definition}\label{defn:s2_functional}
    Let $n \in \ZZnn$ and $m \in [n]$. 
    The \emph{average-swap sensitivity} of $F: \binom{[n]}{m} \to \RR$ is
    \begin{equation}\label{eqn:avg_swap_sensitivity}
    \begin{aligned}
        \mathsf{s}_2(F) 
            &\coloneqq \Bigl\{ \, \bbE\Bigl[ \bigl(\DeltaJK F(\cS)\bigr)^2 \Bigr]\, \Bigr\}^{1/2}\\
            &= \Bigl\{ \, \bbE\Bigl[ \bigl(\, F\bigl(\cS^{(j\leftrightarrow k)}\bigr) - F(\cS) \,\bigr)^2 \Bigr]\, \Bigr\}^{1/2}.
    \end{aligned}
    \end{equation}
\end{definition}
This quantity is exactly the root–mean–square one–swap increment that appears on the right–hand side of the variance bound to be presented below; we use it only as a notational convenience.
Below we state a variance bound as a function of $\mathsf{s}_2(F)$ obtained via the Poincar\'{e} (spectral gap) inequality \citep[Chapter 13]{levin2017markov} specialized to the simple random walk on the Johnson graph \citep[Chapter 12]{brouwer2011spectra}.
\begin{lemma}[Variance bound under SRSWOR via swaps]\label{lem:ES-swap}
    For every $F:\binom{[n]}{m}\to\RR$,
    \begin{equation}\label{eq:ES-swap}
    \begin{aligned}
        \Var\big(F(\cS)\big)
            &\leq \frac{m(n-m)}{2n} \cdot \bbE\Bigl[\Bigl(\, F(\cS^{(J\leftrightarrow K)}) - F(\cS) \,\Bigr)^2\Big]\\
            &= \frac{m(n-m)}{2n} \cdot \mathsf{s}_2(F)^2.
    \end{aligned}
    \end{equation}
\end{lemma}

\begin{proof}[Proof of Lemma~\ref{lem:ES-swap}]
    Let $J(n,m) = \bigl( \cV_{n,m}, \cE_{n,m} \bigr)$ be the Johnson graph, i.e., an undirected graph such that 
    $\cV_{n,m} = \bigl\{ \cS \subseteq [n]: |\cS| = m  \bigr\}$ and $\cE_{n,m} = \bigl\{ \{ v_1, v_2\} \in \cV_{n,m} \times \cV_{n,m}: |v_1 \cap v_2| = m-1  \bigr\}$.
    
    Consider the Markov chain on $J(n,m)$ with state space $\cV_{n,m} = \binom{[n]}{m}$ that moves from a state $\cS$ to a neighbor $\cS' \in \cV_{n,m}$ chosen uniformly at random by swapping one $J\in\cS$ with one $K\notin\cS$, i.e.,
    \[
        P( \cS,\cS' ) = 
            \begin{cases}
                \frac{1}{m(n-m)}    & \text{if }|\cS \cap \cS'| = m-1,\\
                0   & \text{otherwise}.
            \end{cases}
    \]
    
    This chain is irreducible, and it is aperiodic for all $n \geq 3$ and $m \in [n-1]$ because $J(n,m)$ contains 3-cycles. 
    Hence it has a unique stationary measure, which is the uniform measure $\mu_{\mathrm{uni}}$ on $\cV_{n,m}$. 
    In addition, the chain is reversible with respect to $\mu_{\mathrm{uni}}$ because
    \[
        \mu_{\mathrm{uni}}(\cS) \cdot P(\cS, \cS') = \mu_{\mathrm{uni}}(\cS') \cdot P(\cS', \cS), \qquad \text{for all adjacent }(\cS, \cS').
    \]

    Let $\langle f,g\rangle_{\mu_{\uni}} \coloneqq \sum_{\cS} f(\cS) \cdot g(\cS) \cdot\mu_{\uni}(\cS)$.
    Now that $P$ is a reversible transition matrix with stationary distribution $\mu_{\uni}$, the (symmetric) Dirichlet form associated with $(P, \mu_{\uni})$ is defined as follows: for any $f,g:\cV_{n,m}\to\RR$, 
    \[
        Q(f,g) \coloneqq \langle f, Lg\rangle_{\mu_{\uni}}
            = \frac{1}{2}\sum_{\cS,\cS'} \mu_{\uni}(\cS)\,P(\cS,\cS')\cdot \bigl[\,f(\cS)-f(\cS')\,\bigr]\,\bigl[\,g(\cS)-g(\cS')\,\bigr],
    \]
    where $L = I - P$ is the discrete-time generator. 
    Using the neighbor structure in $J(n,m)$, we obtain
    \begin{align*}
        Q(f,g)
            &= \frac{1}{2\,\binom{n}{m}} \sum_{\cS \in \cV_{n,m}} \sum_{j \in \cS} \sum_{k \in\cS^c}
                \frac{1}{m(n-m)}\, \bigl[\,f(\cS)-f(\cS^{(j\leftrightarrow k)})\, \bigr]\, \bigl[\,g(\cS)-g(\cS^{(j\leftrightarrow k)})\, \bigr]\\
            &= \frac{1}{2} \cdot \bbE_{\cS\sim\mu_{\uni}} \biggl[ \, \bbE_{J\sim \Unif(\cS),\,K\sim \Unif(\cS^c)}
                \Bigl[\,\bigl(\,f(\cS)-f(\cS^{(J\leftrightarrow K)} )\, \bigr)\, \bigl(\,g(\cS)-g(\cS^{(J\leftrightarrow K)})\,\bigr) \, \Bigm|\, \cS \, \Bigr] \, \biggr].
    \end{align*}

    In particular,
    \[
        Q(f,f)
        = \frac{1}{2} \bbE_{\cS\sim \mu_{\mathrm{uni}} }\, \bbE_{\cS'} \Big[ \big(f(\cS')-f(\cS)\big)^2 \Bigm|\ \cS\ \Big],
    \]
    where $\cS'$ is a one-step neighbor of $\cS$ drawn as in the chain, and under $\mu_{\uni}$, $(\cS',\cS)$ has the same distribution as $(\cS^{(J\leftrightarrow K)},\cS)$ with $J\sim\Unif(\cS)$ and $K\sim\Unif(\cS^c)$.
    The standard Poincar\'e inequality for reversible chains \citep[Remark 13.8]{levin2017markov} yields that for any $F$ with $\bbE_{\mu_{\uni}}[F]=0$,
    \begin{equation}\label{eqn:poincare}
        \Var_{\mu_{\uni}}(F) \leq \frac{1}{\lambda_{\mathrm{gap}}} Q(F,F),
    \end{equation}
    where $\lambda_{\mathrm{gap}}$ is the spectral gap, i.e., the difference between the largest eigenvalue (=1) and the second largest eigenvalue of $P$. 
    
    It remains to identify $\lambda_{\mathrm{gap}}$. 
    Recall that our Markov chain is built on the Johnson graph $J(n,m) = (\cV_{n,m}, \cE_{n,m})$, which is a regular graph of degree $\degree = m(n-m)$. 
    Note that the transition matrix $P = A / \degree$ where $A$ is the adjacency matrix of $J(n,m)$. 
    It is well known that the Johnson graph has $m+1$ distinct adjacency eigenvalues, which are given in the descending order by 
    \[
        \theta_j = E_1(j) = (m-j) (n-m-j) - j, 
        \qquad
        \text{with multiplicity } m_j = \binom{n}{j} - \binom{n}{j-1},
        \qquad j=0, 1, \dots, m,
    \]
    where the Eberlein polynomial $E_r(j) = \sum_{t=0}^r (-1)^t \binom{j}{t} \binom{m-j}{r-t} \binom{n-m-j}{r-t}$ \citep{delsarte1973algebraic}; see also \citet[Chapter 12.3.2]{brouwer2011spectra}. 
    In particular, the two largest eigenvalues of $A$ are
    \[
        \theta_0 = m (n-m),
        \qquad
        \theta_1 = m(n-m) - n,
    \]
    where $\theta_0$ has multiplicity $1$ and $\theta_1$ has multiplicity $n-1$. 
    Since $P = A / \degree$ with $\degree = m(n-m)$, the first two largest eigenvalues of $P$ are
    \[
        \lambda_1(P) = 1, 
        \qquad
        \lambda_2(P) = 1 - \frac{n}{m(n-m)}.
    \]
    Thus, $\lambda_{\mathrm{gap}} = \frac{n}{m(n-m)}$.
    
    Combining the spectral gap $\lambda_{\mathrm{gap}} = \frac{n}{m(n-m)}$ with \eqref{eqn:poincare} gives
    \[
        \Var(F) \leq \frac{m(n-m)}{n}\ Q(F,F)
        =\frac{m(n-m)}{2n}\ \bbE\Big[\,\big(F(\cS')-F(\cS)\big)^2\,\Big],
    \]
    where $(\cS,\cS')$ are as described above.
    This completes the proof.
\end{proof}

\begin{remark}
    In the proof of Theorem~\ref{thm:main-OLS} to follow in Appendix~\ref{sec:proof_main_theorem}, we apply Lemma~\ref{lem:ES-swap} to the degree–$d'$ arm–wise correction–error functional $F_{\arm,d'}(\cS_\arm)$, with $\arm\in\{0,1\}$ and $0\leq d'\leq d$. 
    We will (i) bound $\mathsf{s}_2(F_{\arm,d'})$ by an explicit, design–based row–swap analysis, and (ii) invoke \eqref{eq:ES-swap} to convert this into a variance (hence $L_2$ and $L_1$) control for the correction–estimation error at degree $d'$.
\end{remark}

\subsection{Completing the proof of Theorem \ref{thm:main-OLS}}\label{sec:proof_main_theorem_completing}

\begin{proof}[Proof of Theorem \ref{thm:main-OLS}]

    We organize this proof in four steps.

    \paragraph{Step 1. Master decomposition and Neumann convergence.}
    Recall from \eqref{eq:NPm-def-final} that 
    \[
        \htauols^{[d]} \coloneqq \hat\tau_{\mathrm{OLS}} + \sum_{d'=0}^{d} \left( \whsfR_1^{[d']}-\whsfR_0^{[d']} \right),
    \]
    and from \eqref{eq:master-decomp} that $\hat\tau_{\mathrm{OLS}}-\tau = \hresdim - \bigl(\sfR_1-\sfR_0\bigr)$, where
    \[
        \hresdim = \frac{1}{n_1}\bone_{n_1}^\top S_1 r^{(1)}-\frac{1}{n_0}\bone_{n_0}^\top S_0 r^{(0)}.
    \]
    Because we assumed Assumption~\ref{assumption:rates} holds with $\alpha_{n,p}(\delta) < 1$, on the ``good'' event $\Omega_\delta$, $\max_{\arm\in\{0,1\}}\ \|\Delta_{\arm}\| \leq \alpha_{n,p}(\delta) < 1$, and therefore, the Neumann series $(I_p - \Delta_{\arm})^{-1} = \sum_{d'=0}^\infty \Delta_{\arm}^{d'}$ converges, cf. \eqref{eqn:R-approx}, and 
    \[
        \sfR_{\arm} = \sum_{d' = 0}^{\infty} \sfR_{\arm}^{[d']}
        \qquad\text{where}\qquad
        \sfR_{\arm}^{[d']} = \xbara^\top  \, \Delta_{\arm}^{d'} \, u_{\arm}.
    \]
    Defining
    \[
        \sfT^{[d]}_{\arm} \coloneqq \sum_{d' \geq d+1} \sfR^{[d']}_{\arm},
    \]
    the Neumann tail truncated at degree $d$, we may write $\sfR_{\arm} = \sum_{d' = 0}^{d} \sfR_{\arm}^{[d']} + \sfT^{[d]}_{\arm}$.
    Then it follows that
    \begin{equation}\label{eq:masterC}
        \htauols^{[d]}-\tau
            =\hresdim 
                +\underbrace{\sum_{d'=0}^{d} \Big[(\widehat{\sfR}^{[d']}_1-\sfR^{[d']}_1)-(\widehat{\sfR}^{[d']}_0-\sfR^{[d']}_0)\Big]}_{\text{Error in correction terms}}
                \underbrace{-\Bigl(\sfT^{[d]}_1-\sfT^{[d]}_0\Bigr)}_{\substack{\eqqcolon \Delta_{n,p}^{[d]}\\\text{Neumann tail}}}.
    \end{equation}
    In the rest of this proof, we show: (i) a geometric bound for the Neumann tail $\Delta_{n,p}^{[d]}$ using Assumption \ref{assumption:rates}, and (ii) an $o_P(n^{-1/2})$ bound for the stochastic error in correction terms using Lemma~\ref{lem:ES-swap}.

    \paragraph{Step 2. Geometric bound on the Neumann tail.} 
    On the event $\Omega_\delta$, for each arm $\arm \in \{0,1\}$, $\sfT^{[d]}_\arm=\sum_{d'\ge d+1}\bar x_\arm^\top \, \Delta_\arm^{d'} \, u_\arm$, and hence,
    \begin{equation}\label{eq:tailC}
        |\sfT^{[d]}_\arm|
            \leq \|\bar x_\arm\|\frac{\|\Delta_\arm\|^{d+1}}{1-\|\Delta_\arm\|}\,\|u_\arm\|
            \leq \frac{\alpha_{n,p}(\delta)^{\,d+1}}{1-\alpha_{n,p}(\delta)} \cdot \beta_{n,p}(\delta) \cdot \gamma_{n,p}(\delta).
    \end{equation}
    Therefore, we obtain the tail envelope claimed in \eqref{eq:two-term-envelope} by triangle inequality:
    \[
        \Delta_{n,p}^{[d]} \leq |\sfT^{[d]}_1| + |\sfT^{[d]}_0| 
            \leq 2 \cdot \frac{\alpha_{n,p}(\delta)^{\,d+1}}{1-\alpha_{n,p}(\delta)} \cdot \beta_{n,p}(\delta) \cdot \gamma_{n,p}(\delta)
            = \varepsilon_{n,p}^{[d]}(\delta).
    \]

    \paragraph{Step 3: Controlling stochastic error in correction terms.} 
    We organize this step in two sub-steps. 
    In Step 3-a, we analyze one-swap sensitivity of the error in correction terms. 
    In Step 3-b, we establish a variance upper bound by applying Lemma \ref{lem:ES-swap}. 
    In Step 3-c, we apply Chebyshev's inequality to obtain $o_P(n^{-1/2})$ error bound and conclude Step 3. 
    Throughout Steps 3–4, we use the symbol ``$\lesssim$'' to hide absolute constants; specifically, we write $A\lesssim B$ to mean $A\le C\,B$ for an absolute constant $C>0$ not depending on $n,p,d'$ or the sample realization.

    \begin{itemize}
        \item 
        \emph{Step 3-a) One-swap increment bound.} 
        For each $\arm\in\{0,1\}$ and each $d'\in\{0,1,\dots,d\}$, define
        \[
            F_{\arm,d'} \coloneqq \widehat{\sfR}^{[d']}_{\arm}-\sfR^{[d']}_{\arm}.            
        \]
        Viewing $F_{\arm,d'} = F_{\arm,d'}(\Sa)$ as a function of a random subset $\Sa \subset [n]$, we bound the average one-swap increment $\mathsf{s}_2(F_{\arm, d'})^2 = \bbE\Bigl[ \bigl(\DeltaJK F(\cS)\bigr)^2 \Bigr]$ for $J\sim\Unif(\cS_\arm)$ and $K\sim\Unif(\Sa^c)$ (conditionally independent given $\Sa$), where $\Delta^{(J\leftrightarrow K)}F_{\arm,d'} = F_{\arm,d'}(\cS_\arm^{(J\leftrightarrow K)})-F_{\arm,d'}(\cS_\arm)$; see \eqref{eqn:avg_swap_sensitivity} in Definition \ref{defn:s2_functional} and \eqref{eqn:swap_operator}. 

        Recall
        \[
            \bar x_\arm = \frac{1}{n_\arm}\sum_{i\in\cS_\arm}x_i,\qquad
            \Sigma_\arm = \frac{1}{n_\arm}\sum_{i\in\cS_\arm}(x_i-\bar x_\arm)(x_i-\bar x_\arm)^\top,\qquad
            u_\arm=\frac{1}{n_\arm}\sum_{i\in\cS_\arm}(x_i-\bar x_\arm)\,r^{(\arm)}_i
        \]
        Let $\cS'_\arm=\cS_\arm^{(J\leftrightarrow K)}$ and write $\Delta^{(J\leftrightarrow K)} Z \coloneqq Z(\cS'_\arm)-Z(\cS_\arm)$. 
        A direct calculation yields the deterministic increments
        \begin{align}
            \Delta^{(J\leftrightarrow K)}\bar x_\arm
                &= \frac{x_K-x_J}{n_\arm},\label{eq:dxbar}\\
            \Delta^{(J\leftrightarrow K)}\Sigma_\arm
                &= \frac{x_Kx_K^\top-x_Jx_J^\top}{n_\arm}\ -\ \big(\bar x'_\arm\bar x_\arm^{\prime\top}-\bar x_\arm\bar x_\arm^\top\big), \nonumber\\
            \Delta^{(J\leftrightarrow K)}u_\arm
                &= \frac{(x_K-\bar x'_\arm)\,r^{(\arm)}_K-(x_J-\bar x_\arm)\,r^{(\arm)}_J}{n_\arm}\ -\ \frac{\Delta^{(J\leftrightarrow K)}\bar x_\arm}{n_\arm}\sum_{i\in \cS_\arm\cap\cS'_\arm} r^{(\arm)}_i,   \nonumber
        \end{align}
        where $\bar x'_\arm \coloneqq \bar x_\arm(\cS'_\arm)$. 
        Therefore,
        \begin{align}
            \opnorm{\Delta^{(J\leftrightarrow K)}\Sigma_\arm}
                &\leq \frac{\|x_K\|^2+\|x_J\|^2}{n_\arm}
                    + \big(\|\bar x_\arm\|+\|\bar x'_\arm\|\big)\,\big\|\Delta^{(J\leftrightarrow K)}\bar x_\arm\big\|
                    \nonumber\\
                &\leq \frac{\|x_K\|^2+\|x_J\|^2}{n_\arm}
                    + \frac{\|\bar x_\arm\|+\|\bar x'_\arm\|}{n_\arm}\,\|x_K-x_J\|,\label{eq:dSigma}\\
            \big\|\Delta^{(J\leftrightarrow K)}u_\arm\big\|
                &\leq \frac{\|x_K-\bar x'_\arm\|\,|r^{(\arm)}_K|+\|x_J-\bar x_\arm\|\,|r^{(\arm)}_J|}{n_\arm}\ +\ \frac{\|x_K-x_J\|}{n_\arm^2}\ \big\|r^{(\arm)}\big\|_1.\label{eq:du}
        \end{align}

        Next, we establish increment bounds for $\sfR^{[d']}_{\arm}$ and $\whsfR^{[d']}_{\arm}$ separately.
        
        \begin{itemize}
            \item 
            \textbf{Increment of the population Neumann term $\sfR^{[d']}_{\arm}$.}
            Recall 
            \[
                \sfR^{[d']}_{\arm}=\bar x_\arm^\top\,\Delta_\arm^{\,d'}\,u_\arm 
                \qquad\text{where}\qquad
                \Delta_\arm=\Sigma_\arm-I_p.
            \]
            By the product rule and submultiplicativity,
            \begin{align}
                \big|\Delta^{(J\leftrightarrow K)}\sfR^{[d']}_\arm\big|
                    &\leq \big\|\Delta_\arm^{\,d'}u_\arm\big\|\,\big\|\Delta^{(J\leftrightarrow K)}\bar x_\arm\big\|
                    \ +\ d'\,\big\|\Delta_\arm^{\,d'-1}\big\|\,\|\bar x_\arm\|\,\|u_\arm\|\,\opnorm{\Delta^{(J\leftrightarrow K)}\Delta_\arm } \nonumber\\
                    &\qquad +\ \big\|\Delta_\arm^{\,d'}\bar x_\arm\big\|\,\big\|\Delta^{(J\leftrightarrow K)}u_\arm\big\|.\label{eq:dR-pop}
            \end{align}
            With the normalization $X^\top X=nI_p$ and $\bone_n^\top X=0$, $\tfrac{1}{n}\sum_i\|x_i\|^2=p$ and thus, $\bbE \left[\|x_J-x_K\|^2\mid \cS_\arm\right]\le 4p$. 
            Using the identities in \eqref{eq:dxbar}–\eqref{eq:du}, we obtain the conditional averaged square increment bound
            \begin{align}
                \bbE \left[\big(\Delta^{(J\leftrightarrow K)}\sfR^{[d']}_\arm\big)^2 \,\Bigm|\, \cS_\arm\right]
                    & \lesssim\
                    \frac{p}{n_\arm^2}\Big(\|\Delta_\arm\|^{2d'}\|u_\arm\|^2
                    \ +\ d'^2\,\|\Delta_\arm\|^{2d'-2}\|\bar x_\arm\|^2\|u_\arm\|^2\Big) \nonumber\\
                    &\qquad +\ \frac{p}{n_\arm^3}\,\|\Delta_\arm\|^{2d'}\|\bar x_\arm\|^2\,\|r^{(\arm)}\|_2^2,
                        \label{eq:pop-one-swap-L2}
            \end{align}
            where we also used $\|r^{(\arm)}\|_1\le \sqrt{n}\,\|r^{(\arm)}\|_2$.

            \item 
            \textbf{Increment of the sample‑analog term $\widehat{\sfR}^{[d']}_\arm$.}
           Recall
            \[
                \widehat{\sfR}^{[d']}_\arm=\frac{1}{n_\arm}\sum_{i\in\cS_\arm}\xi_i^{[d']}(n_\arm;X)\cdot \hat \res^{(\arm)}_i,
            \]
            where the weights $\xi^{[d']}_i(n_\arm;X)$ depends only on $(n_\arm,n,X)$, not on $\Sa$.
            Let $\xi\in\RR^n$ collect these weights.
            Observe that $\hat r^{(\arm)}_{\cS_\arm}=M_\arm(\cS_\arm)\,S_\arm r^{(\arm)}$ with the arm‑wise residual‑maker $M_\arm = I_{n_{\arm}} - Z_{\arm}( Z_{\arm}^{\top} Z_{\arm})^{-1} Z_{\arm}^{\top}$ where $Z_{\arm} = \begin{bmatrix} \bone_{n_{\arm}} & X_{\arm}\end{bmatrix}$; $M_{\arm}$ is an orthogonal projection, and hence $\opnorm{ M_\arm } \le1$. 
            A swap changes both the index set and the projection:
            \begin{align}
            \Delta^{(J\leftrightarrow K)}\widehat{\sfR}^{[d']}_\arm
            &= \frac{1}{n_\arm}\Big(\xi_K\,\hat r^{(\arm)}_{K}(\cS'_\arm)-\xi_J\,\hat r^{(\arm)}_{J}(\cS_\arm)\Big)
            \ +\ \frac{1}{n_\arm}\sum_{i\in\cS_\arm\cap\cS'_\arm}\xi_i\Big(\hat r^{(\arm)}_i(\cS'_\arm)-\hat r^{(\arm)}_i(\cS_\arm)\Big).\nonumber
            \end{align}
            By Cauchy–Schwarz inequality and $\|\hat r^{(\arm)}(\cS)\|_2\le \|r^{(\arm)}\|_2$,
            \begin{equation}\label{eq:samp-one-swap-L2}
            \big|\Delta^{(J\leftrightarrow K)}\widehat{\sfR}^{[d']}_\arm\big|
            \ \leq\ \frac{2}{n_\arm}\,\|\xi\|_2\,\|r^{(\arm)}\|_2,
            \qquad\Rightarrow\qquad
            \bbE \Bigl[\big(\Delta^{(J\leftrightarrow K)}\widehat{\sfR}^{[d']}_\arm\big)^2 \,\Bigm|\, \cS_\arm\Bigr]
             \leq \frac{4}{n_\arm^2}\,\|\xi\|_2^2\,\|r^{(\arm)}\|_2^2.
            \end{equation}
                        
        \end{itemize}

        Since $F_{\arm,d'}=(\widehat{\sfR}^{[d']}_\arm-\sfR^{[d']}_\arm)$, combining \eqref{eq:pop-one-swap-L2}–\eqref{eq:samp-one-swap-L2} together with the inequality $(a+b)^2\le 2a^2+2b^2$, and taking total expectation yields
        \begin{align}
            \bbE \left[\big(\Delta^{(J\leftrightarrow K)}F_{\arm,d'}\big)^2 \right]
                &\le 2\,\bbE \left[\big(\Delta^{(J\leftrightarrow K)}\widehat{\sfR}^{[d']}_\arm\big)^2 \right]
                 +2\,\bbE \left[\big(\Delta^{(J\leftrightarrow K)}\sfR^{[d']}_\arm\big)^2 \right]\nonumber\\
                &\lesssim 
                    \frac{p}{n_\arm^2} \Big(\|\Delta_\arm\|^{2d'}\|u_\arm\|^2 + d'^2\,\|\Delta_\arm\|^{2d'-2}\|\bar x_\arm\|^2\|u_\arm\|^2\Big)
                        + \frac{p}{n_\arm^3}\,\|\Delta_\arm\|^{2d'}\|\bar x_\arm\|^2\,\|r^{(\arm)}\|_2^2   \nonumber\\
                &\qquad + \frac{1}{n_\arm^2} \|\xi\|_2^2\,\|r^{(\arm)}\|_2^2.
                \label{eq:oneswap-L2}
        \end{align}
        On the high-probability event $\Omega_\delta$ of Assumption~\ref{assumption:rates}, we can replace $\|\Delta_\arm\|,\|u_\arm\|,\|\bar x_\arm\|$ by their deterministic envelopes $\alpha_{n,p}(\delta),\beta_{n,p}(\delta),\gamma_{n,p}(\delta)$.
        We retain $\|\xi\|_2$ explicitly for now (it will be handled in Step~3-b).

        \item 
        \emph{Step 3-b) Variance bound via Poincar\'{e} inequality.} 
        Apply Lemma~\ref{lem:ES-swap} to the random subset $\cS_\arm$ (of size $n_\arm$) with $F(\cS_{\arm})\equiv F_{\arm,d'}$:
        \begin{equation}\label{eq:ES-apply-final}
            \Var(F_{\arm,d'}) \leq \frac{n_{\arm}(n-n_{\arm})}{2n} \bbE\Big[\big(\Delta^{(J\leftrightarrow K)}F_{\arm,d'}\big)^2\Big]
                \lesssim\ n\ \bbE\Big[\big(\Delta^{(J\leftrightarrow K)}F_{\arm,d'}\big)^2\Big],
        \end{equation}
        where we use $n_\arm\asymp n$ in the last step (complete randomization with $\rho_n\to\rho\in(0,1)$).
        
        Plugging \eqref{eq:oneswap-L2} into \eqref{eq:ES-apply-final}, on the event $\Omega_{\delta}$, 
        \begin{align*}
            \Var(F_{\arm,d'})
            &\lesssim
            \underbrace{\frac{1}{n}\Big(\alpha_{n,p}(\delta)^{2d'}\,\beta_{n,p}(\delta)^2
            + d'^2\,\alpha_{n,p}(\delta)^{2d'-2}\,\gamma_{n,p}(\delta)^2\,\beta_{n,p}(\delta)^2\Big)}_{=:V_{\rm pop}}
            + \underbrace{\frac{p}{n^2}\ \alpha_{n,p}(\delta)^{2d'}\,\gamma_{n,p}(\delta)^2\ \frac{\|r^{(\arm)}\|_2^2}{n}}_{=:V_{\rm mix}}\\
            &\qquad + \underbrace{\frac{\|\xi\|_2^2}{n}\ \frac{\|r^{(\arm)}\|_2^2}{n}}_{=:V_{\rm samp}}.
        \end{align*}
        \begin{enumerate}
            \item 
            The first term $V_{\rm pop}$ stems from population–side increments and is $O(1/n)$, given Assumption \ref{assumption:rates}.
            \item 
            For $V_{\rm mix}$, we use the premise $\sup_n \|r^{(\arm)}\|_2^2/n <\infty$ to conclude $V_{\rm mix}=O \big(p/n^2\big)$.
            \item 
            For $V_{\rm samp}$, we can use the crude bound $\|\xi\|_2^2\le \sum_{i=1}^n \xi_i^{[d']}(n_\arm;X)^2 < \infty$ that depends only on $(n,p,X,d')$. 
            In particular, for $d'=0$, Example~\ref{example:exp_k.0} gives 
            $\xi^{[0]}_i(n_\arm;X)$ proportional to $h_i-\frac{p}{n}$ with $h_i=\|x_i\|^2/n$, hence
            $\|\xi\|_2^2 \lesssim \sum_{i=1}^n (h_i-\tfrac{p}{n})^2 \le \sum_{i=1}^n h_i^2 = \tr(H_X^2)\le \tr(H_X)=p$.
            For fixed $d'\ge 1$ the (finite) design weight map $X\mapsto \xi^{[d']}(n_\arm;X)$ is linear in the column‑space statistics of $X$ (Section~\ref{sec:design_expectation}); thus $\|\xi\|_2^2\lesssim c_{d'}\,p$ with a constant $c_{d'}$ depending only on $d'$ and $\rho$.\footnote{This follows by expanding $\xi^{[d']}$ in the orthonormal basis induced by $H_X$ and bounding its coefficients uniformly in $n$ and $p$ for fixed $d'$. The coarse envelope $\|\xi\|_2^2\lesssim p$ suffices for our variance purpose.}
            Consequently,
            \[
                V_{\rm samp}\ \lesssim\ \frac{p}{n}\cdot \frac{\|r^{(\arm)}\|_2^2}{n}\ \to\ \frac{p}{n}\cdot v_\arm.
            \]
        \end{enumerate}

        Putting the three pieces together, for each fixed $d'$
        \begin{equation}\label{eq:VarE-final-b}
            \Var(F_{\arm,d'})
                \lesssim
                \frac{1}{n}\Big(\alpha_{n,p}(\delta)^{2d'}\,\beta_{n,p}(\delta)^2
                + d'^2\,\alpha_{n,p}(\delta)^{2d'-2}\,\gamma_{n,p}(\delta)^2\,\beta_{n,p}(\delta)^2\Big)
                + \frac{p}{n}\cdot \frac{\|r^{(\arm)}\|_2^2}{n}\ +\ \frac{p}{n^2}.
        \end{equation}
        Since $\|r^{(\arm)}\|_2^2/n\to v_\arm$, the last two terms are $O(p/n)$ and $O(p/n^2)$, respectively. In particular, under the mild dimension growth $p=o(n)$, we have $\Var(F_{\arm,d'})=o(1/n)$.

        \item 
        \emph{Step 3-c) Conclusion via Chebyshev inequality.} 
        By Chebyshev’s inequality, 
        \[
            F_{\arm,d'}=o_P(n^{-1/2})\qquad\text{for each fixed $d'$ when }p=o(n).
        \]
        Summing over a fixed number of degrees $d' = 0, 1, \dots, d$ and both arms $\arm \in \{0,1\}$,
        \[
            \sum_{d'=0}^{d}\Big[(\widehat{\sfR}^{[d']}_1-\sfR^{[d']}_1)-(\widehat{\sfR}^{[d']}_0-\sfR^{[d']}_0)\Big]
             = o_P(n^{-1/2}).
        \]
    \end{itemize}

    \paragraph{Step 4: Conclusion of the proof.}
    Recall the master decomposition \eqref{eq:masterC}:
    \[
        \htauols^{[d]}-\tau
            =\hresdim 
            +\sum_{d'=0}^{d}\Big[(\widehat{\sfR}^{[d']}_1-\sfR^{[d']}_1)-(\widehat{\sfR}^{[d']}_0-\sfR^{[d']}_0)\Big]
            -\big(\sfT^{[d]}_1-\sfT^{[d]}_0\big).
    \]
    Step~2 gives the geometric envelope for the Neumann tail:
    \[
        \big|\sfT^{[d]}_1-\sfT^{[d]}_0\big|
            \leq 2\,\frac{\alpha_{n,p}(\delta)^{\,d+1}}{1-\alpha_{n,p}(\delta)}\,\beta_{n,p}(\delta)\,\gamma_{n,p}(\delta)
            = \varepsilon_{n,p}^{[d]}(\delta),
    \]
    on the event $\Omega_\delta$ with probability at least $1-\delta$.
    Step~3 shows that the stochastic error in the correction terms is $o_P(n^{-1/2})$ provided $p=o(n)$.
    Therefore,
    \[
        \htauols^{[d]}-\tau = \hresdim + \Delta_{n,p}^{[d]} + o_P(n^{-1/2}),
    \]
    with $\Delta_{n,p}^{[d]}$ satisfying the claimed high‑probability tail bound. This proves Theorem~\ref{thm:main-OLS}.
\end{proof}

%% file: contents/Appx05_additional_experiments.tex
\section{Additional experiments}\label{sec:additional_experiments}

Here we present additional experiments deferred from Section~\ref{sec:experiments}. 
More extensive studies are left for future work.

\paragraph{Conventions.}
We continue to set $n=500$, $n_1=150$, and in each outer repetition we draw $N=2000$ treatment assignments; we repeat the entire procedure $R=50$ times and report the median across repetitions with the $10\%$-$90\%$ region shaded. 
Performance metrics follow Section~\ref{sec:experiments}: (i) \emph{normalized absolute bias} $\big|\frac{1}{N}\sum_{k=1}^N \hat\tau_k-\tau\big|\cdot \frac{\sqrt{n}}{\sigma_n}$ and (ii) \emph{normalized empirical variance} $\Var(\hat\tau_1,\dots,\hat\tau_N)\cdot \frac{n}{\sigma_n^2}$, where $\sigma_n^2 = n_1^{-1}\|r^{(1)}\|_2^2 + n_0^{-1}\|r^{(0)}\|_2^2 - n^{-1}\|r^{(1)}-r^{(0)}\|_2^2$ with $r^{(a)}$ the \emph{population} OLS residuals. 
We also report \emph{normalized RMSE} (nRMSE) defined by 
\[
    \mathrm{nRMSE} \coloneqq \sqrt{\big(\text{normalized bias}\big)^2 + \big(\text{normalized variance}\big)}\,,
\]
which combines the two normalized components on a common scale.

\paragraph{Residual structures: Typical vs worst}
We report results for \emph{typical} i.i.d.\ residuals per \ref{item:typical}
Under i.i.d.\ residuals, $\hat\tau_{\mathrm{OLS}}$ already has small normalized bias and Neumann corrections yield modest additional gains (Figure~\ref{figure:experiment.typical}). 
Under worst‑case residuals, corrections have larger effects: even degrees tend to reduce normalized bias (with higher variance), whereas odd degrees often tilt the trade‑off in the opposite direction (cf.\ Figure~\ref{figure:experiment.1}).

\begin{figure}[t]
    \centering
    \includegraphics[width=0.32\linewidth]{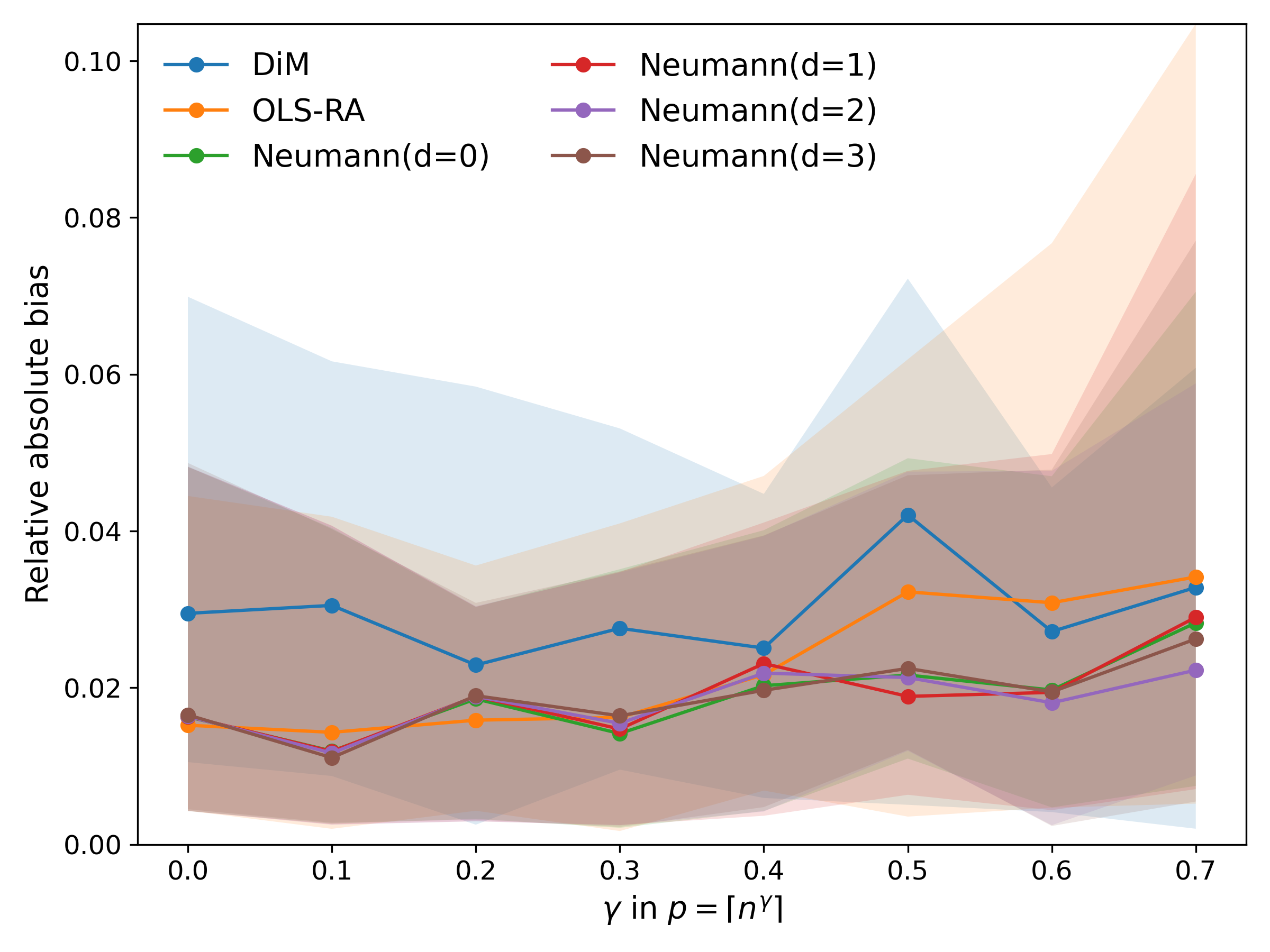}\hfill
    \includegraphics[width=0.32\linewidth]{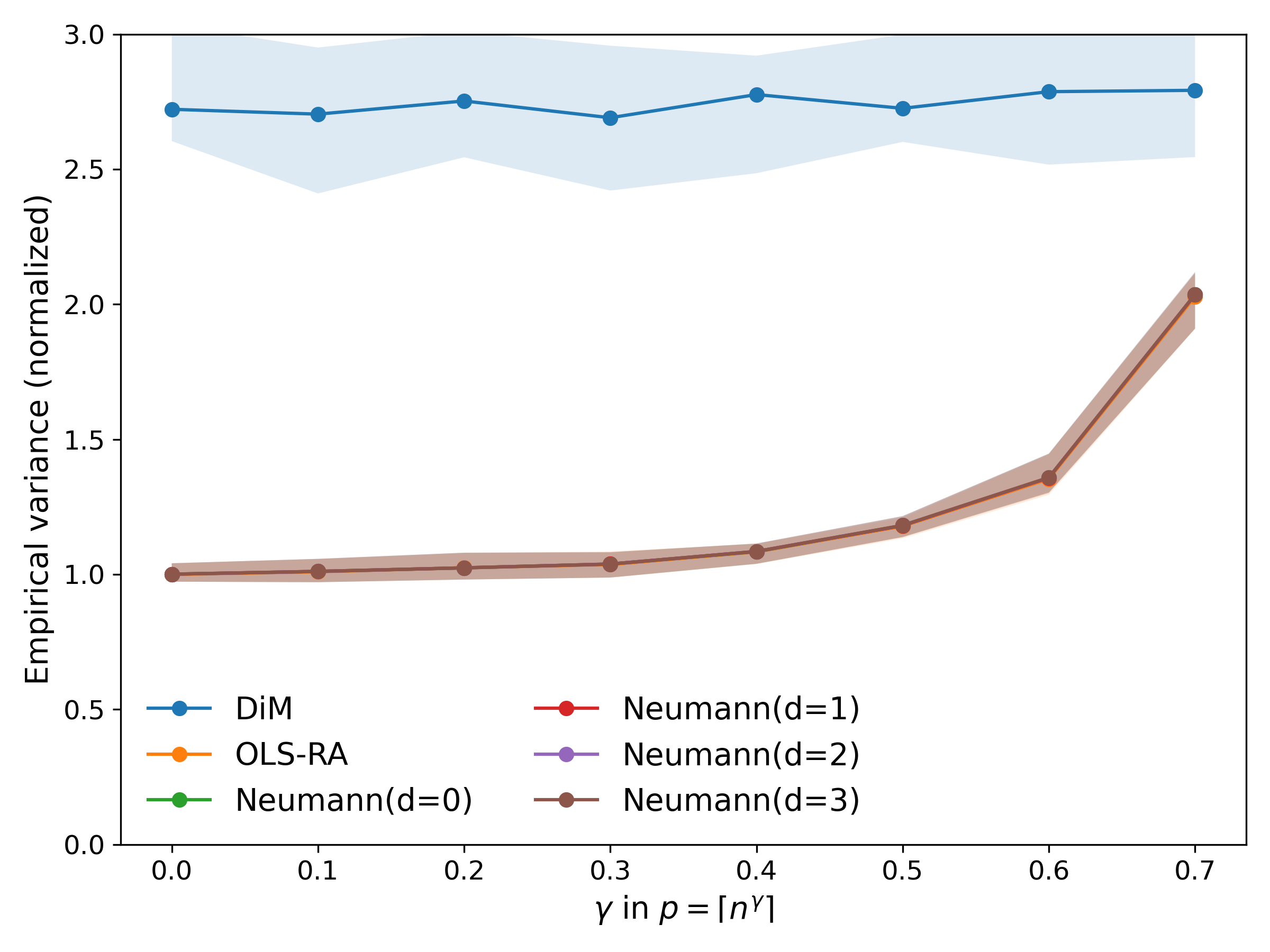}\hfill
    \includegraphics[width=0.32\linewidth]{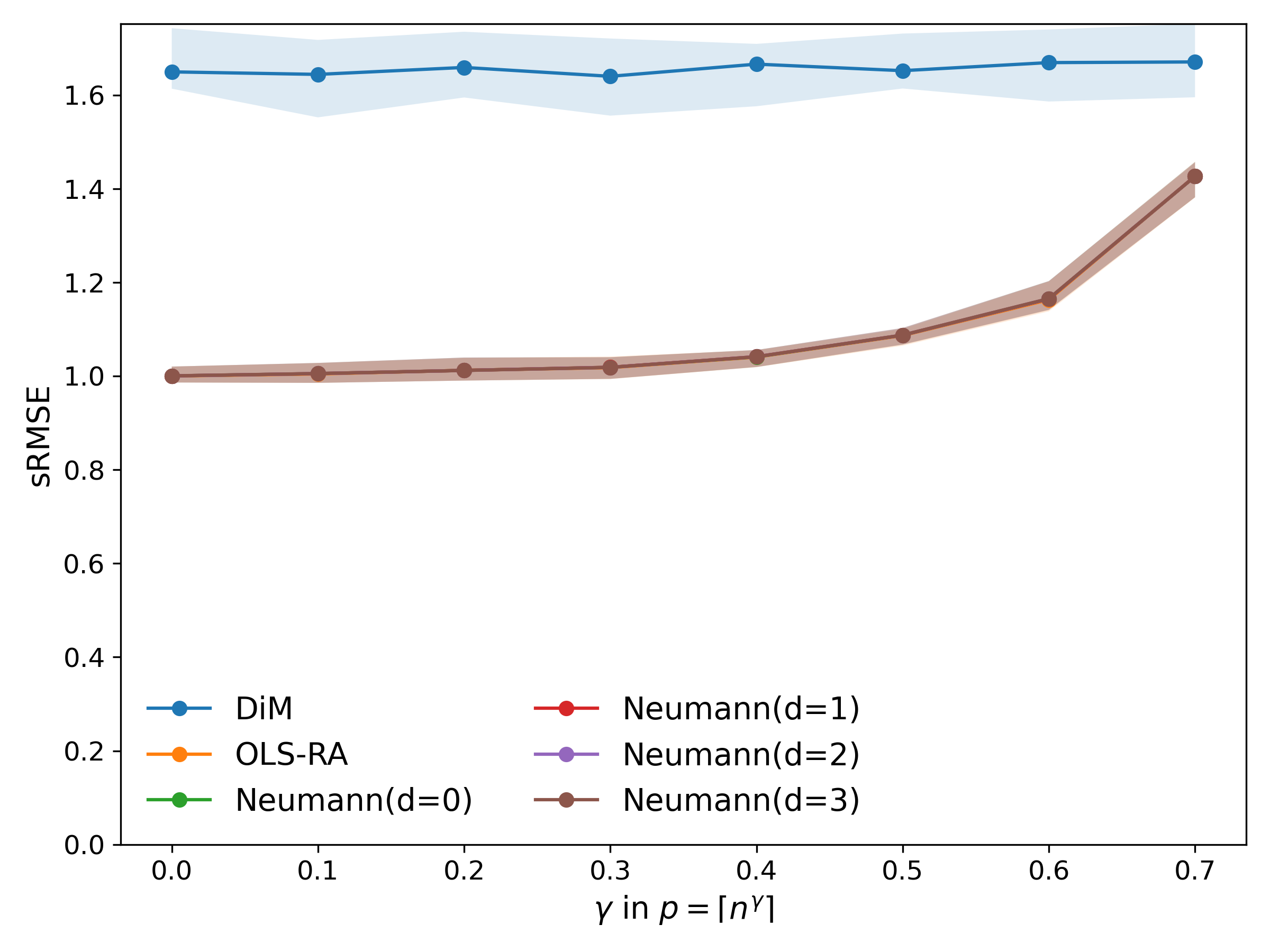}
    \caption{
        Comparison of $\hat\tau_{\mathrm{OLS}}$ versus $\hat\tau_{\mathrm{OLS}}^{[d]}$ for $d \in \{0,1,2,3\}$  under the "typical i.i.d. residuals" (cf. Section~\ref{sec:experiments}) at $n = 500$, $n_1 = 150$, $N = 250$, $R=25$ (median with $10\%$-$90\%$ region shaded).
        \textbf{(Left):} Normalized absolute bias. 
        \textbf{Middle:} Normalized empirical variance.
        \textbf{(Right):}  Normalized RMSE.
        In this setting, corrections provide only modest improvements because the baseline bias is already small.
        } 
    \label{figure:experiment.typical}
\end{figure}